\documentclass[10pt, reqno]{amsart}

\usepackage{latexsym}
\usepackage{amsfonts, amsmath, amsthm, amscd}
\usepackage{amssymb}

\usepackage{mathrsfs}
\usepackage{amsbsy}

\usepackage[linktocpage]{hyperref}
\hypersetup{colorlinks=false}

\DeclareMathAlphabet{\mathdutchcal}{U}{dutchcal}{m}{n} 
\SetMathAlphabet{\mathdutchcal}{bold}{U}{dutchcal}{b}{n}
\DeclareMathAlphabet{\mathdutchbcal}{U}{dutchcal}{b}{n}

\usepackage{enumerate}

\usepackage{appendix}
\usepackage{color}
\usepackage[usenames,dvipsnames]{xcolor}

 \usepackage{wrapfig}
\usepackage{graphicx, psfrag}
\usepackage{pstool}
 \usepackage{caption}
\usepackage{subcaption}

\usepackage{tikz}
\usepackage{tikz-cd}

\numberwithin{equation}{section}

\newtheorem{theorem}{Theorem}[section]
\newtheorem{prop}[theorem]{Proposition}

\newtheorem{lemma}[theorem]{Lemma}
\newtheorem{cor}[theorem]{Corollary}

\newtheorem{defn}[theorem]{Definition}

\newtheorem{remark}[theorem]{Remark}
\newenvironment{rem}{\begin{remark}\rm}{\end{remark}}
\newtheorem{example}[theorem]{Example}

\newtheorem*{theorem*}{Theorem}
\newtheorem*{splittingtheorem*}{Splitting Theorem}
 
\def\be{\begin{equation}}
\def\ee{\end{equation}}
\def\bear{\begin{eqnarray}}
\def\eear{\end{eqnarray}}
\def\best{\begin{eqnarray*}}
\def\eest{\end{eqnarray*}}

 \def\non{\noindent}
 \def\bd{\partial}
 \def\ra{\rightarrow}
\def\lra{\longrightarrow}
\def\rg{\rangle}
\def\lg{\langle}
\def\r#1{\right#1}
\def\l#1{\left#1}
\def\ti{\times}
\def\del{\overline \partial}

\def\ma#1{\mathop {#1} \limits}

\def\al{\alpha}

\def\de{\delta}
\def\ep{\varepsilon}
\def\la{\lambda}
\def\si{\sigma}
\def\Si{\Sigma}
\def\De{\Delta}
\def\La{\Lambda}

\def\ra{\rightarrow}

\def\Z{{ \mathbb Z}}
\def\R{{ \mathbb  R}}
\def\P{{\mathbb  P}}
\def\Q{{ \mathbb Q}}
\def\cx{{ \mathbb C}}

\def\cal{\mathcal}

\def\C{{\mathcal C}}
\def\D{{\mathcal D}}
\def\E{{\mathcal E}}
\def\F{{\mathcal F}}

\def\L{{\mathcal L}}
\def\M{{\mathcal M}}
\def\oM{\overline{\mathcal M}}
\def\N{\mathcal{N}}
\def\O{\mathcal{O}}

\def\U{{\mathcal U}}
\def\T{{\mathcal T}}

\def\mo{\mathfrak{o}} 

\def\wt#1{\widetilde{#1}}
\def\wh#1{\widehat{#1}}
\def\ov#1{\overline{#1}}

\def\cok{\mathrm{coker\,}} 
\def\ind{\mathrm{ind\,}}
\def\Ind{\mathrm{Ind\,}}

\def\Aut{\mathrm{Aut}}

\def\Hom{\mathrm{Hom}}
\def\Fred{\mathrm{Fred}}

\definecolor{darkgreen}{rgb}{0, 0.5, 0.1} 

\definecolor{aqua}{cmyk}{0,1,0, .3}
\definecolor{teal}{rgb}{0.2, .8, .7}

\def\ev{\mathrm{ev}}

\newcommand\cHH{\check{\mathrm{H}}}  

 \def\JV{\mathcal {JV}}

\def\vir{{\rm vir}}

 \def\cP{\mathcal{P}}
 
 \def\fM{{\mathdutchcal{M}}}

\def\fF{{\mathdutchcal{F}}}

\def\fD{{\mathcal{D}}}
 \def\!{|\hspace{-1.5pt}|\hspace{-1.5pt}|}

\let\oldtocsection=\tocsection
\let\oldtocsubsection=\tocsubsection
\renewcommand{\tocsection}[2]{\hspace{0em}\oldtocsection{#1}{#2}}
\renewcommand{\tocsubsection}[2]{\hspace{1em}\oldtocsubsection{#1}{#2}}

\title{Splitting formulas for the local real Gromov-Witten invariants}

\author{Penka Georgieva}
\address{Sorbonne Universit\'e, Universit\'e de Paris, CNRS, Institut de Math\'ematiques de Jussieu - Paris Rive Gauche, F-75005 Paris, France}
\email{penka.georgieva@imj-prg.fr}
\author{Eleny-Nicoleta Ionel}
\address{Department of  Mathematics,  Stanford University}
\email{ionel@math.stanford.edu}

\begin{document}

\begin{abstract} Motivated by the real version of the Gopakumar-Vafa conjecture for 3-folds, the authors introduced in \cite{GI} the notion of 
local real Gromov-Witten invariants associated to local 3-folds over Real curves. This article is devoted to the proof of a splitting formula for these invariants under target degenerations. It is used in \cite{GI} to show that the invariants give rise to a 2-dimensional Klein TQFT and to prove the local version of the real Gopakumar-Vafa conjecture. 
\end{abstract}

\maketitle 

\vspace{-.2in}
\tableofcontents

\setcounter{section}{-1}

\vspace{-.4in}

\section{Introduction} 
\bigskip 

A central problem in Gromov-Witten theory is understanding the structure and properties of the Gromov-Witten invariants. Motivated by the real version of the Gopakumar-Vafa conjecture and the work of Bryan and Pandharipande \cite{bp1}, the authors introduced and studied in \cite{GI} the notion of local real Gromov-Witten invariants. In this article we prove the splitting formula for these invariants, as outlined in \cite[\S4]{GI} and used to establish the structural results of \cite{GI}. 

\medskip

A symmetric (or Real\footnote{We use Real with capital R for spaces or bundles with anti-holomorphic involutions, following Atiyah.}) Riemann surface is a (possibly disconnected and marked) Riemann surface $\Si$ together with an
anti-holomorphic involution $c:\Si\ra \Si$ (also referred to as a real structure). Throughout this paper we restrict attention to the case in which none of the marked points are real, and we denote by  $V$ the collection of marked points of $\Si$. Consider  the real relative moduli space 
$$
\ov\M^{\R}_{d,\chi,\vec\mu}(\Si, V)
$$
of degree $d$ real maps 
$f:C\ra \Si$ from (possibly disconnected) domains of Euler characteristic $\chi$ and having ramification profile $\vec\mu$ over the divisor $V$, reviewed in \S\ref{S.rel.moduli}. Here $\vec\mu$ is a collection of partitions of $d$, one for each pair of conjugate points in $V$. These  moduli spaces are  orientable, but not {\em a priori} canonically oriented; their orientation depends on a choice of (twisted) orientation data $\mo$ on $\Si$, cf.  \S\ref{S.orient}. 
\smallskip

When $L\ra \Si$ is a holomorphic line bundle,  
the \textsf{local RGW invariants} are defined as the pairing 
\bear\label{RGW.chern.0}
RGW^{c,\mathfrak o}_{d,\chi}(\Si,L)_{\vec\mu}=\frac{1}{|\Aut(\vec\mu)|}\int_{[\ov\M^{\R}_{d,\chi,\vec\mu}(\Si, V)]^{\vir, \mo}}c_{b/2}(-\ind \del_L)
\eear
between the appropriate Chern class of the index bundle $\mathrm{ Ind}\; \del_L$ (regarded as an element in the usual complex $K$-theory) and the virtual fundamental class of the real relative moduli space. 
Here $b$ is the (virtual) dimension\footnote{Unless otherwise noted, all dimensions appearing in this paper are real dimensions rather than complex dimensions.} of the moduli space and the presence of the factor $|\Aut(\vec\mu)|$ is explained in Remark~\ref{R.moduli.ordered}. In particular, \eqref{RGW.chern.0} are the coefficients of the generating series \cite[(2.24)]{GI}.  
\medskip

The pairing \eqref{RGW.chern.0} is invariant under smooth deformations of the target $\Si$ (as a symmetric marked curve) together with $L$ and 
$\mo$; in particular it is independent of the holomorphic structure on $L$. By considering a family $\cup_s \Si_s$ of  smooth symmetric marked curves degenerating to a symmetric nodal curve $\Si_0$ (with a conjugate pair of nodes), the splitting formula proved in this paper relates the local RGW invariants of a smooth fiber $\Si_s$ to those of the normalization of the singular fiber $\Si_0$. 
\medskip

Throughout this paper, $\Si_0$ denotes a symmetric nodal curve with a single pair of conjugate nodes and $r$ pairs of conjugate marked points; in particular, we assume that $\Si_0$ has no real special points. We always denote by 
\bear\label{0.Fam}
\F_{/\De}=\ma\cup_{s\in \De} \Si_s
\eear
a (flat) family of deformations of $\Si_0$ parametrized by the unit disk $\De$ as described in \S\ref{S.fam.curves}. The fibers over $s\ne 0$ are smooth (marked, symmetric) curves  
$\Si_s$. We also denote by 
\best
\wt \Si\ra \Si_0
\eest 
the normalization (or resolution) of the singular fiber $\Si_0$ (as a symmetric marked curve).  We denote by $c_s$, $\wt c$ and $V_s$, $\wt V$ the corresponding real structures and markings.   
\smallskip

If $L\ra \F$ is a complex line bundle over the family, we denote by $L_s$ and $\wt L$ its pullback to $\Si_s$ and respectively 
$\wt \Si$. There is also a notion of twisted orientation data $\mo_\F$ on the family $\F$, which pulls-back to give twisted orientation data 
$\mo_s$ on $\Si_s$ and $\wt\mo$ on $\wt \Si$, cf. \S\ref{S.tw.exist}. 

With these preliminaries, the main result of this paper is the proof of the following theorem.

\begin{theorem}[RGW Splitting Theorem]\label{T.gluing} Let  $\F=\cup_s \Si_s$ be a family of deformations of $\Si_0$ and $\wt \Si$ be its normalization as described above. Fix $L\ra \F$ a complex line bundle and $\mo_\F$ a twisted orientation on $\F$.  Then the local RGW invariants \eqref{RGW.chern.0} associated to the smoothing $\Si_s$ are related to those associated to the normalization $\wt \Si$ as follows: 
	\bear\label{RGW.splits.coef}
	RGW_{d, \chi}^{c_s, \mo_s}(\Si_s, L_s)_{\vec \mu} = \sum_{\la\vdash d} \zeta(\la) RGW_{d, \chi+4\ell(\la)}^{\wt c,\wt \mo}(\wt \Si, \wt L)_{\vec\mu, \la, \la} 
	\eear
	for  all $s\neq 0$ and all $d$, $\chi$, and $\vec \mu$. 
Here $\la$ is a partition of $d$, $\ell(\la)$ is its length and the coefficient $\zeta(\la)$ is given by \eqref{mult.la}. 
\end{theorem} 
Theorem \ref{T.gluing}  is an extension to the real setting of the splitting formula of Bryan-Pandharipande \cite[Theorem~3.2]{bp1} proved in \cite[Appendix A]{bp-TQFT}. A direct consequence of Theorem~\ref{T.gluing} is \cite[Theorem~4.1]{GI} cf. Corollary~\ref{C.rgw} and the discussion after it. This is used in \cite{GI} to  show that the local RGW invariants give rise to a 2-dimensional Klein TQFT and to prove the local version of the real Gopakumar-Vafa conjecture.
\medskip

The proof of Theorem~\ref{T.gluing} is a consequence of the splitting properties of the total Chern class of the index bundle established by Bryan and Pandharipande, combined with a splitting formula for the virtual fundamental class of the real relative moduli space. The proof of the latter  -- Theorem~\ref{T.splitVFC} below -- occupies the majority of this paper. The basic idea is an adaptation to the real setting of the classical proof of the splitting theorem \cite{lr, li, ip-sum}.

\medskip 
We consider a family of real moduli spaces 
\bear\label{0.fam}
\ov\M_{d,\chi,\vec\mu}(\F_{/\De})\;\ma=\; \ma\cup_{s\in \De} \ov\M^{\R}_{d,\chi,\vec\mu} (\Si_s, V_s) 
\eear
associated to the family $\cal F$ of targets, cf. \eqref{mod.fam.la}. For every partition $\la$ of $d$, we also consider the real relative moduli space
\best
\ov\M^\R_{d,\chi,\vec\mu,\la,\la} (\wt \Si, \wt V)
\eest
associated to the normalization $\wt \Si$ of $\Si_0$; note that $\wt \Si$ has two additional pairs of conjugate marked points (corresponding to the pair of conjugate nodes of $\Si_0$) and we restrict to the case when the ramification profile is $\la$ over all these additional points. There is a natural map 
\bear\label{0.phi}
\Phi: \ma\bigsqcup_{\la\vdash d} \ov\M^\R_{d, \chi+4\ell(\la), \vec\mu, \la,\la}(\wt \Si, \wt V) \ra \ov\M^\R_{d, \chi, \vec\mu} (\Si_0, V_0)
\eear
which attaches pairs of marked points of both the domain and target to produce nodes, cf. \S\ref{S.nod.target}. 

\begin{theorem}[VFC Splitting Theorem] \label{T.splitVFC} 
With the notation above, for every $d$, $\chi$, and $\vec \mu$ and all  $s\in \De\setminus 0$ the equality 	
\bear\label{split.VFC}
	[\ov \M^{\R}_{d, \chi,\vec \mu }(\Si_s, V_s)]^{\vir, \mo_s}
	=\sum_{\la\vdash d} \frac{\zeta(\la)}  {|\Aut (\la)|^2} 
	\Phi_*[\ov \M^{\R}_{d, \chi+4\ell(\la),\vec \mu, \la,\la }(\wt\Si, \wt V)]^{\vir, \wt\mo}
	\eear
holds in the rational Cech homology of the family \eqref{0.fam} of moduli spaces. 
\end{theorem}

We now outline the key steps involved in the proof of the splitting formula \eqref{split.VFC}. 
In this paper we use Ruan-Tian perturbations adapted to our setting in combination with the thin compactification method of \cite{ip-thin} as summarized in \S\ref{S.thin.cpt}. This approach allows us to use standard arguments after adapting them to the real setting.

\medskip
The notion of Ruan-Tian perturbations $\nu$ extends to the family $\cal F=\cup_s \Si_s$ of targets in a way that is compatible with the real structures and the divisors, cf. \S\ref{S.RT.perturb}. In particular, every RT perturbation $\nu$ on the family is real and pulls back to a RT perturbation on $\Si_s$ and $\wt \Si$, compatible with their divisors. Denote by $\cP$ the space of such RT perturbations $\nu$ on the family $\cal F$.

As $\nu$ varies over the parameter space $\cP$, we get various families of moduli spaces over $\cP$, cf. \S\ref{S.VFC.rel}. Denote by $\ov \fM(\Si_s)$ and $\ov \fM(\wt \Si)$ the family (over $\cP$) of \emph{real relative} moduli spaces associated to the fiber $\Si_s$ and respectively the normalization $\wt \Si$ of the nodal fiber $\Si_0$. Fix a segment $I=[0, s_0]\subset\De$, where $s_0\ne 0$. As in the usual proof of the splitting formula, we consider a family of moduli spaces 
\best
\begin{tikzcd}
\ov \fM(\F_{/I})= \ma \cup_{s\in I}\ov \fM(\Si_s). 
\end{tikzcd}
\eest
The attaching map \eqref{0.phi} extends to a proper map 
\best 
\begin{tikzcd}
\ov \fM(\wt \Si)\ar[r, "\Phi"]&\ov \fM(\Si_0). 
\end{tikzcd}
\eest
These are all families over the parameter space $\cP$ of RT perturbations $\nu$ and we denote by $\ov \fM(-)_\nu$ the corresponding fiber of $\ov \fM(-)$.
\smallskip

In section \ref{S.VFC}, we first show that 
\bear\label{0.fam.moduli}
\ov \fM(\Si_{s_0}), \quad \ov \fM(\wt \Si),\quad \mbox{ and } \quad   \ov \fM(\Si_{0})
\eear
are thinly compactified families over the parameter space $\cP$ (in the sense of \cite{ip-thin}) and therefore carry a VFC, cf. Theorem~\ref{T.VFC.rel.2}. This involves proving that for generic $\nu$ all the strata of the moduli spaces $\ov \fM(-)_\nu$ are cut transversally. There is one subtlety: when the real locus of the target is nonempty, it also involves gluing across the codimension 1 strata to show that for generic $\nu$, the corresponding subsets 
\best
\wt \fM(\Si_{s_0})_\nu, \quad \wt \fM(\wt \Si)_\nu,\quad \mbox{ and } \quad    \wt \fM(\Si_{0})_\nu, 
\eest
(defined as the union of codimension at most 1 strata) are orientable topological  manifolds (without boundary). 
As a consequence of transversality, generically the union of the codimension at least 2 strata of $\ov \fM(-)_\nu$ have homological codimension at least 2 and thus the thin compactification method applies to define the VFC. 
\smallskip

To establish the relation \eqref{split.VFC} it suffices to compare the fundamental classes 
\best
[\wt \fM(\Si_{s_0})_\nu], \quad \Phi_*[\wt \fM(\wt \Si)_\nu], \quad \mbox{ and } \quad 
[\wt \fM(\Si_{0})_\nu]\quad \mbox{ in } \quad \cHH_*(\wt \fM(\F_{/I})_\nu; \Q)
\eest
for generic $\nu\in \cP$. These classes depend on the choice of orientation of the moduli spaces, which in turn is determined by a choice of twisted orientation data $\mo_\F$ on $\F$, cf. \S\ref{S.orient}. There are two natural perspectives on the moduli space associated to the nodal target $\Si_0$, one coming from the normalization and the other from the deformation, cf. \S\ref{S.mod.nodal}. Each perspective gives rise to a natural orientation (after fixing $\mo_\F$) and a key step is showing that these two orientations agree, proved in \S\ref{S.comp.or}.
\smallskip

We prove in \S\ref{S.mod.nodal} that for generic $\nu\in \cP$ the attaching map $\Phi$ restricts to a proper, finite degree map between the two oriented topological manifolds $\wt \fM(\wt \Si)_\nu$ and $\wt \fM(\Si_{0})_\nu$ and thus 
\bear\label{0.phi.deg}
\Phi_*[ \wt \fM(\wt \Si)_\nu]=\deg \Phi \cdot [\wt \fM(\Si_{0})_\nu] \quad \mbox{ in } \quad \cHH_*(\wt \fM(\Si_0)_\nu; \Q).
\eear
Lemma~\ref{L.phi.orient} expresses the degree of $\Phi$ in terms of specific combinatorial factors.
\smallskip

Theorem~\ref{T.splitVFC} would follow if $\wt \fM (\F_{/I})_\nu=\ma\cup_{s\in I} \wt\fM(\Si_s)_\nu$ was generically a topological cobordism; however, as in the usual splitting formula, in general it is branched along $s=0$. After passing to a cover $\wh \fM(\Si_0)$ of $\wt \fM(\Si_0)$, we construct in \S\ref{S.cobord} an auxiliary space 
\best
\wh \fM (\F_{/I})= \wh\fM(\Si_0)\cup \ma \cup_{s\ne 0} \wt\fM(\Si_s). 
\eest 
This space comes with a proper continuous projection to $\wt \fM (\F_{/I})$ which restricts to a map $q_0 : \wh \fM(\Si_0) \ra\wt  \fM(\Si_0) $. We prove in \S\ref{S.cobord} that for generic $\nu\in \cP$, 
\begin{enumerate}[(i)]
\item $q_0$ restricts to a proper map between two oriented manifolds $\wh \fM(\Si_0)_\nu$ and $\wt  \fM(\Si_0)_\nu$.  Thus 
\best
(q_0)_*[\wh \fM(\Si_0)_\nu]= (\deg q_0)\cdot [\wt \fM(\Si_0)_\nu] \quad \mbox{ in } \quad \cHH_*(\wt \fM(\Si_0)_\nu; \Q);
\eest
\item $\wh \fM (\F_{/I})_\nu$ is an oriented topological cobordism between $\wt \fM(\Si_{s_0})_\nu$ and $\wh \fM(\Si_0)_\nu$. Thus 
\best
[\wh \fM(\Si_0)_\nu] = [ \wt \fM(\Si_{s_0})_\nu]\quad \mbox{ in } \quad \cHH_*(\wh \fM (\F_{/I})_\nu; \Q). 
\eest
\end{enumerate}
This allows us to relate $[\wt \fM(\Si_{s_0})_\nu]$ to $[\wt \fM(\Si_{0})_\nu]$ up to specific combinatorial factors. Combined with \eqref{0.phi.deg}, the precise formulas for the degrees of $\Phi$ and $q_0$, and the properties of thin compactifications, this implies \eqref{split.VFC}, completing the proof of  Theorem~\ref{T.splitVFC}, cf. \S\ref{S.pf.splitVFC}. 
\medskip

\noindent {\bf Outline of paper. } In Section~\ref{S.rev} we review notation and background; a summary of the thin compactification method is included in \S\ref{S.thin.cpt} and at the beginning of \S\ref{S.VFC.mod}. In Section~\ref{S.fam.sect} we review the construction of the real relative moduli space and extend it to a family of targets.  
In Sections~\ref{S.constr.vfc} and \ref{S.VFC.mod} we construct the relevant VFCs by turning on Ruan-Tian perturbations to obtain transversality 
strata-wise and applying the thin compactification method. 
The moduli space of maps to a nodal target is analyzed in Section~\ref{S.mod.nodal}, including several equivalent descriptions of the linearization and its orientation sheaf. The auxiliary space $\wh \fM (\F_{/I})$ is constructed and analyzed in Section~\ref{S.cobord}.  The orientations of the various moduli spaces involved are discussed in Section \ref{S.orient}. The main result of Section~\ref{S.orient} is Proposition~\ref{P.phi.orient.2} which compares the two natural orientations on the moduli space associated to the nodal target. Theorems~\ref{T.splitVFC} and \ref{T.gluing} are proved in Sections \ref{S.pf.splitVFC} and \ref{S.pf.RGW}.  
The appendices provide more details on the various linearizations considered and the relations between them.

\medskip

\noindent {\bf Acknowledgments.} This paper is partially based upon work supported by the NSF grant DMS-
1440140 while the authors were in residence at MSRI during the Spring 2018 program “Enumerative
Geometry Beyond Numbers”. The authors would like to thank MSRI for the hospitality. P.G. would also 
like to thank the IHES for the hospitality during Fall 2018 and E.I. the Institut de Math\'ematiques de Jussieu - Paris Rive Gauche 
during Spring 2020 when this paper was completed. The research of E.I. is partially supported by the NSF grant DMS-1905361 and that of P.G. by the ANR grant ANR-18-CE40-0009 and the ERC Consolidator Grant ROGW-864919. We thank the referees for their meticulous reviews and numerous insightful suggestions.

\medskip

 \section{Background and notation} \label{S.rev}
 
\medskip

\subsection{Notations for partitions} Let $d\in\Z$. A partition $\la$ of $d$, denoted $\la\vdash d$, is a finite sequence of positive integers 
$\la=(\la_1\ge \dots\ge \la_\ell)$  such that the sum of its parts, denoted 
$|\la|$, is equal to $d$. The number of its parts, called  length of the partition,  is denoted $\ell(\la)$. We can also write a partition in the form $\la=( 1^{m_1} 2^{m_2} \dots )$ where $m_k$ is the number of parts of $\la$ equal to $k$. Then 
\best
d=|\la|= \sum_{i=1}^\ell \la_i =\sum_{k=1}^\infty  k m_k \quad \mbox{ and } \quad \ell(\la)=\ell=\sum_{k=1}^\infty m_k. 
\eest
Let $\Aut(\la)$ be the automorphism group of $\la$; its order is
\bear\label{aut.la}
|\Aut(\la)| = \prod m_k!.
\eear
We also consider the following combinatorial factor
\bear\label{mult.la}
\zeta(\lambda) = \prod m_k! k^{m_k}.
\eear

 \subsection{The local RGW invariants} A symmetric (or Real) curve $(C,\si)$ is a closed, oriented, possibly nodal, possibly disconnected, possibly marked complex curve $C$ together with an anti-holomorphic involution $\si$, called the real structure. In this paper we only consider the case when all marked points of $C$ come in conjugate pairs.
  
Let $(X,\omega)$ be a symplectic manifold and $\phi$ an anti-symplectic involution on $X$. A \textsf{real map}
 \bear\label{f.real}
 f:(C,\si)\longrightarrow (X,\phi)
 \eear
is a map $f: C \ra X$ such that $f\circ \si = \phi\circ f$. In this paper we restrict to almost complex structures $J$ on $X$ which are real i.e. satisfy $\phi^*J=-J$ and are tamed by $\omega$. Denote by 
\bear\label{moduli.X}
\ov\M_{d, \chi, \ell}^{\R}(X)  
\eear
the (absolute) \textsf{real moduli space} consisting of equivalence classes (up to reparametrization of the domain) of stable degree $d$ $J$-holomorphic real maps \eqref{f.real} from symmetric curves of Euler characteristic $\chi$ and $\ell$ pairs of conjugate marked points.   

 \smallskip
Throughout this paper we restrict ourselves to target manifolds $X$ which are (families of)  symmetric curves; in this case we denote the target curve by $(\Si, c)$ and the family of such by $\F$, see \S\ref{S.fam.curves}. 
 \medskip
 
Assume $(\Si, c)$ is a (smooth) symmetric curve with $r$ (ordered) pairs of conjugate marked points 
 \bear\label{marked.points}
V =\{ (x_1^+,x_1^-), \dots, (x_r^+,x_r^-)\}, \quad \mbox{ where} \quad  x_i^-=c(x_i^+),
 \eear 
and no other marked points. For a collection $\vec \mu=(\mu^1, \dots, \mu^r)$   of $r$ partitions of $d$, denote by 
\bear\label{rel.M.space}
	\ov  \M^{\R}_{d, \chi, \vec \mu}  (\Si, V)
\eear 
the real relative moduli space reviewed in \S\ref{S.rel.moduli}. Note that when $\Si$ has no marked points $V=\emptyset$ and the real relative moduli space \eqref{rel.M.space} becomes an absolute real moduli space.

The virtual dimension of the moduli space \eqref{rel.M.space} is equal to
\bear\label{dim.M=b}
b=d\chi(\Si)- \chi-2\de(\vec \mu)\,, \quad \text{ where }\quad  \de(\vec \mu)= \sum_{i=1}^r  (d-\ell(\mu^i)),  
\eear
cf. \eqref{A.v.dim=b}.

In general, the (absolute) real moduli spaces \eqref{moduli.X} are not always orientable, but there are some criteria that ensure orientability.  In  \cite[Definitions~1.2 and 5.1]{gz} a notion of real orientation was introduced. When the target $X$ is odd complex dimensional, the existence of a real orientation on $TX$ ensures that the real moduli spaces $\ov\M^\R_{d, \chi, \ell}(X)$ are orientable, and a choice determines a canonical orientation of the moduli space cf. \cite[Theorem 1.3]{gz}. This was  extended in \cite[Definitions~2.1 and A.1]{GI} to the notion of a twisted orientation, reviewed in Definition~\ref{TRO.gen} below, and used to orient the real relative moduli spaces. While a real orientation in the sense of \cite[Definition 1.2]{gz} does not exist on a symmetric curve with even genus and fixed-point free involution, a twisted orientation exists on every symmetric marked curve (with no real special points), cf. \S\ref{S.tw.exist}. 
\smallskip

For a smooth marked symmetric curve $\Si$ as in \eqref{marked.points}, its 
\textsf{relative tangent bundle} $\T_{\Si}$ is defined by
 \bear\label{T.punctured}
 \T_{\Si}= T\Si\otimes \O\big(-\ma{\textstyle\sum} _i x_i^+ -\ma{\textstyle\sum}_i x_i^-\big). 
 \eear
It is a holomorphic line bundle over $\Si$ and has a canonical real structure $c_\T$ induced by $c$, see also \eqref{T.family.C}.
\smallskip

As described after \cite[(A.13)]{GI} and reviewed in \S\ref{S.orient}, the moduli space $\ov  \M^{\R}_{d, \chi, \vec \mu}  (\Si, V)$ is orientable, and can be canonically oriented by a choice of twisted orientation data $\mo$ on the relative tangent bundle  $\T_\Si$. This gives rise to a virtual fundamental class 
denoted $[\ov \M_{d, \chi, \vec \mu}^{\R}(\Si,V)]^{\vir, \mo}$, which depends on $\mo$. 
 For any holomorphic line bundle  $L \ra \Si$, we can consider its index bundle $\mathrm{Ind}  \; \del_{L}$ as an element in 
 $K$-theory, see \eqref{def.ind}.  Denote
\bear\label{RGW.chern}
RGW^{c,\mo}_{d, \chi} (\Si, L)_{\vec \mu}\; \ma=^{\mathrm{def}}  \frac 1{|\Aut(\vec \mu)|}\ma \int_{[\ov \M_{d, \chi, \vec \mu}^{\R}(\Si,V)]^{\vir,\mo}}  \hskip-.3in  
c_{b/2}(-\mathrm{Ind} \; \del_{L}). 
\eear
Here $b$ is the (virtual) dimension \eqref{dim.M=b} of the moduli space and $c_k(E)$ denotes the $k$'th Chern class of $E$. 
The \textsf{local RGW invariant} defined by \cite[(2.24)]{GI} is then equal to 
  \bear\label{RGW.real.rel}
 RGW^{c,\mo}_{d} (\Si,L)_{\vec \mu} &=&  \sum_{\chi}RGW^{c,\mo}_{d, \chi} (\Si,L)_{\vec \mu} \; 
 t^{-\chi/2}  (u/t)^{b/2+dk}
 \eear
cf. Remark~\ref{R.moduli.ordered}. Here $k=c_1(L)[\Si]$ is the degree of the line bundle. 
\medskip

While there may be different ways of defining the VFC, in this paper we describe a specific 
construction  of the VFC 
 \bear\label{VFC.M}
 [\ov \M_{d, \chi, \vec \mu}^{\R}(\Si, V)]^{\vir,\mo}\in \cHH_b(\ov \M_{d, \chi, \vec \mu}^{\R}(\Si, V);  \Q)
 \eear
as an element of rational Cech homology, using the thin compactification method as introduced in \cite[\S2]{ip-thin} and briefly reviewed in \S\ref{S.thin.cpt}. This has the advantage that it is very concrete and does not use sophisticated virtual techniques. In the case when the target is a Riemann surface, turning on certain geometric perturbations 
$\nu$ of the $J$-holomorphic map equation as introduced by Ruan-Tian \cite{RT} suffices to obtain transversality strata-wise for the real relative moduli space (after passing to a cover of the Deligne-Mumford moduli space; see Remark \ref{R.RT.cover}). This ensures that  the moduli space is generically a thin compactification, and thus carries a VFC for all Ruan-Tian perturbations, including for $\nu=0$. We present the details of this construction in \S\ref{S.constr.vfc} and \S\ref{S.VFC.mod}.

\subsection{Thin compactifications}\label{S.thin.cpt} Here we briefly summarize the method of thin compactifications as introduced in \cite{ip-thin} and how it applies to construct the virtual fundamental class (VFC) of the real relative moduli spaces considered in this paper; the details of this construction are presented in \S\ref{S.VFC.mod}. 
\smallskip

Throughout this paper, by a $d$-dimensional manifold we mean a Hausdorff space $M$ locally modeled on $\R^d$. If $M$ is also oriented (but not necessarily compact), it has a  \textsf{fundamental class} in Steenrod homology and we denote by $[M]\in \cHH_d(M; \Q)$ its image in rational Cech homology. There are some subtleties involved when working with these homology theories; in particular, to pushforward such classes one needs a {\em proper} continuous map (between locally compact Hausdorff spaces), cf. \cite[\S1]{ip-thin}. However, rational Cech homology satisfies a continuity axiom, a relative homeomorphism axiom and is also exact; in particular, there is a natural long exact sequence
\bear\label{LES.Cech}
\dots \ra \cHH_{*}(A;\Q) \ra \cHH_*(X;  \Q) \ma\ra^\rho \cHH_*(X\setminus A; \Q)\ra \dots 
\eear
associated to a closed pair $(X, A)$, where $\rho$ is the "restriction" to an open set, cf. \cite{milnor} and \cite[\S1]{ip-thin}.

\medskip 

If $M$ is an oriented $d$-dimensional manifold, a \textsf{thin compactification} of $M$ (in the sense of  \cite[\S2]{ip-thin}) is a compact Hausdorff space $\ov M$ containing $M$ such that $S=\ov M\setminus M$ is closed and (homologically) codimension 2, i.e. $ \cHH_*(S; \Q)=0$ for all $*>d-2$. Then $\rho: \cHH_*(\ov M;  \Q)\ra \cHH_*(M; \Q)$ is an isomorphism for $*=d$, cf.  \eqref{LES.Cech}, thus the fundamental class $[M]\in \cHH_*(M;  \Q)$ uniquely lifts to a class on $\ov M$ denoted $[\ov M]\in \cHH_*(\ov M;  \Q)$. Thinly compactified cobordisms are defined similarly, cf. \cite[\S2.3]{ip-thin}.

\medskip

Roughly speaking, the thin compactifications method  of \cite{ip-thin} applies to a moduli problem whenever 
(i) for generic perturbation the moduli space is a thinly compactified manifold, i.e. is an oriented manifold away from 
(homologically) codimension 2 strata 
 and (ii) the moduli space over a generic 1-parameter family of perturbations is a thinly compactified cobordism, i.e. is 
 an oriented cobordism away from (homologically) codimension 2 strata. Since the fundamental class (in rational Cech homology) of a manifold uniquely extends to any thin compactification, the first condition defines the fundamental class of the moduli space for generic parameter. For non-generic parameter, the virtual fundamental class is obtained as the limit (in rational Cech homology) of nearby classes, which is well defined by condition (ii). For more details, see  \cite[Theorem~4.2]{ip-thin}. 
\smallskip
 
The ideal situation in which the thin compactification method applies is when all the strata of the moduli space are cut transversally, and the boundary strata have (virtual) codimension at least 2, cf. \cite[Lemmas~2.2 and 5.3]{ip-thin}.  More generally, it suffices to cover the boundary by images of codimension 2 manifolds, as in  \cite[Lemma~2.2]{ip-short}.

\medskip

Because the targets considered in this paper are holomorphic curves, we are in a very special situation where geometric Ruan-Tian perturbations work particularly well, and are sufficient to obtain transversality strata-wise on all strata of the real relative moduli spaces, with one exception. When the target is a sphere without any marked points, the (absolute) moduli space of maps from a genus 1 domain with no marked points has a few strata which are not cut transversally, cf. Remark~\ref{R.unstable.tori}. In all the other cases, the RT perturbations suffice to ensure that all the strata are cut transversally,  as we outline below.  For this reason, we work under the mild technical assumption that the target $\Si$ has a marked point on each spherical component, which suffices for our considerations, cf. Remark~\ref{R.VFC.tori}; for precise details, see Lemma~\ref{L.transv.rel}. 
\smallskip

Classically, it is known that one can use Ruan-Tian perturbations $\nu$ to get transversality on all stable components of the domain (after first passing to a cover of the Deligne-Mumford moduli space as in Remark~\ref{R.RT.cover}). However, Ruan-Tian perturbations identically vanish on unstable domain components, so these components must be handled by other methods. When the domain is a genus $g$ curve with at least $3-2g$ marked points, these are the spherical components collapsed to points under the map to the Deligne-Mumford moduli space. Otherwise, the entire domain is unstable (when it has genus 1 and no marked points or genus 0 with fewer than 3 marked points). 

However, when the target is aspherical (e.g. a smooth higher genus curve) unstable spherical domain components cannot occur for maps in the moduli space. When the target is a sphere, there could be unstable spherical domain components, but these are well understood classically, and in particular are cut transversally (even though they are multiple covers!). This is also true for all the strata of the moduli space of (holomorphic) maps from a genus 1 curve without marked points, as long as the target has genus one (or higher). 

\smallskip
For the relative moduli spaces, there are also rubber components involved, but in our case these project to constant maps   into the divisor which satisfy an additional condition that depends on $\nabla \nu$, cf. \eqref{ker.nor.not.0}. This condition is automatically satisfied on genus 0 components and is again cut transversally on the other rubber components because we can use $\nabla \nu$  to get transversality. Moreover, any stratum with rubber components has codimension at least 2. For more details, see \S\ref{S.VFC.rel}. 

\smallskip
The moduli spaces of real maps considered in this paper may have codimension 1 strata. In this case, we first argue that for generic parameter $\nu$ the union of the codimension at most one strata is an oriented topological manifold (without boundary), and thus carries a fundamental class (in rational Cech homology); this involves only standard gluing techniques at an ordinary real node of the domain (and no rubber components). All the other boundary strata have codimension at least two, thus the method of thin compactifications applies. For details, see \S\ref{S.VFC.mod}.

\medskip
\section{Family moduli spaces} \label{S.fam.sect}
\medskip

A key step in proving the splitting formula for the RGW invariants is to consider a family of 
moduli spaces associated to a family  of symmetric marked curves degenerating to a nodal symmetric curve (with a conjugate pair of nodes). This family moduli space, denoted  $\ov\M(\F_{/\Delta})$, serves as the ambient space where we can compare the VFCs and integrands used in defining the RGW invariants. In this and the following two sections we set up the necessary notation, review the constructions and show that  the moduli spaces involved in defining the RGW invariants \eqref{RGW.real.rel} extend over families of symmetric marked curves, including across the singular fibers.
 
\subsection{Families of symmetric curves} \label{S.fam.curves} Recall that if $\Si_0$ is a complex nodal marked curve, then we can consider
\begin{enumerate}[(a)] 
	\item a smooth normalization $\wt \Si$ of $\Si_0$ that replaces each node by a pair of marked points; 
	\item a (flat) family $\F$ of deformations $\Si_s$ of $\Si_0$ smoothing out the nodes. 
\end{enumerate} 

\medskip
Let $(\Si_0, c_0)$ be a nodal symmetric curve with $r$ pairs of conjugate marked points $V_0$ and a single pair of conjugate nodes $x^\pm$. Let 
\bear\label{family.smoothings} 
\pi: (\mathcal{F}, c_\mathcal{F})\lra (\Delta^2, c_\Delta), 
\eear
denote a (flat) family of deformations of $(\Si_0, c_0)$ smoothing the nodes as defined in \cite[\S4.2]{GZ2}, cf.  \S\ref{S.A.smoothings}. Here $\De\subset \cx$ is the unit disk and $\De^2=\De\ti \De$.  The total space $\mathcal{F}$ is a smooth K\"ahler manifold with complex structure $J$ and the projection is holomorphic. Moreover, $c_\F$ is a real structure on $\F$ which is anti-holomorphic and restricts to the real structure $c_0$ on the central fiber $\Si_0$, while $c_\Delta: \Delta^2 \ra \Delta^2$ is defined by $(s_1, s_2)\mapsto (\bar s_2, \bar s_1)$.  Finally, marked points give rise to sections of $\F\ra \De^2$; their images define a smooth divisor $V\subset \cal F$.

\smallskip

The fiber of \eqref{family.smoothings} over $(s, \ov s)$ is a symmetric marked curve denoted $\Si_s$; its real structure $c_s$ is the restriction of $c_{\mathcal{F}}$ to $\Si_s$ while its marked points correspond to the restriction $V_s$ of $V$ to $\Si_s$. This determines a family 
\bear\label{fam.smooth.real} 
\F_{/\De}\; \ma=^{\mathrm{def}}\; \ma \cup_{s\in \De} \Si_s
\eear
of symmetric marked curves over $\De$ (the pullback of \eqref{family.smoothings} via  the map $\De \ra \De^2$, $s\mapsto (s, \ov s)$ for all $s\in \De$). We will later consider restrictions of  this family \eqref{fam.smooth.real}  to a line, path, etc.  

\medskip
Consider also 
\bear\label{nor.Si.0} 
(\wt \Si, \wt c)\ra (\Si_0, c_0)
\eear 
the normalization of the singular fiber $\Si_0$ (as a marked symmetric curve). Here $\wt c$ denotes the real structure of $\wt \Si$ and we denote by $\wt V$ the collection of marked points of $\wt \Si$ consisting of one point over each marked point of $\Si_0$, and a pair of points over each node of $\Si_0$. The curves $\Si_s$ and  $\wt \Si$ come with natural maps  
\bear\label{si.to.F}
\iota_s:(\Si_s, c_s)\hookrightarrow (\F, c_\F) \quad \mbox{ and }\quad \phi: (\wt \Si, \wt c) \ra (\F, c_\F). 
\eear
into the total space $\F$ of the family; the second map factors through the nodal fiber.  

Finally, let 
\bear\label{rel.T.fam}
(\T, c_\T)\ra (\F, c_\F)
\eear
denote the relative tangent bundle to the family \eqref{family.smoothings}; here $\cal T$ is a holomorphic line bundle (locally free sheaf) over $\F$  which comes with an induced real structure $c_\T$, cf. \cite[Lemma 4.8]{GZ2} and \eqref{T.family.C}. The pullback of  $(\T, c_\T)$ to both $\Si_s$ and $\wt \Si$ under \eqref{si.to.F} gives their corresponding relative tangent bundle \eqref{T.punctured}. The relative tangent bundle $\T_{\Si_0}$ to the nodal curve $\Si_0$ is defined as the restriction of $\T$ to $\Si_0$, and it fits in the normalization short exact sequence of holomorphic sheaves
\bear\label{T.SES}
0\ra \T_{\Si_0} \ra \T_{\wt \Si} \ra \T_{|nodes} \ra 0
\eear
compatible with the real structures.


\subsection{Family moduli space}\label{S.fam.moduli} The real relative moduli spaces of maps into the smooth fibers $\Si_s$ naturally form a family which can be compactified by adding a fiber over $s=0$.   In parallel with the proof of the usual splitting formula e.g \cite[(14)]{bp-TQFT}, we consider the family 
\best
\ov{\M}_{d, \chi, \vec \mu}(\F_{/\De})=\ma\cup_{s\in \De} \ov{\M}^\R_{d, \chi, \vec \mu}(\Si_s,V_s)\longrightarrow \De
\eest 
of moduli spaces whose fiber at  $s\in \De\setminus 0$ is  the real relative moduli space $\ov\M^\R(\Si_s, V_s)$, while the fiber 
$\ov{\M}^\R(\Si_0, V_0)$ over $s=0$ includes maps with rubber components over both the nodes and the marked points of $\Si_0$. We describe these spaces in more detail below.

\subsection{Relative real moduli spaces} \label{S.rel.moduli} 
Even in the complex category, there are several versions of the relative moduli space of holomorphic maps to a complex curve $\Si$ relative to a divisor $V$. The version used by Bryan-Pandharipande in \cite[Definition~3.1]{bp1} is more convenient for computational purposes, and is a finite quotient of the standard one defined by Jun Li in \cite{li}. The latter has the property that all the contact points are ordered, and is  more convenient for analytical considerations, including for constructing the VFC and describing its behavior under target degenerations. Of course, the virtual fundamental classes of these two versions of the moduli space are essentially the same up to a combinatorial factor. 
\smallskip
 
In this section we outline the construction of the relative moduli space, in which all the contact points are ordered, adapted to the real setting; see also Remark~\ref{R.moduli.ordered}. We include some of the standard arguments for ease of reference when we extend these arguments to families of degenerating targets.  

\medskip
Let $(\Si, c)$ be a (smooth) symmetric marked curve with $r$ pairs of conjugate marked 
points $V$ as in \eqref{marked.points}. Fix $d$, $\chi$ and a collection $\vec \mu=(\mu^1, \dots, \mu^r)$ of $r$ partitions of $d$. Consider first the "top stratum" 
\bear\label{top.stratum.la}
\M^\R_{d, \chi, \vec \mu} (\Si, V)
\eear
  consisting of equivalence classes (up to reparametrizations of the domain) of $J$-holomorphic real maps $f:(C,\si)\ra (\Si, c)$ such that 
\begin{enumerate}[(i)]
	\item the domain $C$ is a smooth, marked, possibly disconnected symmetric Riemann surface of Euler characteristic $\chi$;
\item $f$ has ramification pattern $\mu^i$ over $x_i^+$ (and thus also over its conjugate $x_i^-$), for all $i=1, \dots, r$;
\item $f$ is nontrivial on each connected component of $C$;
\item $f$ is a degree $d$ map, i.e. the image of $f$ represents $d[\Si]$ in $H_2(\Si; \Z)$;
\end{enumerate}
The points in $f^{-1}(V)$ are called \textsf{contact points} of $f$ to $V$, and here all the contact points are {\em marked} points of the domain (and the domain has no other marked points); see also Remark~\ref{R.moduli.ordered} below.
Specifically, condition (ii) means that
\begin{itemize} 
\item $f^{-1}(x_i^\pm)=\{ y_{ij}^\pm\}_{j=1, \dots, \ell(\mu^i)}$ for every  $i=1, \dots, r$; in particular, $f(y_{ij}^+)=x_i^+$.
\item the ramification order of $f$ at $y_{ij}^\pm$ is $\mu_j^i$ and  $y_{ij}^\pm$ are conjugate points; 
\item $\{ y_{ij}^\pm \;|\; j=1, \dots, \ell(\mu^i), \; i=1, \dots, r \}$ are marked points of the domain.
\end{itemize}
A   map $f$  which satisfies these conditions is said to have its contact to $V$ prescribed by $\vec\mu$. 

The moduli space \eqref{top.stratum.la} has a compactification $\ov \M^\R_{d, \chi, \vec \mu}(\Si, V)$ in which both the domain and target of the maps is allowed to degenerate. We start by briefly describing the targets, denoted $\Si[m]$, and called \textsf{buildings obtained by rescaling $\Si$ around to $V$}; for the general rescaling construction 
normal to a divisor see for example \cite[\S4]{i-nc}.
In our case $\Si[m]$ is a nodal symmetric marked curve, obtained from $\Si$ by attaching chains of spheres at its marked points in the following manner. Let $N_V$ be the normal bundle of $V$ in $\Si$, and consider the projectivization $\P_V=\P(N_V\oplus \cx)$. Then 
\bear\label{def.PV}
\P_V=\P^1\ti V= \ma\sqcup_{i=1}^r \P^1\ti \{x_i^\pm\}=  \ma\sqcup_{i=1}^r \P_{x_i^\pm}
\eear
is a $\P^1$ bundle over $V=\{ x_1^\pm, \dots, x_r^\pm\}$ with
\begin{enumerate}  [(i)] 
	\item a zero and infinity section $V_0$ and $V_\infty$,
	\item a real structure induced by the one on $N_V$, thus covering $c(x_i^+)=x_i^-$ in the base, and 
	\item a $\cx^*$ action on each $\P_{x_i^\pm}$, fixing (i) pointwise and compatible with (ii) i.e. satisfying $c(\la\cdot z)= \ov \la \cdot c(z)$. 
\end{enumerate} 
Assume first for simplicity that $V$ consists of a single pair of marked points, i.e. $r=1$. Starting with $\Si$ and inductively rescaling it $m$ times around $V$ gives rise to the building 
\best
\Si[m]= \Si \ma\cup_{V=V_\infty} \P_V \ma\cup _{V_0=V_\infty}  \dots \ma\cup _{V_0=V_\infty}  \P_V \longrightarrow \Si, 
\eest
with a divisor $V[m]\subset \Si[m]$ corresponding to the zero divisor of the last copy of $\P_V$. 
For the general case, we allow $\Si$ to be {\em independently} rescaled $m_i$ times around each pair $\{x_i^\pm\}$ of conjugate points in $V=\{x_1^\pm, \dots, x_r^\pm\}$. This similarly gives rise to a building denoted $\Si[m]$ with a divisor 
$V[m]=\{x^\pm_i\}_{i=1}^r$.  The $\cx^*$ action on each $ \P_{x_i^\pm}$ induces a $(\cx^*)^{|m|}$ action on the building $\Si[m]$, denoted 
$$
t \mapsto R_t\in \Aut (\Si[m])
$$
and called the  \textsf{rescaling action}; here $m=(m_1, \dots, m_r)$ and $|m|=m_1+\dots+m_r$. Finally, let 
\bear\label{colapse.map}
p: \Si[m] \longrightarrow \Si
\eear 
denote the projection induced by collapsing all the $m_i$ copies of $\P_{x_i^\pm}$ for  $i=1, \dots, r$ down to $V$. 
\medskip

Then the compactification $\ov\M^\R_{d, \chi, \vec \mu} (\Si, V)$ of the top stratum \eqref{top.stratum.la} is defined as follows. 
\begin{defn}\label{D.moduli.rel.or} 
An element of the moduli space $\ov\M^\R_{d, \chi, \vec \mu} (\Si, V)$ is an equivalence class, up to reparametrizations of the domain and rescaling the target, of real $J$-holomorphic maps 
\bear\label{f.to.Sm}
f:C\ra \Si[m]
\eear
from some  possibly nodal, possibly disconnected symmetric Riemann surface $C$ to some symmetric building $\Si[m]$ such that 
\begin{enumerate}[(i)]
\item $f$ has prescribed contact to the marked points of the target, i.e. the preimage of $\{x^\pm_i\}$ consists only of marked points of the domain, denoted $\{ y^\pm_{ij}\}$, and $f$ has order of contact $\mu_j^i$ at $y^\pm_{ij}$ to $x_i^\pm$; in particular, $f(y^+_{ij})= x_i^+$ and $y^\pm_{ij}$ are conjugate marked points. 
\item $f$ satisfies the following matching condition: the preimage of the nodes of the target  consists only of nodes of the domain with the same order of contact on the two local branches. 
\item the restriction of $f$ to every connected component of the domain is nontrivial. 
\item $f$ has degree $d$ and its domain has (virtual) Euler characteristic $\chi$.
\item $f$ is \textsf{relatively stable}, i.e. $\Aut f$ is finite; here an automorphism of $f:C\ra \Si[m]$ is a pair 
$(\varphi, t)\in \Aut(C) \ti (\cx^*)^{|m|}$ such that $R_t \circ f \circ \varphi=f$. 
\end{enumerate} 
\end{defn} 
The real relative moduli space 
\bear\label{rel.order.all}
 \ov\M^\R_{d, \chi, \vec \mu} (\Si, V)
\eear
comes with natural maps induced by forgetting some of the data such as the real structures, the order of the contact points, the divisor $V$, etc. Forgetting the real structure defines a map to the (usual) relative moduli space of holomorphic maps to $\Si$ relative the divisor $V$. 
\smallskip

\begin{rem}\label{R.moduli.ordered}
 Unlike in \cite[Definition~2.5]{GI} (or \cite[Definition~3.1]{bp1}), throughout this paper we are using the standard definition of the relative moduli space (cf. \cite{li} or \cite{ip-rel}) in which all the contact points are marked. The moduli space \eqref{rel.order.all} comes with a group action permuting the contact points with same image and multiplicity; the quotient by this $\Aut (\vec \mu)$ action is the moduli space in \cite[Definition 2.5]{GI}. 
 In particular, the former moduli space is the degree $|\Aut(\vec\mu)|=\prod_i |\Aut(\mu^i)|$ cover of the latter moduli space, obtained by ordering the contact points.
 \end{rem}

\subsection{Rubber components and stratification}\label{S.rubber} There is an equivalent description of the elements of the relative moduli space that is more convenient for transversality purposes. Decompose any real map  \eqref{f.to.Sm} satisfying conditions (i)-(v) of Definition \ref{D.moduli.rel.or} into components, obtained by restricting $f$ to the irreducible components of its domain.  
These components can be grouped according to their image in $\Si[m]$; those that are mapped to $\Si$ are called \textsf{level zero components} and the rest are called \textsf{rubber components}. In turn, the rubber components can be grouped according to which copy of $\P_{x_i^\pm}$ in $\Si[m]$ their image lands in, and the vertices of the dual graph of $f:C\ra \Si[m]$ are decorated by this discrete data (along with the genus and the homology class represented by each irreducible component). 

By Definition~\ref{D.moduli.rel.or}(ii), each node $q$ of the domain has an associated order of contact $\la(q)$, where $\la(q)=0$ if $q$ is an ordinary node i.e. mapped away from the nodes of $\Si[m]$. Each marked point of the domain also comes with an associated order of contact cf. Definition~\ref{D.moduli.rel.or}(i). Altogether, this associates to every special point $q$ of the domain a 
multiplicity $\la(q)$, and the dual graph of $f$ is decorated by this data as well. 

\begin{rem}\label{R.rubber.comp.pts} Note that every rubber component $f_k$ of $f$, regarded as a map $f_k:C_k \ra \P_V$ either represents 0 in 
homology (and its image is disjoint from the  $0$ and $\infty$ divisors) or else its domain has at least two special points (the inverse image of $0$ and $\infty$), cf. Definition~\ref{D.moduli.rel.or}(i)-(ii).
\end{rem}

For each decorated dual graph $\tau$, let $\M_\tau$ denote the corresponding stratum of the real relative moduli space. Elements of this stratum can also be described 
in terms of their projection under the collapsing map \eqref{colapse.map}, and the lifts of this projection to meromorphic sections of a line bundle as follows.

 Since $V$ is 0 dimensional, each rubber component projects to a constant map to $V$ under \eqref{colapse.map}, and is a holomorphic map to one of the $\P_V$'s, with prescribed contact to 0 and $\infty$ (determined by $\tau$), cf. Definition \ref{D.moduli.rel.or}(i)-(ii). Equivalently, it is a meromorphic section $\xi\ne 0$ of the pullback normal bundle of $V$ with zeros and poles at the special points of prescribed order (and no other zeros/poles). Note that any two such sections, if they exist, must be constant multiples of each other. For more details,  see \S\ref{S.normal.compl}. 
\smallskip

Some rubber components may be multiple covers $\P^1\ra \P^1$ totally ramified over 0 and $\infty$; these are called \textsf{trivial components}, and their domain is a sphere with precisely two special points. All the other (nontrivial) rubber components project to a stable map to $V$ cf. Remark \ref{R.rubber.comp.pts}. 

\medskip

In particular, the projection \eqref{colapse.map} induces a forgetful map 
\bear \label{forget.V} 
\ov\M^\R(\Si, V) \ra \ov \M^\R(\Si, \emptyset)
\eear
to the (absolute) real moduli space; it takes $f: C\ra \Si[m]$ to the map obtained from 
$p\circ f:C\ra \Si$ after collapsing all the chains of trivial components in the domain to obtain a stable map to $\Si$. 

\begin{rem} \label{R.stable.vs.rel} (Stability vs relative stability) Let $f: C\ra \Si[m]$ be a real map which satisfies conditions (i)-(iii) of Definition~\ref{D.moduli.rel.or}. 
Then $f$ is relatively stable if and only if (a) $f$ is stable as a map to $\Si[m]$ and (b) every copy of $\P_{x_i^\pm}$ in $\Si[m]$ contains at least one nontrivial rubber component of $f$. 
\end{rem} 

\begin{rem}\label{R.im.rel}  By definition, the image of the forgetful map \eqref{forget.V} consists of stable, real holomorphic maps $f:C\ra \Si$ that have a holomorphic lift $\wh f:\wh C\ra \Si[m]$ which satisfies the matching conditions over the singular locus of $\Si[m]$ and has prescribed contact with $V[m]$, cf. Definition \ref{D.moduli.rel.or}.  In particular, the components of $f$ that are mapped to $V$ must have a lift to a rubber component. 
\end{rem}

\subsection{Nodal targets}\label{S.nod.target} Definition~\ref{D.moduli.rel.or} extends verbatim to nodal targets, as long as 
we  include maps with rubber components over both the marked points and the nodes of the target. Assume $\Si_0$ is a nodal symmetric  marked curve with one ordered pair $\{x^\pm\}$ of conjugate nodes, $r$ pairs of conjugate marked points $V_0$, 
and no other special points. Then the real relative moduli space 
\best
\ov\M^\R_{d, \chi, \vec \mu}(\Si_0, V_0)
\eest
is defined as in Definition~\ref{D.moduli.rel.or} except that $\Si[m]$ is replaced by a building $\Si_0[m]$ obtained by rescaling $\Si_0$ at both the marked points and the nodes; in particular, chains of spheres are also inserted at the nodes of $\Si_0$, in addition to those inserted at the marked points. 
Therefore an element of this moduli space is an equivalence class of real $J$-holomorphic maps 
\best
f:C_0 \ra \Si_0[m]
\eest
satisfying conditions (i)-(v) of Definition~\ref{D.moduli.rel.or}. Now the domain $C_0$ must be nodal and the top stratum $\M^\R(\Si_0, V_0)$ is the subset consisting of maps to $\Si_0$ whose domain has no other nodes besides those in the inverse image of the nodes 
of $\Si_0$. 

Recall that \eqref{nor.Si.0} attaches pairs of marked points of $\wt \Si$ to produce the nodes of $\Si_0$; it extends to a map between any building associated to $\wt \Si$ and the corresponding building associated to $\Si_0$. The attaching map 
\bear\label{map.attach.0}
\begin{tikzcd}
\ma\bigsqcup_{\la\vdash d} \ov\M^\R_{d, \chi+4\ell(\la), \vec \mu, \la, \la}(\wt \Si, \wt V) \ar[r, "\Phi"]& \ov\M^\R_{d, \chi, \vec \mu} (\Si_0, V_0) 
\end{tikzcd}
\eear
is then induced by attaching pairs of marked points of both the domain and target to produce nodes (then forgetting the order of these nodes). Note that in the domain of \eqref{map.attach.0} the contact points are {\em ordered}; however $\Phi$ factors through the quotient by the diagonal $\Aut(\la)$ action to produce {\em unordered} nodes in its image.
\subsection{Maps to the family} Consider the family 
 \bear\label{mod.fam.la}
\ov \M_{d, \chi, \vec\mu} (\F_{/\De})=\ma\cup_{s\in \De} \ov\M^\R_{d, \chi, \vec \mu} (\Si_s, V_s)\longrightarrow \De
\eear 
moduli spaces associated to the family \eqref{fam.smooth.real} of targets; its fiber over $s\in \De$ is the real relative moduli space associated to the fiber of $\F$ over $s$  defined above. In particular, the fiber $\ov{\M}^\R(\Si_0, V_0)$ over $s=0$ includes maps with rubber components over both the nodes and the marked points of $\Si_0$.

The inclusion of $\Si_s$ into $\F$ is holomorphic, and induces a proper map
\bear\label{map.include} 
\begin{tikzcd}
\ov\M^\R_{d, \chi,\vec\mu} (\Si_s, V_s) \ar[r, hook]& \ov\M_{d, \chi,\vec\mu}(\F_{/\De}) 
\end{tikzcd}
\eear
for every $s\in \De$. The map \eqref{nor.Si.0} is also holomorphic and induces the proper map \eqref{map.attach.0} at the level of moduli spaces; the composition of the latter with the map \eqref{map.include} for $s=0$ is a proper map 
\bear\label{map.attach}
\begin{tikzcd}
\ma\bigsqcup_{\la\vdash d} \ov\M^\R_{d, \chi+4\ell(\la), \vec \mu, \la, \la}(\wt \Si, \wt V) \ar[r]& \ov\M_{d, \chi, \vec \mu} (\F_{/\De}).  
\end{tikzcd}
\eear
\begin{rem} The topology of these real relative moduli spaces is a refinement of the usual Gromov topology, defined by a process of rescaling the target normal to the divisors cf. \cite{ip-rel} and \cite{ip-sum}; in particular, the topology on the subset \eqref{top.stratum.la} consisting of maps without rubber components is the usual Gromov topology. 
\end{rem}
\section{Perturbations and transversality}\label{S.constr.vfc}
\medskip

In this section we introduce spaces of Ruan-Tian perturbations adapted to our setting and show that the strata of the real relative moduli spaces described above are cut transversally over such parameter spaces.

\subsection{Ruan-Tian perturbations} \label{S.RT.perturb} To obtain transversality stata-wise, we fix the (integrable) complex structure $J$ on the target, but turn on Ruan-Tian perturbations $\nu$ adapted to the situation. In particular, for the real relative moduli space $\ov\M^\R(\Si, V)$, we restrict to the space $\JV^\R(\Si, V)$ of RT-perturbations  on $\Si$ compatible with both the real structure and the divisor $V$, as described in this section. 

In general, recall that if $J$ is an (almost) complex structure on $X$ and $\nu$ is a Ruan-Tian perturbation \cite{RT}, 
then a  $(J, \nu)$-holomorphic map to $X$ is a solution $f:C\ra X$ of the equation
\bear\label{j.nu.holo.eq}
\del_J f =\nu|_{f}  
\eear
or equivalently the graph $F$ of $f$ is $J_\nu$-holomorphic, as reviewed for example in \cite[\S3.1]{ip-don}. If the domain $C$ has trivial automorphism group, one can use the variation in $\nu$ to show that the linearization is cut transversally, essentially because the graph of $f$ is an embedding. This extends to the case the domain $C$ is stable after passing to a  regular cover of the Deligne-Mumford moduli space as in Remark~\ref{R.RT.cover} below. 

\begin{rem}\label{R.RT.cover} In general, for higher genus domains, passing to covers of the moduli spaces is needed to kill the automorphism groups and turn on Ruan-Tian perturbations, as reviewed for example in \cite[\S3.1]{ip-don}. This is classically achieved by working on a regular cover of the Deligne-Mumford moduli space, defined by considering curves with extra decorations such as level structures or twisted bundles eg. as constructed by Abramovich-Corti-Vistoli in \cite{acv}, cf. \cite[Chapter XVI, Theorem 7.1]{acgh}; see also Chapter XVI, \S10 of \cite{acgh}. Such regular cover comes with an universal curve $\U$ whose fiber at a (decorated stable) curve $C$ is $C$. 
\end{rem}
\smallskip

When $(X, c)$ is a manifold with a real structure $c$, denote by $\JV^\R(X)$ the space of real Ruan-Tian perturbations on $X$, as defined for example in \cite[\S2 and \S3.1]{Z}. These are constructed as follows. Using the forgetful map 
from the real Deligne-Mumford moduli space $\ov\M_{\chi, \ell}^\R$ to the complex one, and a regular cover of the latter, one constructs a cover of $\ov\M_{\chi, \ell}^\R$ and a universal curve $(\U, c_\U)$. The fiber of $(\U, c_\U)$ at a (decorated stable) symmetric curve $(C, \si)$ is $(C, \si)$.

A real RT perturbation on $X$ is an element $\nu \in \Hom^{01}(\T_{\U}, TX)$ defined on 
$\U \ti X$, such that
\bear \label{nu.real}
\mbox{ $\nu$ is real and is supported away from the special points of the domain. }
\eear
Here  $\T_{\U} \ra \U$ is the relative tangent bundle to the fibers of $\U$, cf. \eqref{T.family.C}, and $\Hom^{01}$ denotes the space of anti-complex linear homomorphisms i.e. such that $J\nu+\nu j =0$, where $J$ is the (almost) complex structure on $(X,c)$, and $j$ is the family of complex structures on the fibers of $\U$. Such a homomorphism $\nu$ can be regarded as section of the bundle 
$(\T_\U)^\vee \boxtimes_\cx TX$ which has an induced real structure; $\nu$ is called {\em real} if this section is invariant with respect to the real structures.  For details, see \cite[\S2 and \S3.1]{Z}. 
\medskip

When $X=\Si$ is a symmetric curve with conjugate marked points $V$, consider the subset 
 \bear\label{RT.relative} 
 \JV^\R(\Si, V)  \subseteq \JV^\R(\Si) 
 \eear
 of real RT perturbations $(J,\nu)$ which are compatible with $V$ in the sense of \cite[Definition~3.2]{ip-rel}. 
 Since here $J$ is integrable and $V$ is 0 dimensional, these conditions reduce to the requirement that $\nu$ vanishes along $\U\ti V$ and $\nabla \nu$ is complex linear along $\U\ti V$. Specifically, the conditions are 
 \bear\label{nu.rel}
 \nu|_{\U \ti V} =0 \quad \mbox { and } \quad (\nabla_{J w} \nu)(v)= J(\nabla_{w} \nu)(v)  \mbox{ for all }
 w\in TX|_V  \mbox { and } v\in \T_\U. 
 \eear
 Note that the last condition is equivalent to 
 \bear\label{nu.rel.2}
  \nabla_w \nu= \al_V \otimes_\cx w \quad \mbox{ for all $w\in N_V$}
  \eear 
 where $\al_V \in \Hom^{01}(\T_\U; \cx)$ is defined on $\U \ti V$. 
 
Thus, as in \cite[\S3]{ip-rel}, the compatibility condition with the divisor $V$ ensures that 
\begin{enumerate}[(a)]
\item $\nu$ restricts to a RT perturbation on $V$ (i.e. $\nu$ takes values in $TV$ along $V$);
\item for maps with image in $V$, the normal operator $L^N=\del -\nabla \nu$ is complex linear; 
\item the 1-jet of $\nu$ determines a RT perturbation on $\P_V$ and on (the normalization of) 
the building $X[m]$ obtained by rescaling $X$ normal to $V$.
\end{enumerate} 
Property (c) follows using the fact that 
\bear\label{PV.swap} 
\P_V=\P(N_V\oplus \cx)=\P(\cx \oplus (N_V)^\vee)
\eear
has a $\cx^*$ action and an involution swapping the zero and infinity divisors (induced by $z\mapsto z^{-1}$). Let $\zeta$ denote the canonical $\cx^*$-equivariant vector field  on $\P_V$. This vanishes along the zero and infinity divisors $V_0\sqcup V_\infty$, and its restriction to $N_V\subset \P_V$ under the inclusion $w\mapsto [w, 1]$ is given by $\zeta_w=(w; w)\in  T_w N_V$ for all 
$w\in N_V$. 
Then the restriction \eqref{nu.rel.2} of $\nabla \nu$ induces a RT perturbation on $N_V\subset \P_V$, defined as the tensor product of the pullbacks of $\al_V$ and $\zeta$. This tensor product $\al_V\boxtimes \zeta$  can be regarded as a $\cx^*$-equivariant RT perturbation on $\P_V$, compatible with its zero and infinity divisors as in \eqref{nu.rel}. Thus $\nu$ induces a RT perturbation on $X[m]$, whose restriction to $X\subset X[m]$ is equal to $\nu$, and the restriction to each copy $\P_V$ of $X[m]$ is equal to $\al_V\boxtimes \zeta$. 

\begin{rem}\label{R.defn.J.n} 
Note that a map $f:C\ra \Si[m]$ to the building is $(J,\nu)$-holomorphic (with respect to this lift of $\nu$ to $\Si[m]$) if and only if
\begin{enumerate}[(i)] 
\item the projection $\rho:=p(f): C \ra \Si$ is $(J, \nu)$-holomorphic
\item each rubber part $f_r:=f|_{C_r}:C_r\ra \P_V$ of $f$, regarded as a section $\xi$ of the pullback normal bundle $\rho_r^*N_V$, satisfies $\del \xi =\nabla_\xi \nu$ i.e. $L^N\xi=0$; for details, see \S\ref{S.normal.compl}. 
\end{enumerate}
In fact, since $L^N$ is complex linear, then $\xi$ is a meromorphic section; thus up to scale is determined by its zero and infinity divisor. Note that if $\xi \ne 0$ then $\zeta=\xi^{-1}$ is a section of the dual bundle $\rho_r^*(N_V)^\vee$ which solves $\del \zeta =-\nabla_\zeta \nu$; the poles of $\zeta$ correspond to the zeros of $\xi$ and viceversa. 

This perspective neatly encodes the conditions describing the relative stable map compactitication constructed in \cite{ip-rel}, cf. \cite[Remark~7.7 and Definition~7.2]{ip-rel}. 
\end{rem}

The space of RT perturbations \eqref{RT.relative} easily extends to the family of targets. As before, we start with the space of 
real Ruan-Tian perturbations $\nu$ on the total space $(\F,c_\F)$ of the family \eqref{family.smoothings}. These satisfy condition \eqref{nu.real} for $X$ equal to the total space of the family $\F$. Then we restrict to the subspace 
\bear\label{RT.rel.fam}
\JV^\R(\cal F_{/\De})
\eear
of such perturbations which additionally satisfy the following conditions:
\begin{enumerate}[(a)]
\item  $\nu$ is compatible with the fibration i.e. $\pi_*\nu=0$; 
\item $\nu$ is compatible with the divisor $V$ of $\F$ i.e. the 1-jet condition \eqref{nu.rel} holds for $X$ equal to the total space of the family $\F$ and $V\subset X$ the union of the marked points of the family.
\item $\nu$ is compatible with the nodal locus i.e. the pullback $\wt \nu=\phi^*\nu$ to the normalization $\wt\Si$ takes values in $T\wt \Si\subset \phi^*T\F$ and satisfies the 1-jet condition  \eqref{nu.rel} for $X$ equal to $\wt \Si$ and $V$ equal to the preimage of the nodes of $\Si_0$. 
\end{enumerate}

Condition (a) implies that any $(J, \nu)$ holomorphic map to $\F$ projects to a holomorphic, thus (locally) constant map to $\De$, and therefore its image is contained in a fiber of $\F$ (if the domain is connected). Moreover, conditions (a)-(b) imply that the pullback $\nu_s$ of $\nu$ to $\Si_s$ is a real RT perturbation on $\Si_s$, which is compatible with the divisor $V_s=V\cap \Si_s$. Similarly, the pullback $\wt \nu$ of $\nu$ to the normalization $\wt \Si$ is a real RT perturbation on $\wt \Si$ compatible with the divisor $\wt V$ (the inverse image of the special points of $\Si_0$). 
This means that the maps $\iota_s$ and $\phi$ in  \eqref{si.to.F} induce maps
\begin{align}
\label{nu.pullsback}
\JV^\R(\cal F_{/\De})& \ra \JV^\R(\Si_s, V_s)&  \JV^\R(\cal F_{/\De})& \ra \JV^\R(\wt \Si, \wt V)
\\
\nonumber
\nu&\mapsto \nu_s& \nu&\mapsto \wt \nu
\end{align}  
at the level of parameter spaces. 
\smallskip

Note that the family $\F$  has a specific local model  \eqref{loc.coord.C.sm.node} in a neighborhood of the nodes of $\Si_0$, and is a product away from a smaller neighborhood. Thus in \eqref{RT.rel.fam}, condition (a) implies (c) away from $\U\ti \{nodes\}$. A standard calculation shows that 
(a) and (c) imply that for every node $q$ of $\Si_0$, 
\bear\label{cond.node.new}
\nu|_{\U \ti q}=0 \quad \mbox{ and }  \quad \al_{q_1}+\al_{q_2}=0.
\eear
Here $q_i$ are the lifts of $q$ to $\wt \Si$, and $\al_{q_i}\in \Hom^{01}(\T_\U; \cx)$ over $\U\ti q$ is determined by $\wt \nu$ via \eqref{nu.rel.2}.  
\begin{rem}\label{R.nu.node} Property \eqref{cond.node.new} implies that the lift of $\nu$ to the building $\Si_0[m]$ is well defined, i.e. its value on $\P_q$ is independent of the choice of lift $q_i\in \wt \Si$ of the node $q$ of $\Si_0$. Note that for each node $q$ of $\Si_0$ we get two copies $\P_{q_1}$ and $\P_{q_2}$ in $\P_{\wt V}$ and there is an identification 
$\P_{q_1}\cong \P_{q_2}$ swapping the zero and infinity divisors,  induced by $(q_1,z)\mapsto (q_2, z^{-1})$. For a map $\rho:C \ra  \{ nodes\}\subset \Si_0$, each lift $\rho_i$ to $\wt \Si$ has an associated normal operator $L^N$. Property \eqref{cond.node.new} implies that the condition $\ker L^N\ne 0$ is well defined, independent of the choice of lift of $\rho$ to the normalization $\wt \Si$, cf. Remark \ref{R.defn.J.n}.
\end{rem}

\subsection{Transversality for each stratum}\label{S.VFC.rel}  Since our targets are curves, a generic real RT perturbation 
$\nu$ compatible with the divisor ensures that all strata are of expected dimension and therefore as in \cite{ip-thin}, the real relative moduli spaces carry a VFC for all parameters, including $\nu=0$. For completeness, we provide the details of this argument here. 
   \medskip
   
Assume first $\Si$ is a smooth symmetric curve with conjugate marked points $V$ as in \eqref{marked.points} and a fixed integrable complex structure $J$. As $\nu$ varies in the space of RT perturbations  $\JV^\R(\Si, V)$ defined in \eqref{RT.relative}, we get a family 
 \bear\label{rel.moduli.jv.all} 
 \begin{tikzcd}  
 	\ov \fM^\R_{d, \chi, \vec \mu} (\Si, V)\ar[r, "\pi"]&\JV^\R(\Si, V)
 \end{tikzcd}  
 \eear
of real relative moduli spaces, one for each choice of topological data $d$, $\chi$, $\vec \mu$  and we denote by 
$\ov \fM^\R_{d, \chi, \vec \mu} (\Si, V)_\nu$ its fiber at $\nu$; it can be regarded as a deformation of the moduli space $\ov \M^\R_{d, \chi, \vec \mu} (\Si, V)$ in  \eqref{rel.order.all}, which corresponds to $\nu=0$. Specifically, an element of the moduli space \eqref{rel.moduli.jv.all} is a pair $([f], \nu)$ where  
 \bear\label{f.to.Si}
 f:C\ra \Si[m]
 \eear 
 is a (relatively stable) real $(J, \nu)$-holomorphic map satisfying the properties {(i)-(v)} listed in Definition~\ref{D.moduli.rel.or}.

 As in \S\ref{S.rubber}, the moduli space \eqref{rel.moduli.jv.all} is stratified with strata $\M_\tau$ indexed by the decorated dual graph $\tau$ 
 of  the maps \eqref{f.to.Si}; therefore, along a stratum, the topological type of the domain and target is fixed, as is the ramification pattern of $f$ over the special points of the target. Note that the preimage of the nodes/marked points of $\Si[m]$ consists only of nodes/marked points of $C$. Denote by $\bf x$ the collection of special points of the target $\Si[m]$, and by ${\bf y}=f^{-1}(\bf x)$ its preimage. An element $y$ of $\bf y$ is called a contact point, and comes decorated by the contact multiplicity $\la(y)$ of $f$ at $y$, cf. \eqref{lead.coef.f}. The other points of the domain are called ordinary points; they are mapped away from the special points of the target, and their contact multiplicity is 0.  Note that some of the nodes could be ordinary ones, but all the marked points are contact points.
 
  \medskip
  
Fix a stratum $\M_\tau$ of the moduli space.  Restrict first to the case $m=0$, i.e. the target $\Si[m]=\Si$ is smooth; here the domain could be nodal,  but all its nodes are ordinary (mapped away from $V$). When $f:C \ra \Si$ has contact order $\la$ at the points ${\bf y}=f^{-1}(V)$, the linearization (to this moduli problem) is 
\begin{equation}
\label{lin.f.2} 
\begin{split}
&L_f:  \Gamma_{\la; \bf y} (f^* T\Si)^\R \oplus  T_C\ov\M^\R_{\chi, \ell} \ra \Gamma_{\la-1; \bf y}(\Omega^{01}_C\otimes_\cx f^*T\Si)^\R
\\ 
&L_f (\xi, h)= \del \xi- [ \nabla_\xi \nu + \tfrac 12 J df  h]^{01},
\end{split}
\end{equation}
cf. \eqref{L.f.A.non} and \S\ref{S.lin.smooth}. Here $\Gamma_{\la;\bf y} (E)^\R$ denotes the subspace of invariant 
sections of $E$ which vanish to order $\la(y)$ at $y$, for all $y$ in ${\bf y}$,  
and $\ov\M^\R_{\chi, \ell}$ denotes the real Deligne-Mumford moduli space containing $C$. Here for simplicity we assume $C$ is stable, but as in \cite[\S4.2]{ip-gv}, one can always use a "local slice"  \eqref{tau.tauC.A} to parametrize the variations in the domains, after first locally stabilizing the domains if necessary. Those considerations extend to our case, as long as we are working with symmetric choices throughout, by adding  pairs of conjugate marked points in the domain, mapped to fixed pairs of non-special conjugate points of the target. 
For the strata-wise linearization, the second term in \eqref{lin.f.2} is replaced by the tangent space $T_C\cal S$  to the stratum 
$\cal S$ containing $C$. 

As in \cite[\S3.1-\S3.2]{ms}, there are several choices of norms/completions one can use to locally describe a stratum of the family \eqref{rel.moduli.jv.all} as the zero locus of a Fredholm map 
\bear\label{eq.psi}
\Psi(f, \nu)= \del_{J} f -\nu|_{f} 
\eear
between Banach manifolds\footnote{the Banach manifolds used here are Hausdorff, separable, and paracompact, so that Sard-Smale theorem applies.}. 
When $\Psi$ is transverse to 0 at $f$ (i.e. $\de\Psi$ is onto) then a neighborhood of $f$ in its stratum is a Banach manifold modeled on the kernel of the strata-wise linearization. For more details on the standard set-up, we refer the reader to \cite[\S3.1-\S3.2]{ms}; see also \cite[\S4.2]{ip-gv} for a summary of some analytic preliminaries.

\smallskip

After completing \eqref{lin.f.2} in the $\la$-weighted Sobolev norms \eqref{wSob.maps.la.B}, the strata-wise linearization (to this moduli problem) at $f:C\ra \Si$ becomes the Fredholm operator 
\begin{equation}
\label{lin.f.stratum}
\begin{split}
&\L_{f}: \E_f^{k, p} \oplus  T_C{\cal S}\ra \fF_f^{k-1,p}
\\
&\L_f(\xi, h)=\del \xi- [ \nabla_\xi \nu + \tfrac 12 J df h]^{01}
\end{split}
\end{equation}
cf. \eqref{lin.rel.norms.weighted}. While the weighted completions $\E_f^{k, p}$ and $\fF_f^{k-1, p}$ depend on the choice of 
$k, p \ge 1$ (where $kp \ge 2$) the kernel and cokernel of $\L_f$ are independent of these choices, and consist of smooth elements. The full linearization $\de\Psi_f$, which also includes the variation $\mu=\de \nu$ in the parameter $\nu$, is given by 
\begin{equation}\label{lin.ful.lin}
(\xi, h, \mu)\mapsto \L_f(\xi, h)- \mu|_f.
\end{equation}
\smallskip

When $m\ne 0$, $(J, \nu)$-holomorphic maps $f:C \ra \Si[m]$ satisfying the properties listed in Definition~\ref{D.moduli.rel.or} can be analyzed
 as in Remark \ref{R.defn.J.n} by considering their projection $p(f):C \ra \Si$ under the collapsing map $\Si[m]\ra \Si$ in the target. The rubber components of $f$ continue  to project to constant maps to $V$ (since $V$ is 0 dimensional), and satisfy the condition 
\bear\label{ker.nor.not.0}
\ker L^N \ne 0  
\eear 
along this projection, cf. \cite[(6.3) and (7.1)]{ip-rel}. Here $L^N$ is the normal operator, given by 
\bear\label{D.N}
L^N\xi= \del \xi-  \nabla_\xi \nu, 
\eear 
see \eqref{xi.not.0.A}-\eqref{xi.in.ker.A}. When the index of $L^N$ is negative, condition \eqref{ker.nor.not.0} imposes additional restrictions on the rubber components as in \cite[Lemma~6.4]{ip-rel}. 
\medskip
 
\begin{rem}\label{R.unstable}  Assume $\Si$ is connected and $V\ne \emptyset$. If $f:C\ra \Si[m]$ satisfies conditions (i)-(iii) of Definition~\ref{D.moduli.rel.or}, then every connected component of its domain  has at least one marked point (the preimage of $V[m]$). Thus the unstable domain components of a map in the moduli space \eqref{rel.moduli.jv.all} have genus 0; see also Remarks~\ref{R.rubber.comp.pts} and \ref{R.stable.vs.rel}. One can also show that when $(\Si, V)$ is a {\em stable} curve, the trivial rubber components are the only unstable domain components  of the maps in the moduli space \eqref{rel.moduli.jv.all}.  
\end{rem}

Below we assume for simplicity that $\Si$ is connected and $V\ne \emptyset$, or more generally that every rational component of $\Si$ has at least one marked point; otherwise, see Remark~\ref{R.unstable.tori} below. 
  
Standard arguments  (cf. \cite[\S3.2]{ms}) imply the following result whose proof we sketch for completeness.
\begin{lemma}\label{L.transv.rel}  Assume that every rational component of $\Si$ has at least one marked point. Then over the parameter space $\JV^\R(\Si, V)$ of real RT-perturbations on $\Si$ compatible with $V,$ every stratum of the real relative moduli space \eqref{rel.moduli.jv.all} is cut transversally. 
\end{lemma}
\begin{proof} As in the proof of \cite[Proposition~3.2.1]{ms}, transversality follows provided we have enough variations in the parameters $\nu$ to ensure that the linearization of the equations cutting out each stratum of the moduli space has trivial cokernel. 

Assume $f:C\ra \Si[m]$ is a real $(J,\nu)$-holomorphic map as in \eqref{f.to.Si} for some  $\nu\in \JV^\R(\Si, V)$. 
Decompose its domain $C$ into stable and unstable components. When the divisor $V$ is non-empty, also decompose $f$ into rubber components (collapsed to points in $V$ under the projection $p:\Si[m]\ra \Si$) and non-rubber ones (level zero components), and consider the projection of $f$ to $\Si$; see also \eqref{lift.f.A}-\eqref{f.decomp.2}. 
 
By definition, the Ruan-Tian perturbations $\nu$ used here are pulled back from $\U \ti \Si$ and $\nu$ must vanish along $ \U \ti V$. Therefore such perturbations identically vanish on both unstable domain components and on the projection to $V$ of the rubber components, so these must be treated separately. To obtain transversality on the stable components, we use the fact that the restriction to these components defines an embedding into $\U\ti \Si$.  Transversality then follows by using variations either in $\nu$ (on the non-rubber components) or else in $\nabla \nu$ (on the rubber components) which have prescribed values at a suitable collection of points; we just need to ensure that the variations can be chosen so that they are tangent to the parameter space $\JV^\R(\Si, V)$, i.e. satisfy conditions \eqref{nu.real} and \eqref{nu.rel}. 

\smallskip

Specifically, decompose $f$ as in  \eqref{f.decomp} into the non-rubber part $f_0:C_0\ra \Si$ and rubber parts $f_r:C_r \ra \P_V$ and consider the projection 
\bear\label{3.p.r}
\rho:=p(f):C \ra \Si
\eear 
of $f$. 
Then the restriction of $\rho$ to $C_0$ is $f_0$ while its restriction $\rho_r$ to $C_r$ is a map to $V$. Moreover, 
\begin{enumerate} 
\item $\rho:C\ra \Si$ is $(J,\nu)$-holomorphic 
\item the restriction $\rho_r:C_r\ra V$ is a real map which satisfies the condition \eqref{ker.nor.not.0}, cf. \eqref{xi.not.0.A}-\eqref{xi.in.ker.A}. 
\end{enumerate}
We will first consider the conditions on the non-rubber part and then on each rubber part, and show they are cut transversally. 

\medskip

\non \textsc{Step 1. (Non-rubber components)} Consider first the conditions on the non-rubber part of $f$, i.e. that $f_0: C_0 \ra \Si$ is $(J, \nu)$ holomorphic (here $C_0\subseteq C$ may be nodal). As in the proof of \cite[Proposition~3.2.1]{ms}, the surjectivity of the linearization $\de\Psi_{f_0}$ fails only if we can find a nonzero $\eta\in (\fF^{0,p}_{f_0})^\vee$ such that 
\bear\label{D.trans.eta}
\int_{C_0} \lg \L_{f_0}\xi,  \eta\rg=0 \quad \mbox{ and }\quad \int_{C_0} \lg \mu, \eta\rg=0 
\eear
for all  $\xi\in \E_{f_0}^{1,p}$ and all variations $\mu$ in the parameter $\nu$. But then $\eta\in \cok \L_{f_0}$, thus by elliptic regularity $\eta$ is continuous on each component of $C_0$. 
We next show that the restriction of $\eta$ must vanish on every component of $C_0$. 
\medskip 

\non\textsc{Step 1a. (Non-rubber, stable components)}  Consider first the restriction of $f$ to a stable, non-rubber component. Assume that the restriction of $\eta$ to this component is nonzero. Then we can find a point $x$ on it where $\eta(x)\ne 0$. By the continuity of $\eta$, we can assume that 
\begin{enumerate}[(i)]
\item $x\in \U$ is not real and not special 
\item $f(x)$ is not in $V$ (because the image of a non-rubber component cannot lie entirely in $V$).
\end{enumerate}
Regard $\eta(x)$ as an element of $\Hom^{01}(T_xC, T_{f(x)}\Si)$. Then as in the standard transversality proof \cite[Proposition~3.2.1]{ms} (but after symmetrizing), we can find a symmetric  variation $\mu$ in $\nu$ supported sufficiently close to the image of $x$ (and its conjugate) such that the values of $\mu$ and $\eta$ agree at $x$ (and therefore also at $c(x)$), i.e.  
\bear\label{mu=eta} 
\mu|_{(x, f(x))} = \eta(x). 
\eear
By construction, such variation $\mu$ satisfies condition \eqref{nu.real}. Since $f(x)\in X\setminus V$, it also satisfies \eqref{nu.rel}  whenever $\mu$ is supported sufficiently close to $(x, f(x))$ and its conjugate. Furthermore, if $\beta_\ep$ is a symmetric bump function 
supported near $(x, f(x))$ and its conjugate, then $\beta_\ep \mu$ is also a variation in $\nu$ satisfying \eqref{mu=eta} and whose support is arbitrarily close to $x$ (and its conjugate). But $\eta$ is continuous and $\eta(x) \ne 0$, thus for sufficiently small $\ep$, 
\best
\quad \int_{C_0} \lg \beta_\ep\mu, \eta\rg \ne 0
\eest
contradicting \eqref{D.trans.eta}. Therefore $\eta$ vanishes on all the stable, non-rubber components.

 \medskip
 \non\textsc{Step 1b. (Non-rubber, unstable components)} Consider next the restriction of $f:C\ra \Si[m]$ to the union $C_{0u}\subseteq C$ of all the non-rubber but unstable domain components. This is a $J$-holomorphic map $f_{0u}:C_{0u} \ra \Si$ since the perturbation $\nu$ vanishes on these components. Consider the restriction $f_i$ of $f$ to a connected component $C_i$ of $C_{0u}$ (here $C_i\subseteq C$ may be nodal). Then the preimage of $V\subset \Si$ consists of finitely many points  ${\bf y}_i$ of $C_i$, all of them special points of $C$.  Moreover, $f_i$ is a stable map (since $f$ is relatively stable), thus has positive degree. Assume for simplicity that $\Si$ is connected. Then  either (i) $C_i$ has genus 0, 
 and ${\bf y}_i$ and $V$ have at most 2 points or else (ii) $C_i$ has genus 1 and $V=\emptyset$; see also Remark~\ref{R.unstable}.
  
Case (i). This situation does not occur when the target $\Si$ satisfies $g(\Si)\ge 1$ since such $\Si$ is aspherical.  So the only possibility left is when $V\subset \Si=S^2$ consists of at most one pair of conjugate points. Then the restriction of $\eta$ is in the cokernel of a holomorphic $\del$-operator  on this component $C_i$ (since $\nabla \nu$ vanishes on unstable domain components). But $C_i$ is a genus zero curve thus $H^1(C_i, f_i^* \T_\Si)$  vanishes by Serre duality since $c_1(\T_\Si)=\chi(\Si\setminus V)\ge 0$ in this case. 

Case (ii). This situation cannot occur when  $g(\Si)\ge 2$, since there are no positive degree holomorphic maps from a genus 1 curve to a higher genus curve. Since we assumed that $\Si$ has no genus 0 components without marked points, the only possibility left is that $g(\Si)=1$ and $f_i$ is an unramified cover of a torus by a torus. But then the cokernel of the restriction of the linearization \eqref{lin.f.stratum} to $C_i$ also vanishes. 
 
Therefore $\eta$ also vanishes on all the unstable, non-rubber components.

\smallskip

Thus $\eta=0$ on all non-rubber components. This implies that $\de\Psi_{f_0}$ is surjective, which means that the conditions on the non-rubber part $f_0$ of $f$ are cut transversally. 
 \smallskip
 
\non\textsc{Step 2. (Rubber components)} Finally, consider the conditions on rubber parts of $f$ described in the paragraph containing  \eqref{3.p.r}, which we also want to show are cut transversally. 

Since $V=V^+\sqcup c(V^+)$ cf. \eqref{marked.points}, the rubber parts of $f$ come in conjugate pairs, so it suffices to restrict attention to each connected component $C_i^+$ which is mapped to $V^+$ under the projection $\Si[m]\ra \Si$. 

Consider the restriction of $f$ to such connected component $C^+_i$. Here $C_i^+\subset C$ may be nodal, but note that it is either a trivial rubber component or else it is stable cf. Remark \ref{R.rubber.comp.pts}. Either way, the pullback of $\nu$ identically vanishes along this component since it projects to a point in $V$. The restriction $\rho_i$ of $\rho=p(f)$ to $C_i^+$ is a constant map to $V^+$ which satisfies the additional condition
\bear\label{DN.not.0.5}
\ker L^N_{\rho_i}\ne 0.
\eear 
Here $L^N_{\rho_i}=\del - \nabla \nu$ is a holomorphic $\del$-operator on the trivial complex 
line bundle $\rho_i^*N_V$ over $C^+_i$; its complex index is equal to $1-g(C^+_i)$ cf. \eqref{A.ind.LN}.  In particular, \eqref{DN.not.0.5} is automatic on the genus zero components $C_i^+$, so we may assume that $C_i^+$ is stable, has genus at least 1, and $\ind L^N_{\rho_i} \le 0$.   
\medskip

To check that the subset of such pairs $(\rho, \nu)$ which satisfy the condition \eqref{DN.not.0.5} is cut transversally, we similarly adapt the proof of \cite[Lemma~6.4]{ip-rel} to the real setting; see also the proof of \cite[Proposition~5.3]{ip-gv}. Regard  the normal operator as a section $(\rho, \nu)\mapsto L^N_\rho$ of a bundle $\rm{Fred}\ra \cal B$ of {\em complex} Fredholm operators. Here $\cal B$ is the space of pairs $(\rho, \nu)$ where $\nu\in \JV^\R(\Si, V)$, and $\rho:C_i^+\ra V^+$ is a constant map on $C_i^+\subseteq C$. Since $C_i^+$ is connected, then by the last sentence of Remark~\ref{R.xi.nor.A} it suffices to prove that this section is transverse to 
\best
\rm{Fred}^1=\{ D\in \Fred\; | \; \text{dim}_\cx \ker D=1\}. 
\eest
The fiber at $D$ of the normal bundle to $\rm{Fred}^1$ is $\Hom (\ker D; \cok D)$. 

Let $0\ne \kappa\in  \ker L^N$ and $\eta\ne 0$ an element of the cokernel of $L^N$. By unique continuation, we can find a point $x$ on the domain so that 
\begin{enumerate}[(i)]
\item both $\kappa(x)$ and $\eta(x)$ are nonzero
\item $x\in \U$ is not real and not special, but of course $x$ must be mapped by $\rho$ to $V$. 
\end{enumerate}
It therefore suffices to construct a variation $\mu$ in $\nu$ such that 
\bear\label{var.nabla.pos}
\int_{C^+_i} \lg (\de_\mu L^N)(\kappa), \; \eta\rg =- \int_{C^+_i} \lg \nabla_\kappa \mu, \; \eta\rg \ne 0. 
\eear
Here $\de_\mu L^N= - \nabla \mu$ is the variation in the normal operator as we vary $\nu$ but fix the map $\rho:C^+_i\ra V$. Since $C_i^+$ is stable, then as in the proof of  \cite[Lemma~6.4]{ip-rel}, after symmetrizing, we can find a variation $\mu$ in $\nu$ compatible with $V$, supported near the image of $x$ (and its conjugate) such that the values of $\nabla_{\kappa}\mu$ and $\eta$ agree at $x$ (and therefore also at its conjugate), i.e. 
\bear\label{mu=nabla.eta} 
(\nabla_{\kappa}\mu)|_{(x, p(f)(x))}= \eta(x). 
\eear
This variation $\mu$ then satisfies both conditions \eqref{nu.real} and \eqref{nu.rel}. The result then follows as before: multiplying $\mu$ by (symmetric) bump functions $\beta_\ep$ supported around the image of $x$ (and its conjugate) in $\U \ti \Si$ gives a sequence of variations $\beta_\ep\mu$, still satisfying \eqref{nu.real}, \eqref{nu.rel} and \eqref{mu=nabla.eta}, but whose support is arbitrarily close to $x$ (and its conjugate). This shows that \eqref{var.nabla.pos} is satisfied.
\end{proof} 

\begin{rem}\label{R.codim.strata} The virtual dimension of a stratum is the index of the strata-wise linearization. For the real relative moduli space $\ov \M^\R_{d, \chi, \vec \mu} (\Si, V)$ or more generally for \eqref{rel.moduli.jv.all}, the virtual (co)dimensions of its strata can be calculated as in \cite[\S7]{ip-rel}, see also Remark~\ref{R.dim.counts}. In particular, a stratum consisting of maps $f:C \ra \Si$ (i.e. without rubber components) has virtual codimension $l$, where $l$ is the number of (ordinary) nodes of the domain. A stratum consisting of maps $f:C \ra \Si[m]$ has virtual codimension (at least) $l+2|m|$.
\end{rem}
 
 \begin{rem}\label{R.unstable.tori} In the proof above, the assumption that every rational component of $\Si$ has at least one marked point is only used in Step 1B case (ii). When $V=\emptyset$ and $\Si$ has genus 0 components, there are a few strata of the (absolute) moduli space $\ov \fM^\R_{d, \chi} (\Si, \emptyset)$ in \eqref{rel.moduli.jv.all} which are not cut transversally. These strata consist of stable maps $f:C\ra \Si$ whose domain has a connected component $C_i$ which is a nodal genus 1 curve without any marked points. The restriction of $f$ to $C_i$ is holomorphic since RT perturbations vanish on such components. If $C_i$ contains a torus on which $f$ is constant, then the cokernel of the linearization  \eqref{lin.f.stratum} on $C_i$ is 1-dimensional. All the other strata are cut transversally.
 \end{rem}

\subsection{Transversality for a family of targets}\label{S.transv.fam} Consider next a family $\F=\cup_s \Si_s$ of targets as in \eqref{family.smoothings}, and let $\cP=\JV^\R(\cal F_{/\De})$ be the space of real RT perturbations defined in \eqref{RT.rel.fam}. For the rest of the paper, we fix the topological data $d, \chi, \vec \mu$ and denote by 
\bear\label{mo.over}
 \ov\fM(\Si_s)= \ma\cup_\nu \ov\fM^\R_{d, \chi,\mu}(\Si_s, V_s)_ \nu \ra \cP
\eear
the family of real relative moduli spaces defined as in \eqref{rel.moduli.jv.all} for $\Si=\Si_s$ but for parameters $\nu \in \cP$. Note that the fiber over 
$\nu\in \cP$ depends only on the pullback $\nu_s$ to $\Si_s$, cf. \eqref{nu.pullsback}.  We also fix $s_0\ne 0$ in $\De$ and denote by $I$ the segment $[0, s_0]\subset \De$. 

The families in \eqref{mod.fam.la} and \eqref{map.attach.0} similarly extend to families 
\bear\label{mo.over.fam}
\ov\fM(\F_{/I})=\ma\cup_{s\in I}  \ov\fM(\Si_s) \quad \mbox{ and } \quad  \ov\fM(\wt \Si)= \ma\sqcup_{\la \vdash d} \ov \fM_\la(\wt \Si), 
\eear
over $\cP$. Here $\wt \Si$ is the normalization of the nodal fiber $\Si_0$ and the fiber of 
$ \ov\fM_\la(\wt \Si)$ at $\nu\in \cP$ is by definition 
\bear\label{4.fam.M.res}
\ov \fM_\la (\wt \Si)_\nu= \ov \fM^\R_{d, \chi+4\ell(\la), \vec \mu,\la, \la}(\wt \Si, \wt V)_ \nu. 
\eear

\begin{rem}\label{R.m.0.same} Recall that elements $f$ in $\ov \fM (\Si_0)$ have the property that the preimage of the nodes of the target is a subset of the nodes of the domain, cf.  \S\ref{S.nod.target} and Definition~\ref{D.moduli.rel.or}(ii). When $f$ has no rubber components, let $\wt C$ be the (partial) normalization of its domain at these contact nodes, and let $\wt f:\wt C\ra \wt \Si$ be the lift of $f$ to the normalizations. If $f:C_0 \ra \Si_0$ is a real map satisfying the conditions (i)-(v) in Definition~\ref{D.moduli.rel.or}, then its lift $\wt f:\wt C \ra \wt \Si$ is also a real map satisfying essentially the same conditions, the only difference being that each contact node $y$ has been replaced by two points $y_1, y_2$ at which $\wt f$ has the same contact multiplicity $\la(y)$. Thus {\em locally} there is no difference between the conditions describing a stratum of $\ov \fM (\Si_0)$ near $f:C_0 \ra \Si_0$ and a stratum of $ \ov \fM(\wt \Si)$ near $\wt f: \wt C \ra \wt \Si$. The same is true when $f$ has rubber components. 
\end{rem} 

Below we assume that $\Si_0$ has no connected component which has genus 0 but has no marked points on it.  Then the normalization $\wt \Si$ of $\Si_0$ and the smooth deformations $\Si_s$ of $\Si_0$ also have this property. 

The proof of Lemma~\ref{L.transv.rel} extends essentially verbatim, after replacing $\Si$ by the family $\F$, to give the following result. 
\smallskip

\begin{lemma}~\label{L.transv.rel.F} Assume that every rational connected component of $\Si_0$ has at least one marked point.  Then over the space $\JV^\R(\cal F_{/\De})$ of real RT perturbations on the family, every stratum of the following families of real relative moduli spaces 
\bear\label{mod.to.consider}
\ma\cup_{s\in I} \ov\fM(\Si_s), \quad  \ov\fM(\Si_{s_0}),    \quad \ov\fM(\wt \Si), \quad \mbox{ and } 
\quad \ov\fM(\Si_0)  
\eear
are cut transversally. 
\end{lemma}
\begin{proof} A stratum of each one of these families consists of maps $f$ whose domain and target are symmetric surfaces with fixed topological type. It suffices to check that the variations in $\nu$ constructed  in the proof of Lemma~\ref{L.transv.rel} can be chosen tangent to the space of perturbations  
$\cP=\JV^\R(\cal F_{/\De})$ on the family. This means they can be defined on $\U \ti \F$ and chosen so that they satisfy (a) $\pi_*\nu=0$,  (b) condition \eqref{nu.real} for $X=\F$, and (c) condition \eqref{nu.rel} for $X=\F$ and $V\subset X$ the union of the marked divisor of $\F$ and the nodes of $\Si_0$.

Let $f$ be an element of any one of the moduli spaces \eqref{mod.to.consider}. Its domain is a real (possibly nodal) marked curve $C$ and its target is a building $\Si[m]$ as in \eqref{f.to.Si}, where $\Si$ is either some fiber $\Si_s$ of $\F$ for $s\in I$ or else is the normalization $\wt \Si$ of the nodal fiber $\Si_0$. In turn, both $\Si_s$ and $\wt \Si$ map to $\F$ cf. \eqref{si.to.F}. In particular, $f$ descends to a map $p(f)$ to $\F$, such that its rubber components are mapped to the marked points of $\Si_s$ or else to the marked points and nodes of $\Si_0$. 

Moreover, locally every stratum of $\ov \fM(\Si_0)$ can be described by the same local defining equations as those of a stratum of $\ov\fM (\wt \Si)$ cf. Remark~\ref{R.m.0.same}. 

Thus it suffices to show that the strata of the moduli spaces $\ov\fM(\wt \Si)$ and $\ov\fM(\Si_{s\ne 0})$ are cut transversally over the parameter space $\cal P$; for these moduli spaces, the rest of the proof proceeds as in that of Lemma~\ref{L.transv.rel}. 

We again separately consider the condition that $p(f):C\ra \Si_s \subset \F$ be $(J, \nu)$-holomorphic, and the additional 
condition $\ker L^N\ne 0$ on the rubber components. Unstable domain components continue to be holomorphic, and must be treated separately exactly as in the proof of Lemma~\ref{L.transv.rel}.
 
\medskip

\non\textsc{Step 1 (Non-rubber components).} Here it suffices to construct a variation $\mu$ tangent to $\cP$ satisfying \eqref{mu=eta} at a point $x$ on the domain such that 
\begin{enumerate}[(i)]
\item $x\in \U$ is not real and not special
\item the image $f(x) \in \Si_s\subset \F$ is  not a marked point nor a node.  
\end{enumerate}
For that, we use the fact that both $\U $ and $\F$ are locally trivial fibrations in a sufficiently small neighborhood of such a point $(x, f(x))$ (and its conjugate), thus the variation $\mu$ constructed on a single fiber of $\F$ as in Lemma~\ref{L.transv.rel} extends to the family. 
\medskip

 \non\textsc{Step 2 (Rubber components)} For each rubber component $C'$ of the domain, denote by 
 $\rho$ the restriction of $p(f)$ to $C'$. Then $\rho$ maps $C'$ to a single point of $\F$;  denote it $q \in \Si_s\subset \F$. Note that $q$ must be a special point of $\Si_s$. We want to show that the condition on $(\rho, \nu)$ that $\ker L^N\ne 0$ is cut transversally. Here we separate the case when 
 \begin{enumerate}[(a)] 
 \item $q$ is a marked point of the family; then $L^N_\rho$ is an operator on $\rho^*N_q=\rho^* T_q\Si_s$.
 \item $q$ is a node of $\Si_0$;  then $q$ has two lifts $q_1, q_2$ to the normalization. For each one, we get a normal operator 
 $L^N_{\rho_i}$ on $\rho^*N_{q_i}=\rho^* T_{q_i}\wt \Si$. However, the condition $\ker L^N_{\rho_i}\ne 0$ is independent of the choice of lift, cf. Remark \ref{R.nu.node}.
 \end{enumerate}
 \smallskip
 
 \non \textsc{Step 2a.} Assume $q$ is a marked point of $\Si_s$. As in the paragraph containing \eqref{var.nabla.pos}, it 
 suffices to ensure that there exists a variation $\mu$ tangent to $\cP$ satisfying  \eqref{mu=nabla.eta} at a non-real, non-special point $x\in C'$ where $\kappa(x) \ne 0$ and $\eta(x)\ne 0$. But at a marked point $q$ of the family $\F$, 
 $\pi$ is a local fibration in a sufficiently small neighborhood of the marked point, thus the variation $\mu$ constructed on a single fiber of $\F$ as in Lemma~\ref{L.transv.rel} extends to the family.
 \medskip
 
\non   \textsc{Step 2b.} Assume $q$ is a node of $\Si_0$, and chose a lift $q_1\in \wt \Si$ to the normalization. As in the proof of Lemma~\ref{L.transv.rel}, we have $\kappa\in \ker L^N$ and $\eta\in \cok L^N$ such that both $\kappa(x)$ and $\eta(x)$ are nonzero, cf. (i)-(ii) above \eqref{var.nabla.pos}. Here $x$ is a fixed  nonspecial, nonreal point of $C$, $\kappa(x)\in T_{q_1}\wt\Si\hookrightarrow T_{q}\F$ while $\eta(x) \in  \Hom^{01}(T_xC, T_{q_1}\wt\Si)$. It suffices to check that we can construct a variation $\mu$ tangent to $\cP$ such that \eqref{mu=nabla.eta} holds, 
i.e. $(\nabla _\kappa \mu)|_{(x,q)}=\eta(x)$. 

Using the local model 
\eqref {loc.coord.C.sm.node} of $\F$ around the node $q$, we can construct a variation $\mu$ defined in a neighborhood of the point $(x, q)$ in  $\U \ti \F$ as follows. Identify a neighborhood of $q$ in $\F$ with a neighborhood of $0$ in  $\cx^2$ as in the paragraph containing \eqref {loc.coord.C.sm.node}. Let $\mu\in \Hom^{01}(\T_\U, T\F)$ be given by $\mu_{(w, z)}=(z_1\al_1(w),z_2 \al_2(w))$ for all $w\in \U$ close to $x$ and $z=(z_1, z_2)\in \cx^2$ small.  Here $\al_i \in \Hom^{01}(\T_\U, \cx)$ and $\al_1+\al_2=0$. This ensures that $\mu$ satisfies conditions (a)-(c) below \eqref{RT.rel.fam}; see also  \eqref{cond.node.new}. Moreover, since $\kappa(x), \eta(x)\ne 0$ we can arrange that $\al_1(x)\otimes \kappa(x)= \eta(x)$ at the point $x$, thus $(\nabla _\kappa \mu)|_{(x,q)}=\eta(x)$. Multiplying by a bump function and symmetrizing gives  the desired variation tangent to $\cP$.
\end{proof} 
Combining the Sard-Smale Theorem with Lemma~\ref {L.transv.rel.F} gives the following result.
\begin{cor}\label{C.baire} There exists a Baire subset $\cP^*\subset \cP=\JV^\R(\cal F_{/\De})$ such that 
for all $\nu\in \cP^*$, all the strata of the fibers over $\nu$ of all the moduli spaces  \eqref{mod.to.consider} are cut transversally, thus are smooth manifolds of the expected dimension. 
\end{cor}
By a \textsf{generic parameter} we mean an element of a Baire subset of the parameter space (a Baire subset is a countable intersection of open and dense sets).

\medskip
\section{VFCs for real relative moduli spaces}\label{S.VFC.mod} 

\medskip

In this section we use the transversality results in \S\ref{S.VFC.rel} and \S\ref{S.transv.fam} to show that the moduli spaces considered there are thin compactifications (as briefly reviewed in \S\ref{S.thin.cpt}) 
or more generally thinly compactified families in the sense of \cite[\S3]{ip-thin}. This implies that the moduli spaces carry a virtual fundamental class (VFC) in rational Cech homology, cf. \cite[\S4]{ip-thin}.
\smallskip

\subsection{Thinly compactified families and their VFC} First we recall the general results of \cite[\S4]{ip-thin}. 
Roughly speaking, a proper map
\bear\label{famMP} 
	\ov\fM\ra \cP 
\eear
is a thinly compactified family over $\cP$ provided (i) for generic $\nu\in \cP$ the fiber $\ov\fM_\nu$ is a thin compactification, i.e. an oriented topological manifold away from a (homologically) codimension 2 strata  and (ii) over a generic path $\gamma$ in $\cP$, $\ov \fM_\gamma$ is a thinly compactified cobordism, i.e. an oriented cobordism away from a (homologically) codimension 2 strata. See \cite[Defn~2.1 and 3.1]{ip-thin} and \cite[\S2.4]{ip-thin} for the precise details.

Since the fundamental class (in rational Cech homology) of a manifold uniquely extends to any thin compactification by \cite[Thm~2.4]{ip-thin},  condition (i) above defines the fundamental class of the moduli space for generic parameter $\nu$. For non-generic parameters, the virtual fundamental class (VFC) is obtained as the limit (in rational Cech homology) of nearby classes, which is well defined by condition (ii). In particular, as in \cite[Thm~4.2]{ip-thin} any thinly compactified family $\ov \fM\ra \cP$ carries a unique VFC\footnote{called the relative fundamental class in \cite{ip-thin}.}
\bear\label{VFC.gen}
[ \ov\fM_\nu]^\vir \in\cHH_* (\ov \fM_\nu; \Q)  
\eear
defined for all $\nu\in \cP$, and satisfying a consistency condition over all paths $\gamma$ in $\cP$, cf. \cite[Defn~4.1]{ip-thin}. 

Finally, thin compactifications behave well under passing to covers, cf. \cite[\S2.1]{ip-thin} or enlarging the parameter spaces cf. \cite[\S6]{ip-thin}. 
\medskip

\subsection{The VFC for real relative moduli spaces}\label{S.VFC}  We next apply these consideration to the families of moduli spaces considered in section \S\ref{S.constr.vfc}. Before we proceed, note that being a finite dimensional manifold is a topological condition, which can be verified by providing a local model; on the other hand, its orientability and a choice of orientation are global questions. This section summarizes why the transversality results proved so far suffice to 
address the first question for the generic fiber of these moduli spaces (away from virtual codimension 2 strata), and to identify their (actual) orientation sheaf with a certain determinant line bundle (away from virtual codimension 2 strata). Orientations are deferred to \S\ref{S.orient}.

\smallskip 
We start with the case of a fixed, smooth target, and then move to a family of targets. Let $\Si$ be a (smooth) symmetric curve with $r$ pairs of conjugate marked points $V$, and assume that $\Si$ has no rational components without any marked points.  For every $d$, $\chi$, $\vec \mu$ fixed, consider the family 
\bear\label{rel.moduli.jv} 
\begin{tikzcd}  
\ov \fM ^{\R}_{d, \chi, \vec \mu}(\Si, V)\ar[r]&\JV^\R(\Si, V)  
\end{tikzcd}  
\eear
of moduli spaces defined by \eqref{rel.moduli.jv.all}. Its virtual orientation line at $f$ is the determinant of the 
linearization. Recall that when $f:C \ra \Si$ (i.e. $f$ has no rubber components) the linearization $L_f $ is given by \eqref{lin.f.2} and 
$\det L_f$ is defined using its Fredholm completion \eqref{lin.rel.norms.weighted}. However, it is more convenient here to use an equivalent description of the linearization, as an operator of the form 
\begin{equation}
\label{lin.rel}
\begin{split}
&D_f: \Gamma (f^* \T)^\R \oplus  T_C\ov\M^\R_{\chi, \ell} \ra \La^{01}(f^* \T)^\R
\\ 
&D_f (\xi, h)=\del_{f^*(\T, c_\T)} \xi + A_f(\xi)+ b_f(h)
\end{split}
\end{equation}
cf. \eqref{lin.f.rel.A}. Here $\T=\T_\Si$ is the relative tangent bundle  \eqref{T.punctured} of the marked curve $\Si$. The first term of $D_f$ is the pullback of the $\del_{(\T, c_\T)}$ operator on the target, while $A_f$ and $b_f$ are 0'th order terms. By construction, the image of both $A_f$ and $b_f$ are (smooth) $(0,1)$-forms supported away from a neighborhood of the special points of the domain. Operators of the form \eqref{lin.rel} are completed to Fredholm operators as in \eqref{Real.CR}-\eqref{ker.D0}; we refer the reader to Lemma~\ref{L.all.lin.sameA} and rest of \S\ref{S.lin.smooth} for the identification between the two linearizations 
-- $L_f$ in \eqref{lin.f.2} and $D_f$ in \eqref{lin.rel} -- as well as between the kernels and cokernels of their Fredholm completions. 
\medskip

Up to 0'th order terms, the operator $D_f$ in \eqref{lin.rel} is the same as  $\del_{f^*(\T, c_\T)}\oplus 0_W$, where $W= T_C\ov\M^\R_{\chi, \ell}$ parametrizes the variations in the domain of $f$ and $0_W:W\ra 0$. 
In particular, the virtual dimension of the moduli space is the index of $D_f$, which is equal to \eqref{dim.M=b}. 
Moreover, since the virtual relative orientation sheaf $\mo_{\fM}$ of the family \eqref{rel.moduli.jv} is the determinant line bundle of the family of linearizations $D_f$, we get canonical  (up to homotopy) identifications 
	\bear\label{or.sheaf.top}
	\mo_{\fM}\cong \det D \cong \det (\del_{(\T, c_\T)}\oplus 0) = \det \del_{(\T, c_\T)} \otimes \mathfrak{f}^* \det T \oM^{\R}_{\chi, \ell(\vec \mu)}, 
	\eear
cf. \cite[(A.13)]{GI}. Here $ \det \del_{(\T, c_\T)}$ denotes the determinant of the family of pullback operators $\del_{f^*(\T, c_\T)}$ as $f$ varies in $\fM$, and 
\best
\mathfrak f: \ov \fM ^{\R}_{d, \chi, \vec \mu}(\Si, V) \ra \oM^{\R}_{\chi, \ell(\vec \mu)}
\eest 
denotes the forgetful morphism to the real Deligne-Mumford moduli space parametrizing the domains. 
\medskip

The following theorem defines the VFC of the real relative moduli spaces. 

\begin{theorem}\label{T.VFC.rel} The family of real relative moduli spaces \eqref{rel.moduli.jv} is a thinly compactified family. It therefore carries a VFC 
	\bear\label{VFC.S.V}
	[ \ov\fM^\R_{d, \chi, \vec \mu}(\Si, V)_\nu ]^\vir \in\cHH_b  (\ov\fM^\R_{d, \chi, \vec \mu}(\Si, V)_\nu; \Q)  
	\eear
	for all $\nu\in \JV^\R(\Si, V)$, including $\nu=0$, with $b$ given by \eqref{dim.M=b}. The VFC depends on the choice of (twisted) orientation data on the relative tangent bundle $\T$ of $\Si$ as described in \S\ref{S.orient}. 
\end{theorem} 
\begin{proof} Fix $d, \chi$ and $\vec \mu$. For simplicity, denote the family \eqref{rel.moduli.jv}  by $\ov \fM\ra \cP$ as in \eqref{famMP} and by $\ov \fM_\nu$ its fiber over $\nu\in \cP$. Consider the open subset 
	\bear\label{wt.M}
	\wt \fM\subseteq \ov\fM
	\eear
consisting of the union of strata of virtual codimension at most one. The elements of these strata are maps $f:C\ra \Si$ whose domains have at most 1 real node and this node must be an ordinary real node, i.e. its image is away from $V$; in particular these elements cannot have any rubber components and they are all maps into a fixed, smooth target $\Si$. All the other strata have virtual codimension at least 2 by the formula in Remark~\ref{R.codim.strata}. 

By Lemma~\ref{L.transv.rel}  and the Sard-Smale Theorem, for generic parameter $\nu \in \cP$, all the strata of the real relative moduli space are cut transversally; in particular, the cokernel of the linearization along the top stratum and along the codimension one stratum of $\wt \fM_\nu$ vanishes. Furthermore all the maps in $\wt \fM_\nu$ have the same fixed, smooth target $X=\Si$ and at most one node, which is an ordinary node. Therefore the proof of the usual gluing theorem in \cite[\S B]{pardon} or  \cite{li-tian}  applies verbatim to the moduli space $\wt \fM_\nu$ after making the following inconsequential changes: 
\begin{enumerate}[(a)]
\item working throughout with symmetric choices (of metrics, bump functions etc) as in  \cite{gz}.
\item the linearization is completed in norms which are weighted around the ordinary node 
as in \cite[\S B.4]{pardon}, and have different weights around  the contact points. But these contact points are marked points of the domain so they stay {\em away} from the node where the gluing is performed.
\end{enumerate}
Therefore as in the usual gluing theorem, the generic fiber $\wt \fM_\nu$ of \eqref{wt.M} is a $b$-dimensional topological manifold (without boundary), locally modeled on the kernel of the linearization,  cf. \cite[Theorem B.1.1]{pardon}. Moreover, the same arguments show that the determinant of the family of linearizations is locally trivial (thus forms a line bundle) as $f:C\ra \Si$ varies in $\wt \fM$, exactly as in \cite[(3.3)]{gz} but with $E=TX$ replaced in our case by $E=\T_\Si$. 

Let $\det D$ be the determinant line bundle of the family of linearizations $D_f$ as $f$ varies in $\wt \fM$, cf. \eqref{lin.rel}. As in \cite{gz}, there is a canonical homotopy class of isomorphisms
\bear\label{det.D.T}
\det D \cong \det (\del_{(\T, c_\T)} \oplus 0_W)= \det \del_{(\T, c_\T)} \otimes \mathfrak{f}^* \det T \oM^{\R}_{\chi, \ell(\vec \mu)}
\eear
as $f:C \ra \Si$ varies over $\wt \fM$, even as the domains become nodal. This identifies the virtual  (relative) orientation sheaf $\mo_{\wt{\fM}}$ of $\wt \fM$ with 
 	\bear\label{or.sheaf}
	\mo_{\wt{\fM}} \cong \det D\cong \det \del_{(\T, c_\T)} \otimes \mathfrak{f}^* \det T \oM^{\R}_{\chi, \ell(\vec \mu)}.
	\eear
In \S\ref{S.orient} below we recall how a choice of (twisted) orientation data trivializes \eqref{or.sheaf}. Note that a trivialization of $\det D$ orients the generic fibers $\wt \fM_\nu$ of the family $\wt \fM \ra \cal P$.

\smallskip

Consider next the ``singular locus"
	\bear\label{wt.S}
	\cal S=\ov\fM \setminus \wt \fM.
	\eear
By Lemma~\ref{L.transv.rel} and the Sard-Smale theorem, for generic parameter $\nu\in \cP$, its fiber $\cal S_\nu=\ov\fM_\nu\setminus \wt\fM_\nu$ is stratified by smooth manifolds of the expected dimension, which is at most $\dim \wt \fM_\nu-2$ as noted above. Then by \cite[Lemma~2.2]{ip-thin} the singular locus $\cal S_\nu$ is homologically codimension 2. Thus for generic $\nu$,  $\ov\fM_\nu$ is a thin compactification of the oriented manifold $\wt \fM_\nu$ and therefore carries a fundamental class $[\ov\fM_\nu]$. 
	\smallskip
	
Similarly, the moduli space $\ov\fM_\gamma$ over a generic path $\gamma$ is a thinly compactified cobordism in the sense of \cite[\S2.4]{ip-thin}. This means that the corresponding subset $\wt\fM_\gamma$ is an oriented topological manifold with boundary $\wt \fM_{\bd \gamma}$, and its complement in $\ov\fM_\gamma$ is stratified by codimension 2 manifolds, as is the complement of $\wt \fM_{\bd \gamma}$ in $\ov \fM_{\bd \gamma}$. 
	
\smallskip
Therefore by \cite[Lemma~5.4]{ip-thin}, $\ov\fM\ra \cP$ is a thinly compactified family in the sense of \cite[Definition~3.2]{ip-thin}. Consequently, by  \cite[Theorem~4.2]{ip-thin}, it carries a VFC as in \eqref{VFC.gen} for all $\nu\in \cP$, including $\nu=0$.  The VFC is defined by $[\ov\fM_\nu]$ for generic parameter $\nu$, and  extended uniquely to all parameters using the continuity property of rational Cech holomogy. 
\end{proof} 
\begin{rem} 
In particular, the VFC \eqref{VFC.M} of the (unperturbed) real relative moduli space $\ov\M^\R_{d, \chi, \vec \mu}(\Si, V)$ can be defined by turning on RT perturbations $\nu$ in $\JV^\R(\Si, V)$ to obtain transversality strata-wise and then turning them off by taking $\nu\ra 0$. 
\end{rem}

 We next show that the proof of Theorem~\ref{T.VFC.rel} extends to the moduli spaces 
	\bear\label{fam.moduli.consider}
	\ov \fM(\Si_{s\ne 0}), \quad \ov \fM(\wt \Si)=\ma \sqcup_{\la\vdash d}  \ov \fM_\la (\wt \Si),\quad \mbox{ and } \quad   \ov \fM(\Si_{0})
	\eear
described in the paragraph containing \eqref{mo.over}, for fixed $d, \chi, \vec \mu$ as in \S\ref{S.transv.fam}. As in \S\ref{S.transv.fam}, we continue to assume that the nodal curve $\Si_0$ has no real special points, and no rational connected components without any marked points.    Note that  \eqref{fam.moduli.consider} are families of real relative moduli spaces over the 
parameter space $\cP=\JV^\R(\cal F_{/\De})$. 
\begin{theorem}\label{T.VFC.rel.2} Assume $\Si_0$ has no rational connected components without any marked points. Then the moduli spaces \eqref{fam.moduli.consider} are thinly compactified families over $\cP=\JV^\R(\cal F_{/\De})$. In particular, each one carries a VFC 
	\best
	[ \ov\fM(- )_{\nu}]^\vir \in\cHH_b  ( \ov\fM(-)_{\nu};\; \Q) 
	\eest
for all $\nu\in \cP$. The VFC depends on the choice of (twisted) orientation data $\mo_\F$ on the relative tangent bundle $\T \ra \F$ to the family $\F$; the virtual dimension $b$ is given by \eqref{dim.M=b}. 
\end{theorem} 
\begin{proof} The strata of the moduli spaces \eqref{fam.moduli.consider} are cut transversally by Lemma~\ref{L.transv.rel.F} and every stratum with at least one rubber component has codimension at least 2, cf. Remark~\ref{R.codim.strata}.  Denote by 
 \bear\label{fam.moduli.consider.td}
	\wt \fM (\Si_{s\ne 0}), \quad \wt \fM(\wt \Si)=\ma\sqcup_\la  \wt \fM_\la (\wt \Si),  \quad \mbox{ and } \quad   \wt \fM(\Si_0)
\eear
the union of their codimension at most 1 strata. Note that besides the nodes in the preimage of the nodes of the target, the domains of maps in $\wt \fM(\Si_0)$ can have at most one additional node, which must be a real ordinary node.  Moreover, these maps have no rubber components. 

Recall that by definition the fibers of the families $\ov \fM (\Si_{s})$ and $ \ov \fM_\la (\wt \Si)$ at $\nu\in \cP$ are the real relative moduli spaces 
$\ov\fM^{\R}_{d, \chi, \vec \mu}(\Si_s, V_s)_\nu$ and respectively $\ov\fM^{\R}_{d, \chi+4\ell(\la), \vec \mu, \la, \la}(\wt \Si, \wt V)_\nu$, and that $\Si_s$ and 
$\wt \Si$ are smooth for $s\ne 0$. These moduli spaces have the same virtual dimension
\best
b=d \chi(\Si_{s}\setminus V_{s})-\chi+2\ell(\vec \mu)= d \chi(\wt \Si\setminus \wt V) - (\chi+ 4\ell(\la)) + 2\ell(\vec \mu)+4\ell (\la)
\eest
cf. \eqref{dim.M=b}; see also  \eqref{A.v.dim=b}. Moreover, the relative orientation sheaf of the family $\wt \fM (\Si_{s\ne 0})$ and respectively $ \wt \fM_\la(\wt \Si)$  over the parameter space $\cP$ continues to be given by the corresponding formula \eqref{or.sheaf} and is similarly orientable, as reviewed in \S\ref{S.orient} below. 

\textsc{Case 1.} When $\Si=\Si_{s\ne 0}$ or $\wt \Si$, the rest of the proof of Theorem~\ref{T.VFC.rel} applies verbatim to 
$\ov\fM (\Si)\ra \cal P$. 

\smallskip
 \textsc{Case 2.} When $\Si=\Si_0$, recall that every stratum of $\ov \fM(\Si_0)$ can be described by the same local defining equations as those of a stratum of $\ov\fM (\wt \Si)$, cf. Remark~\ref{R.m.0.same}. In particular, since for generic $\nu$ every stratum of $\ov\fM (\wt \Si)_\nu$ is a manifold, then so is every stratum of $\ov \fM(\Si_0)_\nu$; the same is true over a generic path.

Consider next $\wt \fM(\Si_0)$, the subset consisting of maps $f_0:C_0\ra \Si_0$ whose domain has at most one {\em ordinary} node, in addition to the contact nodes (recall that ordinary nodes are mapped away from the special points of $\Si_0$, while contact nodes are the preimage of the nodes of $\Si_0$). At the beginning of \S\ref{S.deg.phi} below we show that a finite (global) cover of $\wt \fM(\Si_0)$ (obtained by ordering the contact nodes) can be identified with $\wt \fM(\wt \Si)$. Since the generic fiber of $\wt \fM(\wt \Si)$ is a finite dimensional 
manifold, then so is the generic fiber of $\wt \fM(\Si_0)$; the same is true over a generic path.
 
 The relative orientation sheaf of $\wt \fM(\Si_0)$ is analyzed in \S\ref{S.mod.nodal} below, and is orientable as well, cf. paragraph above  Lemma~\ref{L.phi.orient}. The rest of the proof proceeds as in that of Theorem~\ref{T.VFC.rel}. 
\end{proof}

The same proof also provides the following specific result about the generic fibers of these families. Denote by $\cP^*$ the Baire subset of  $\cP=\JV^\R(\cal F_{/\De})$ appearing in Corollary~\ref{C.baire}.
\begin{prop}\label{P.or.mfl}  Assume $\Si_0$ has no rational connected components without any marked points. Then for all $\nu \in \cP^*$, we have the following properties:
\begin{enumerate}[(i)] 
\item the spaces 
\bear\label{mod.or.consider}
\ma\cup_{s\in (0, s_0)} \wt \fM(\Si_s)_\nu, \quad \wt \fM(\Si_{s_0})_\nu, \quad \wt \fM(\Si_0)_\nu, \quad \mbox{ and } \quad \wt \fM(\wt \Si)_\nu 
\eear
are orientable topological manifolds;
\item $\ov\fM(\Si)_\nu$ is a thin compactification of $\wt \fM(\Si)_\nu$ for $\Si$ equal to $\Si_{s_0}$, $\Si_0$, and respectively $\wt \Si$;
\item $\ma\cup_{s\in (0, s_0]} \wt \fM(\Si_s)_\nu$ is a topological manifold with boundary $\wt \fM(\Si_{s_0})_\nu$; 
\item the complement of $\ma \cup_{s\in  [0, s_0]}  \wt \fM(\Si_s)_\nu$ in  $\ma \cup_{s\in  [0, s_0]} \ov \fM(\Si_s)_\nu$ is homologically codimension 2. 
\end{enumerate}
\end{prop} 
In general, $\ma\cup_{s\in [0, s_0]} \wt \fM(\Si_s)_\nu$ is branched over $s=0$, but a cover of it is a topological cobordism cf.  \S\ref{S.cobord}. 
\medskip

{\em A priori}, the VFC as defined in \cite{ip-thin} may depend on the parameter space of perturbations used to achieve transversality. Using the results of 
\cite[\S6]{ip-thin} we end this section by comparing the VFCs of the (unperturbed) real moduli spaces $\ov \M^\R(\Si, V)$ constructed using perturbations defined on $\Si=\Si_{s\ne 0}$ or $\wt \Si$ with those constructed using perturbations defined on the corresponding family $\F$. 
\begin{lemma}\label{L.VFCs.match} For the first two families in \eqref{fam.moduli.consider}, the VFC associated by Theorem~\ref{T.VFC.rel.2} agrees 
under \eqref{nu.pullsback} with the VFC associated by Theorem~\ref{T.VFC.rel}. 
\end{lemma} 
\begin{proof} Assume for simplicity $\Si$ is equal to $\Si_{s\ne 0}$ (the proof for $\wt \Si$ is similar). Denote by $\cP \ra \cP'$ the map $\JV^\R(\cal F_{/\De}) \ra \JV^\R(\Si_s, V_s)$ 
defined as in \eqref{nu.pullsback}. We can then consider the family of moduli spaces 
\best
\ma\cup_\nu\ov\fM^\R_{d, \chi, \vec \mu}(\Si_s, V_s)_\nu
\eest
as $\nu$ varies over either $\cP$ or $\cP'$; denote these two families by $\ov\fM$ and respectively $\ov \fM'$. By \cite[Lemma~6.2]{ip-thin} it suffices to check that the conditions in \cite[Definition~6.1]{ip-thin} are satisfied. But the fiber $\ov\fM_\nu$ of $\ov\fM\ra \cP$ is the same as the fiber of $\ov \fM'\ra \cP'$ over the image of $\nu$ in $\cP'$; moreover the former fiber is cut transverally strata-wise if and only if the latter one is (because the linearization is the same for both problems). Therefore the Baire subsets of regular parameters (where the strata are cut transversally) satisfy the required conditions of \cite[Definition~6.1]{ip-thin}, completing the proof.
\end{proof} 
\medskip

\section{Moduli spaces for nodal targets}\label{S.mod.nodal} 
\medskip

In this section we describe in more detail the family of real relative moduli spaces associated to the nodal curve $\Si_0$ and analyze its orientation sheaf. 
Recall that $\Si_0$ is a nodal curve with a single pair $x^\pm$ of conjugate nodes and $r$ pairs of conjugate marked points, and $\wt \Si$ is its normalization \eqref{nor.Si.0}. We fix the topological data $d,\chi, \vec \mu$ and consider the moduli spaces $\ov \fM(\Si_0)$ and $\ov \fM(\wt \Si)$  introduced in 
\eqref{mo.over}-\eqref{mod.to.consider}. We continue working over the parameter space $\cP=\JV^\R(\cal F_{/\De})$ of RT perturbations on the family $\cal F$. Then the attaching map \eqref{map.attach.0} extends to give a map
\bear\label{attach.P}
\Phi: \ov\fM(\wt \Si)\ra \ov\fM(\Si_0) 
\eear
of moduli spaces over the parameter space $\cP$. Its restriction $\wt \fM(\wt \Si)\ra \wt \fM(\Si_0)$ to the codimension at most 1 strata \eqref{fam.moduli.consider.td} can be described as follows. Consider the cover of 
$\wt \fM(\Si_0)$ obtained by ordering the nodes in the preimage of the node $x^+$ of $\Si_0$. It comes with a group action reordering those nodes that have the same contact multiplicity. This cover is in fact homeomorphic to $\wt \fM(\wt \Si)$ and $\wt\fM(\Si_0)$ is the quotient by the corresponding group action, as we next show.

\subsection{The degree of the attaching map}\label{S.deg.phi}  We start with a useful result about the attaching map  \eqref{attach.P}; to state it precisely we need to introduce some extra notation.
Recall that by definition $\ov\fM(\wt \Si)$ is a disjoint union of spaces $\ov\fM_\la(\wt \Si)$ indexed by $\la$, cf. \eqref{mo.over.fam}, thus so is its open subset $\wt\fM(\wt \Si) = \ma\sqcup_{\la\vdash d} \wt\fM_\la (\wt \Si)$.

Moreover the elements $f_0:C_0 \ra \Si_0$ in $\wt \fM( \Si_0)$ have nodal domains and the preimage ${\bf y}^+=f_0^{-1}(x^+)$ of the target node $x^+$  consists only of nodes with matching contact multiplicity. We can similarly decompose the moduli space 
\bear\label{wt.M.0}
\wt \fM(\Si_0)=\ma \sqcup_{\la\vdash d} \wt \fM_{\la}( \Si_0)
\eear 
as a disjoint union of (open and closed) subsets indexed by the contact order $\la$ at the nodes $\bf y^+$. Denote by $x^+_1$, $ x^+_2$ the two marked points of $\wt \Si$ that get attached to produce the node $x^+$ of $\Si_0$. Order the $\ell=\ell(\la)$ nodes $y_1^+, \dots,  y_{\ell}^+$ in ${\bf y}^+$ so that the contact multiplicity of $f_0$ at $y_i^+$ is equal to $\la_i$,  the $i$'th part of the partition $\la$. 

Then $C_0$ is a nodal curve with $\ell$ pairs of conjugate nodes ${\bf y}=\{y_1^\pm, \dots,  y_{\ell}^\pm\}$ in the preimage of the 
nodes of $\Si_0$ and at most one other node.  Denote by $\wt C$ the (partial) normalization of $C_0$ resolving {\em only} the nodes in $\bf y$, and by $\wt f:\wt C\ra \wt \Si$ the unique lift of $f_0$ which satisfies the following conditions:
\begin{enumerate}[(i)] 
\item the lift to $\wt C$ of the node $y_i^+$ consists of two marked points denoted $y_{i1}^+$ and $y_{i2}^+$. 
\item $\wt f(y_{i1}^+)=x_1^+$ for all $i=1, \dots, \ell$. 
\end{enumerate} 
These conditions imply that the preimage $\wt f^{-1}(x_2^+)$ consists of the marked points $y_{i2}^+$, for $i=1, \dots, \ell$ and the contact multiplicity of $\wt f$ at both $y_{i1}^+$ and $y_{i2}^+$ is equal to $\la_i$, exactly as is the case for an element 
of $\wt\fM_\la(\wt \Si)$. 

Ordering the contact nodes $\bf y^+$ as above therefore gives rise to a cover of $\wt \fM_{\la}( \Si_0)$, and this cover is  canonically identified with $\wt \fM_{\la}( \wt\Si)$; the elements of 
$\wt \fM_{\la}( \wt\Si)$ are precisely the lifts $\wt f: \wt C \ra \wt \Si$ of the elements $f_0:C_0\ra \Si_0$ of $\wt \fM_{\la}( \Si_0)$. Moreover, $\wt \fM_{\la}( \wt\Si)$ comes with two natural $\Aut (\la)$ actions permuting the marked points in the preimage of $x_1^+$ and 
respectively $x_2^+$ (and therefore also permuting the corresponding conjugate points); the restriction 
\bear\label{attach.phi.0.tilde}
\begin{tikzcd}
\wt \fM_{\la}(\wt \Si)\ar[r, "\Phi|"]& \wt \fM_{\la}(\Si_0). 	
\end{tikzcd}
\eear
of the attaching map is the quotient by the diagonal $\Aut(\la)$ action. 
\medskip

Both moduli spaces in \eqref{attach.phi.0.tilde} are families over the parameter space $\cP=\JV^\R(\cal F_{/\De})$. The first one is relatively orientable over 
$\cP$ and an orientation is determined by a choice of twisted orientation data $\wt \mo$ on $\wt \Si$, as reviewed in \S\ref{S.orient} below. Moreover, the group action above is orientation preserving, thus the quotient is oriented. In particular, the moduli space $\wt \fM(\Si_0)$ is (relatively) orientable over $\cP=\JV^\R(\cal F_{/\De})$.
\begin{lemma}\label{L.phi.orient} 
The attaching map  \eqref{attach.P} restricts to a proper map \eqref{attach.phi.0.tilde} of degree $|\Aut(\la)|$ with respect to the orientations described above. Therefore, if $\Si_0$ has no rational connected components without any marked points, then 
\bear\label{Phi.0.push} 
\Phi_*[ {\wt \fM}_{\la} (\wt \Si)_\nu ]=  |\Aut\; \la|\cdot [{\wt \fM}_{\la} ( \Si_0)_\nu] 
\eear
for generic $\nu$ (i.e. for all $\nu\in \cP^*$).
\end{lemma} 
\begin{proof} Since the restriction of $\Phi$ is the quotient map by the $\Aut (\la)$ action, it is a proper map and all its fibers are finite; the generic fiber consists of $|\Aut\; \la|$ points. Moreover, for generic $\nu$, the fibers of both families of moduli spaces in \eqref{attach.phi.0.tilde} are topological manifolds by Proposition~\ref{P.or.mfl}. They are oriented by the procedure described above and the group action is orientation preserving. Therefore their fundamental classes in rational Cech homology are related by \eqref{Phi.0.push}, cf. \cite[\S1 and \S2.1]{ip-thin}. 
\end{proof}

\subsection{Linearizations and the orientation sheaf}\label{S.mod.nod.lin} This section provides several equivalent descriptions of the linearization to the moduli space $\wt \fM(\Si_0)$ and its orientation sheaf. The first one is induced via \eqref{attach.phi.0.tilde} from that of $\wt \fM(\wt \Si)$; another one naturally extends over the family of moduli spaces $\cup_s \wt \fM(\Si_s)$. For the 
proof of the splitting formula we must keep track of the identifications between these two perspectives, as they affect the induced orientations and therefore the signs of the  coefficients in \eqref{split.VFC}. 
\medskip

Recall that the elements of $\wt \fM(\Si_0)$ are maps $f:C_0 \ra \Si_0\subset \F$ without any rubber components, and that the total space $X=\F$ of the family \eqref{family.smoothings} is smooth. When 
\bear\label{5.E.Real.bd}
(E, c_E)\ra (\F, c_\F)
\eear 
is a Real bundle, and $W$ is a finite dimensional real vector bundle over $\wt \fM(\Si_0)$, then as $f$ varies in $\wt \fM(\Si_0)$ we get a family of (pullback) operators $\del_{f^*(E, c_E)}\oplus 0_{W_{f}}$. The families of linearizations to $\wt \fM(\Si_0)$ described in this section are of this form up to 0'th order terms. In that case, the 
results of \cite{gz}, applied to families of real maps to the smooth target $(\F, c_\F)$, imply that the determinant of such family of operators is identified 
with 
\bear\label{det.W}
\det (\del_{(E, c_E)}\oplus 0_{W})=  \det \del_{(E, c_E)} \otimes \det W 
\eear
(canonically up to homotopy).

The linearization at $f_0:C_0 \ra \Si_0$ to the moduli space $\wt \fM(\Si_0)$ is induced by the linearization 
 \bear\label{D.wt.f}
 D_{\wt f}: \Gamma(\wt f^* \T_{\wt \Si})^\R  \oplus T_{\wt C} \ov\M^\R \ra 
\La^{01}(\wt f^* \T_{\wt \Si})^\R 
\eear
at a lift $\wt f: \wt C \ra \wt \Si$ to the (partial) normalizations cf. \S\ref{S.deg.phi}. Because $\wt \Si$ is smooth, $D_{\wt f}$ is given by the formula \eqref{lin.rel} with $f$ replaced by $\wt f$. In particular, as in \eqref{or.sheaf}, this identifies the orientation sheaf of $\wt \fM(\wt \Si)$ with 
\bear\label{or.sheaf.resolve}
\mo_{\wt{\fM}(\wt \Si)}\cong \det \del_{(\T_{\wt \Si}, c_\T)}\otimes \mathfrak f^* \det  T \ov \M^\R
\eear
(uniquely up to homotopy). Above $\T_{\wt \Si}$ is the relative tangent bundle of $\wt \Si$, while $T \ov \M^\R$ denotes the tangent bundle of the Real Deligne-Mumford moduli space parametrizing the variations in $\wt C$ (We assume for simplicity $\wt C$ is stable).
\smallskip

\begin{rem}\label{R.pull.norm} For any complex bundle $E\ra C_0$, denote by $\wt E\ra \wt C$ its lift to the normalization $\wt C$ of $C_0$. Recall that $\Gamma(C_0; E)\subset\Gamma(\wt C; \wt E)$ is the subspace of sections that match at the nodes of $C_0$ while $\La^{01}(C_0; E):=\La^{01}(\wt C; \wt E)$. Moreover, note that the pullback $\wt E$ of $E=f^*_0 \T_{\Si_0}$ to the normalization  $\wt C$ is equal to $\wt f^* \T_{\wt \Si}$.
\end{rem}
Using \eqref{T.SES}, together with the normalization short exact sequence  \eqref{split.gamma} and \eqref{split.gamma.chose} for the bundle $E=f^*_0 \T_{\Si_0}$ allows us to rewrite the linearization \eqref{D.wt.f} as
\bear\label{lin.nodal.push}
\wt D_{f_0}: \Gamma(f_0^* \T_{\Si_0})^\R \oplus  (f_0^*\T_{\Si_0})_{|{\bf y}^+}  \oplus T_{\wt C} \ov \M^\R \ra 
\La^{01}(f^*_0 \T_{\Si_0})^\R.
\eear
Here an element $\zeta \in \Gamma(f_0^* \T_{\Si_0})^\R$ gives rise to a variation in $f_0$ with fixed domain, target, and also fixed product of the leading coefficients \eqref{lead.coef.f} at each one of the contact nodes, while the  middle term records the variation in the product of leading coefficients, cf. \S\ref{S.lin.nodal}. The operator \eqref{lin.nodal.push} has the form
\bear\label{lin.node.push}
\wt D_{f_0} (\zeta, \al, h)= \del_{ f_0^*(\T_{\Si_0}, c)} \zeta+ A_{f_0}(\zeta)+\gamma_{f_0}(\al)+ b_{f_0}(h). 
\eear
where $A_{f_0}$ and $b_{f_0}$ are as in \eqref{lin.rel}, while $\gamma_{f_0}$ is induced by  \eqref{split.gamma.chose} after a choice of splitting of \eqref{split.gamma} for 
$E= f_0^* \T_{\Si_0}$, as reviewed in \S\ref{S.lin.nodal}. Moreover, we can arrange that 
$\gamma_f$ be induced by pullback from a splitting of the normalization short exact sequence for the bundle 
$E=  \T_{\Si_0}$.  
\smallskip

As $f_0$ varies in $\wt\fM(\Si_0)$, the middle term in the domain of \eqref{lin.nodal.push} is intrinsically the pullback $\ev_{\bf y^+} ^*(\T_{|x^+})$  of $\T_{|x^+}$ under the evaluation map at the $\ell$ nodes ${\bf y}^+$ in the preimage of the node $x^+$, while the last term is $T _{C_0}\ov \N_\ell^\R$.  Here $\ell$ is the number of nodes of $C_0$ in the inverse image of $x^+$, and $\ov\N_{\ell}^\R$ denotes the nodal stratum of the real Deligne-Mumford moduli space consisting of Real curves with $\ell$ pairs of conjugate nodes. Up to 0'th order terms, the family of operators $\wt D_{f_0}$ given by \eqref{lin.nodal.push} is the same as $\del_{(\T_{\Si_0}, c_\T)}\oplus 0\oplus 0$.
This describes a family $\wt D$ of linearizations to $\wt\fM(\Si_0)$, and identifies the virtual orientation sheaf of $\wt \fM(\Si_0)$ with 
\bear\label{or.sheaf.nodal.push}
\mo_{\wt \fM(\Si_0)}\cong \det \wt D\cong  \det \del_{(\T_{\Si_0}, c_\T)}\otimes \det \ev_{\bf y^+}^*(\T_{|x_+}) \otimes  \mathfrak f^* \det  T \ov \N_{\ell}^\R 
\eear
(uniquely up to homotopy).

\begin{rem} \label{R.cover.triv.T} When $f$ is in $\wt \fM(\Si_0)$, for every node $y$ in the preimage $\bf y^+$ of the node $x^+$ of $\Si_0$, we can consider the product $a_y(f):=a(f; y_1)a(f; y_2)$ of the two leading coefficients \eqref{lead.coef.f}-\eqref{lead.coef.f.2} of $f$ at $y$. Then after passing to the cover of the moduli space $\wt \fM(\Si_0)$ obtained by ordering the nodes as in \S\ref{S.deg.phi}, this product is a nowhere vanishing section of a complex line bundle. This induces a canonical complex linear isomorphism 
\bear\label{lead.coef.iso}
\ev_{y}^*( \T_{|x^+} )\cong (\T_{|y})^{\la(y)}
\eear
for each $y$ in ${\bf y^+}$, where $\la(y)$ is the contact multiplicity at $y$. Here $\T_{|y}=\T_{|y_1}\otimes \T_{|y_2}$ denotes the relative tangent bundle at the node $y$, which is the pullback of the relative tangent bundle $\T$ on the universal curve by the section $f \mapsto (f, y)$. Note that $x^+$ is fixed, so the first term in \eqref{lead.coef.iso} is trivial (and canonically trivialized  by a choice of a complex linear isomorphism $\T_{|x^+} \cong \cx$). 
\end{rem}

For a map $f_0:C_0\ra \Si_0$ to a nodal target, there is another description of the linearization  \eqref{lin.nodal.push} that is better suited for smoothing the target and is used in the next two sections. It corresponds to the linearization defined in \cite[\S7]{ip-sum}. For that, we consider instead the operator 
\bear\label{lin.nodal.0}
D_{f_0}: \Gamma(f_0^* \T_{\Si_0})^\R  \oplus T_{C_0}\ov  \M^\R \ra 
\La^{01}(f^*_0\T_{\Si_0})^\R, 
\eear
defined by the same formula \eqref{lin.rel} as $D_{\wt f}$, but with $f=f_0$ instead of $f=\wt f$. Note that the first term of the domain of $D_{f_0}$ is a subspace of the first term of the domain of $D_{\wt f}$, while the second term $T_{C_0}\ov  \M^\R$ involves all variations in $C_0$, not just the ones tangent to the nodal stratum (the latter correspond to $T_{\wt C} \ov \M^\R$).  Note that the operators $D_{f_0}$ and $D_{\wt f}$ have the same target, 
 cf. Remark \ref{R.pull.norm}. 
\smallskip

Decompose $T_{C_0}\ov  \M^\R$ into tangent and normal directions to the nodal stratum $\N_\ell$ as in \eqref{tan-norDM}, where the normal direction is identified with $ \T_{|\bf y^+}$.  As reviewed in \S\ref{S.lin.nodal}, we can use the normalization short exact sequence for the bundle $E=\T_{C_0}$ to rewrite \eqref{lin.nodal.0} as an operator 
\bear\label{lin.nodal}
\wh D_{f_0}: \Gamma(f_0^* \T_{\Si_0})^\R \oplus \T_{|\bf y^+} \oplus T_{\wt C} \ov\M^\R \ra 
\La^{01}(\wt f^* \T_{\wt \Si})^\R
\eear
\best
\wh D_{f_0} (\zeta, v, h) =\del_{ f_0^*(\T_{\Si_0}, c)}\zeta + A_{f_0}(\zeta) + \wh \gamma_{f_0}(v)+ b_{f_0}(h). 
\eest
Here $A_{f_0}$ and $b_{f_0}$ are exactly as in the formula \eqref{lin.node.push} for the operator 
$\wt D_{f_0}$, while $ \wh \gamma_{f_0}$ depends on the splitting of the normalization sequence associated to $E=\T_{C_0}$. The middle term in \eqref{lin.node.push} keeps track of the variation in the product of the leading coefficients at the nodes, while the middle term in \eqref{lin.nodal} keeps track instead of the variation normal to the nodal stratum in the real Deligne-Mumford moduli space.

\begin{rem}\label{R.lin.many}  The discussion above describes the linearization along $\wt \fM(\Si_0)$ in several equivalent ways: the linearization $\wt D_{f_0}$ is obtained from $D_{\wt f}$ via a splitting of the normalization short exact sequence for $E=\T_{C_0}$, while the linearization $\wh D_{f_0}$ is obtained from $D_{f_0}$ via a splitting of the normalization short exact sequence for $E=\T_{\Si_0}$. The intermediate operators $\wt D_{f_0}$ and $\wh D_{f_0}$ defined by \eqref{lin.nodal.push} and \eqref{lin.nodal} are equivalent via a complex linear isomorphism between the middle terms of their domain induced by the linearization of  \eqref{lead.coef.iso} cf. \S\ref{S.lin.nodal}. In particular, these linearizations are all equivalent (just written using different coordinates); either one could be used to obtain a local model for  
 $\wt \fM(\Si_0)$ or to (globally) identify its virtual orientation sheaf.
\end{rem} 
In particular, we get the following result. 
\begin{lemma} 
The (relative) orientation sheaf $\mo_{\wt \fM(\Si_0)}$ of $\wt \fM(\Si_0)$ is canonically (up to homotopy) isomorphic to:
\bear\label{or.sheaf.nodal.push.4}
\mo_{\wt \fM(\Si_0)}&\cong& \det \del_{(\T_{\Si_0}, c_\T)}\otimes \det \ev_{\bf y^+}^*(\T_{|x_+}) \otimes  \mathfrak f^* \det  T \ov \N_{\ell}^\R 
\\
\label{or.sheaf.node}
\mo_{\wt \fM(\Si_0)}&\cong& \det \del_{(\T_{\Si_0}, c_\T)} \otimes 
	\mathfrak{f}^* \det T \oM^{\R}.
\eear
\end{lemma}
\begin{proof} The first isomorphism is \eqref{or.sheaf.nodal.push}, obtained by using the family of linearizations  \eqref{lin.node.push}. The second one is obtained the same way as \eqref{or.sheaf.nodal.push}, but using instead the family of linearizations \eqref{lin.nodal.0}. In  \eqref{or.sheaf.node}, the last term $T \ov \M^\R$ denotes the tangent bundle of the real Deligne-Mumford moduli space parametrizing all the variations in the nodal domain $C_0$, including those that smooth out the nodes. 
\end{proof} 

\section{The construction of the cobordism}\label{S.cobord}
\medskip

In this section we continue working with the moduli spaces \eqref{mod.or.consider} over the parameter space 
$\cP=\JV^\R(\cal F_{/\De})$ of RT perturbations on the family $\F$. We define an auxiliary space denoted $\wh \fM(\F_{/ I})$ and prove that generically its fiber $\wh \fM(\F_{/ I})_\nu$ is a cobordism between $\wt \fM(\Si_{s_0})_\nu$ and a cover $\wh \fM(\Si_0)$ of $\wt \fM(\Si_0)_\nu$. This will allow us to compare their fundamental classes (in rational Cech homology). 

\subsection{The cover $\wh \fM(\Si_0)$ of $\wt \fM(\Si_0)$} We start by defining a cover 
\bear\label{def.q.0}
q_0: \wh \fM(\Si_0)\ra \wt \fM(\Si_0)
\eear
of the moduli space $\wt \fM(\Si_0)$. Decompose $\wt \fM(\Si_0)$ as in \eqref{wt.M.0} according to the contact multiplicity $\la$ at the nodes in the preimage of the node $x^+$ of the target $\Si_0$. Fix local coordinates on $\Si_0$ around $x^+$; this induces a trivialization $\T_{x^+}\cong \cx$.  The elements of $\wt \fM_\la(\Si_0)$ are $(J, \nu)$-holomorphic maps $f:C_0\ra \Si_0$, without any rubber components, and such that
\begin{enumerate}[(i)]
\item the inverse image $f^{-1}(x^+)={\bf y}^+$ consists only of nodes $y_i^+$ of the domain with matching contact multiplicities $\la_i$, for $i=1, \dots, \ell(\la)$ and 
\item the domain has at most 1 other node (which must be real).
\end{enumerate}
The cover 
\bear\label{wh.M}
\wh \fM(\Si_0)=\ma \sqcup_{\la\vdash d} \wh \fM_\la (\Si_0)
\eear
consists of essentially the same objects, except that   
\begin{enumerate}[(i)]
\item [(iii)] for each contact node $y_i^+$ with contact multiplicity $\la_i$, we also choose a $\la_i$-root $\al_i \in \T_{|y_i^+}^\vee$ of the product of the leading coefficients  of $f$ at $y_i^+$, for every $i=1, \dots, \ell(\la)$, cf. \eqref{lead.coef.f}-\eqref{lead.coef.f.2} and \cite[(6.1)-(6.2)]{ip-sum}. 
\end{enumerate}
Denote the objects of $\wh \fM(\Si_0)$ by 
\bear\label{f.0.al}
(f, \al), \quad \text{ where } f \in  \wt \fM(\Si_0) \quad \mbox{ and }\quad \al= (\al_i)_i \in  (\T_{|\bf y^+})^\vee
\eear
satisfies
\bear\label{al.j}
\al_{i}^{\la_i}= a(f;  y_{i1}^+)  a(f; y_{i2}^+)  \quad \mbox{ for all }\quad i=1, \dots, \ell(\la).  
\eear
Here $a(f; y)=a_y(f)$ denotes the leading coefficient of $f$ at the point $y$ cf. \eqref{lead.coef.f}--\eqref{lead.coef.f.2}, while $\la_i$ is the contact multiplicity at $y=y_{i1}^+$ and $y_{i2}^+$.  

 With this definition, the map \eqref{def.q.0} is given by $q_0(f, \al)=f$. 
\medskip

Recall that the moduli space $\wt \fM(\Si_0)$ over the parameter space $\cP$ is relatively orientable as discussed above Lemma~\ref{L.phi.orient}. Fix any orientation on it and pull it back to $\wh\fM(\Si_0)$. 
\begin{lemma}\label{C.q.deg} 
The map  $q_0: \wh\fM_\la(\Si_0)\ra  \wt\fM_\la(\Si_0)$ is proper and has degree 
$\zeta(\la)/|\Aut(\la)|$ with respect to these orientations. Therefore  if $\Si_0$ has no rational connected components without any marked points, 
\bear\label{q.0.push} 
(q_0)_*[ {\wh \fM}_{\la} (\Si_0)_\nu]= \frac{\zeta(\la)}{|\Aut\, \la|} \cdot [{\wt \fM}_{\la} ( \Si_0)_ \nu] 
\eear
for generic $\nu$ (i.e. for all $\nu\in \cP^*$).
\end{lemma} 
\begin{proof} By construction $\wh \fM(\Si_0)$ comes with a group action on the choice of root 
$\al_i^+$ by a $\la_i$-root of 1. So the quotient map $q_0$ is proper and the order of the group is 
\best
\prod_{i}\la_i= \frac{\zeta (\la)}{|\Aut \,\la |} 
\eest
cf. \eqref{aut.la}-\eqref{mult.la}. Since here the orientation of $\wh\fM(\Si_0)$ is pulled back from that of $\wt \fM(\Si_0)$, then the action is orientation preserving. Moreover for generic $\nu$, the fiber of $\wt \fM(\Si_0)_ \nu$ is an oriented topological manifold by Proposition~\ref{P.or.mfl}, and therefore so is its cover (with the pullback orientation), giving \eqref{q.0.push}. 
\end{proof}

\subsection{The cobordism moduli space and its topology} Start with the subset $\wt\fM(\F_{/ I})= \ma \cup_{s\in I} \wt \fM(\Si_s)$ of the family \eqref{mo.over.fam}. Consider next the {\em set}
\bear\label{wh.M.fam}
\wh \fM(\F_{/ I})= \wh \fM(\Si_0) \sqcup \ma \cup_{s\ne 0} \wt \fM(\Si_s) 
\eear
with the topology defined below. Extend $q_0$ as the identity on $\ma \cup_{s\ne 0} \wt \fM(\Si_s)$ to get a map 
\bear\label{def.q}
q: \wh \fM(\F_{/ I}) \ra \wt \fM(\F_{/ I})= \ma \cup_{s\in I} \wt \fM(\Si_s). 
\eear
The RHS is a subset of $(J, \nu)$-holomorphic maps $f:C \ra \Si_s \subset \F$ without any rubber components, for some $s\in I$. Therefore it comes with the topology induced by the usual Gromov topology on the moduli space of real $(J, \nu)$-holomorphic maps to the total space $X=\F$ of the family of targets, cf. \cite[\S4.2]{gz}. 

\medskip

Below we define a Hausdorff topology on the LHS of \eqref{def.q}  with the following properties: (a) $q$ is proper and continuous (b) $q$ restricts to the covering map $q_0$ for $s=0$ and to the identity for $s\ne 0$. In particular, since $q$ will be continuous, the topology on the LHS will be a refinement of the Gromov topology as considered in \cite[\S4-5]{ip-sum}. 

By definition, \eqref{wh.M.fam} is the disjoint union of two sets $\wh \fM(\Si_0)$ and  $\ma \cup_{s\ne 0} \wt \fM(\Si_s)$ and 
condition (b) canonically determines the topology on each one of these two subsets, where $\wh \fM(\Si_0)=q^{-1}(\wt \fM(\Si_0))$ must be a closed subset by (a).  Therefore it suffices to describe what it means for a sequence 
of maps $f_n:C_n \ra \Si_{s_n}$ with $s_n\ne 0$ to converge to an element $(f_0, \al)\in \wh \fM(\Si_0)$, beyond the fact that it must converge to $f_0$ in the usual Gromov topology. In short, as in \cite{ip-sum}, the domains of these maps must converge to the domain $C_0$ of $f_0$ along a fixed direction normal to the nodal stratum of domains containing $C_0$.  To state this precisely, we use the set-up of \cite[\S3-5]{ip-sum}, adapted to our setting. We start with a choice of local coordinates as follows; see also \cite[\S4]{ip-sum}. 
\smallskip

Fix an element $(f_0, \al)$ of $\wh \fM(\Si_0)$ and let $C_0$ denote the domain of $f_0$. Let $w_1$, $w_2$ be the fixed local coordinates around the node $x^+$ of the target; then in a neighborhood of $x^+$ the curve $\Si_s\subset \F$ is described by $w_1w_2=s$, where $s\in I=[0, s_0]\subset \De$ is the gluing parameter. After rotating the disk $\De$ parametrizing the family $\F_{/\De}$ of targets, we may assume that $I\subset \R_{\ge0}$. Fix local coordinates $z_1$, $z_2$ around each node $y$ of $C_0$ as in \eqref{loc.coord.C.node} and construct a local slice $\cal S$ as in \eqref{tau.tauC.A} parametrizing local 
deformations $C_{\tau, u}$  of $C_0$. Here $u$ are local coordinates along the nodal stratum $\N_\ell$ of the real Deligne-Mumford moduli space while the gluing parameters 
\bear\label{tau.tauC}
\tau=(\tau_{i})_i \in \T_{|\bf y^+}, \quad \text{ where }\tau_i\in  \T_{| y^+_i} \text{ for all }i=1, \dots, \ell(\la)
\eear 
cf. \eqref{loc.coord.C.sm.node} provide normal coordinates to this stratum cf. \eqref{SES.T.h1}-\eqref{tan-norDM}. Note that a choice of local coordinates on $C_0$ at $y^+_i$ induces a local trivialization of the relative tangent bundle $ \T$ around $y_i^+$.
\medskip

Below we also use the notion and properties of $\de$-flat maps in the sense of \cite[Definition~3.1]{ip-sum} to describe a neighborhood of $f_0\in \wt \fM(\Si_0)\subset \cup_s \wt\fM(\Si_s)$ in the Gromov topology on the target of \eqref{def.q}; for precise details see \cite[(3.1)-(3.5)]{ip-sum} and  \cite[Lemma~3.2]{ip-sum}. Roughly speaking, $\de$-flat maps $f:C\ra \Si_s \subset \F$ do not have enough energy in the pre-image of a $\de$-neighborhood $U_\de \subset \F$ of the nodes $\bf x$ of $\Si_0\subset \F$ for a rubber component to start forming in there, cf.  \cite[(3.1)-(3.5)]{ip-sum}, thus giving rise to the uniform estimates of \cite[\S5]{ip-sum}. As $\de\ra 0$, $\de$-flat maps exhaust the complement of the strata with at least one rubber component, cf.  \cite[Lemma~3.2]{ip-sum}.  
\smallskip 

Consider now a sequence of  $(J, \nu)$-holomorphic maps  $f_n:C_n \ra \Si_{s_n}$ with $s_n\ne 0$ and which converges in the Gromov topology to the map $f_0\in \wt \fM(\Si_0)$. Since $f_0$ has no components mapped to the singular locus $\bf x$ of $\Si_0$, it is $\de$-flat in the sense of 
 \cite[Definition~3.1]{ip-sum} for all sufficiently small $\de>0$, cf.  \cite[(3.1)-(3.4)]{ip-sum}. Gromov convergence then implies that after decreasing $\de$, $f_n$ is $\de$-flat for all sufficiently large $n$. Moreover, up to reparametrizations of the domains and for large $n$, we may identify $$C_n=C_{\tau_n, u_n}$$ as in the paragraph containing \eqref{tau.tauC}, where $(\tau_n, u_n) \ra 0$ as $n\ra \infty$. 
Therefore by \cite[Lemma~5.3]{ip-sum}
\bear\label{lim.is.prod}
\lim_{n\ra \infty} \frac{s_n}{\tau_{i, n}^{\la_i}} = a(f_0; y_{i1}^+)  a(f_0; y_{i2}^+)  
\eear
for each contact node $y_i^+$ of $C_0$. Here $\tau_n=(\tau_{i, n})_i$ are the gluing parameters \eqref{tau.tauC} corresponding to 
$C_n= C_{\tau_n, u_n}$ (for $n$ large), while as before $a(f; y)\ne 0$ is the leading coefficient of $f$ at $y$ (in these local coordinates). In particular, since $s_n\ne 0$, this implies that $\tau_{i, n}\ne 0$ for all $i=1, \dots, \ell$ and $n$ large. 
\smallskip

With these preliminaries understood, declare that the sequence $f_n:C_n \ra \Si_{s_n}$ with $s_n\ne 0$ \textsf{converges in $\wh \fM(\F_{/ I})$ to} 
$(f_0, \al)$ if and only if $f_n$ converges to $f_0$ in the Gromov topology and moreover 
\bear\label{lim.is.prod.root}
\lim_{n\ra \infty} \frac{(s_n)^{1/\la_i}}{\tau_{i, n}} = \al_i
\eear
for all $i=1, \dots, \ell$. Here  $s_n\in \R_{>0}$ (by our simplifying assumption that $I\subset \R_{\ge 0}$) thus it has a unique $\la_i$-root $(s_n)^{1/\la_i}\in \R_{>0}$. The limit $(f_0, \al)$ if it exists is unique and is independent of the choice of local coordinates around $C_0$,  cf. \eqref{f.0.al} and \eqref{tau.tauC}. This completes the definition of the topology on the domain of \eqref{def.q}, which by construction satisfies all the properties listed in the paragraph after \eqref{def.q}. 

\begin{rem} In light of \eqref{al.j} and \eqref{lim.is.prod}, condition \eqref{lim.is.prod.root} can be equivalently replaced by 
\best
\lim_{n\ra \infty} \arg \tau_{i, n} = -\arg \al_i,  
\eest
where $\arg w$ denotes the argument of $w$; in particular, this means that the domains $C_n=C_{\tau_n, u_n}$ of $f_n$ converge to the domain $C_0$ of $f_0$ along a fixed direction normal to the nodal stratum $\N_\ell$; see also \cite[Lemma~4.2 and (4.10)]{ip-rel}.  By \cite[Lemma~5.4]{ip-sum}, this also implies that $f_n$ converge to $f_0$ in the $\la$-weighted Sobolev norms of  \S\ref{S.A.weighted} used in the proof of the splitting formula; see also \eqref{lin.rel.norms.weighted}. 
\end{rem}
\medskip

This motivates considering the following \textsf{local model}, regarded as a section of the (pullback) bundle $\T_{\bf y_+}$ over $\wh \fM(\Si_0)\ti \R_{\ge 0}$. It consists of tuples 
\bear\label{local.mod.hat}
(f, \al, \tau, s) \quad \mbox{ such that }  \quad \al_i \tau_i= s^{1/\la_i}  \quad \mbox{ for all }\quad i=1, \dots, \ell.  
\eear
Here $(f, \al)$ is an element of $\wh \fM(\Si_0)$, while $\tau=(\tau_i)_i\in \T_{\bf y_+}$ and $s\in \R_+$. Note that in cylindrical coordinates the conditions defining \eqref{local.mod.hat} become linear (i.e. $\log \al_i+\log \tau_i={1/\la_i}\log s$). 
\subsection{The cobordism} Consider the family 
\bear\label{wh.M.fam.2} 
\wh \fM(\F_{/ I})= \wh \fM(\Si_0) \cup \ma \cup_{s\ne 0} \wt \fM(\Si_s)  \ra \cP
\eear
with the topology defined above; it is a family over the parameter space $\cP=\JV(\F_{/ \De})$ of Ruan-Tian perturbations, as well as a family over the interval $I=[0, s_0]$. We next verify that for generic $\nu\in \cP$, the fiber 
$\wh \fM(\F_{/ I})_\nu$ of \eqref{wh.M.fam.2} is a topological cobordism. This is the main step in the proof of  Theorem~\ref{T.splitVFC}. 
\begin{prop}\label{P.cob} Assume $\Si_0$ has no rational connected components without any marked points. Then for generic RT perturbation $\nu\in \cP$, the fiber $\wh \fM(\F_{/ I})_\nu$ of \eqref{wh.M.fam.2} is a (possibly non-compact) orientable topological cobordism between $\wt \fM(\Si_{s_0})_\nu$ and the cover $\wh \fM(\Si_0)_\nu$ of $\wt \fM(\Si_{0})_\nu$ defined by  \eqref{def.q.0}. 
\end{prop} 
\begin{proof} As in Proposition~\ref{P.or.mfl},  for generic $\nu$, 
\begin{enumerate}[(i)]
\item $\ma \cup_{s\in (0, s_0]}\wt \fM(\Si_s)_\nu$ is an orientable topological manifold with boundary $\wt \fM(\Si_{s_0})_\nu$ and 
\item $\wt \fM(\Si_0)_\nu$ is an orientable manifold and therefore so is its cover $\wh \fM(\Si_0)_\nu$. 
\end{enumerate} 
Moreover, as in \eqref{mod.to.consider}, for generic $\nu$ all the strata of $\ov \fM(\Si_0)$ are cut transversally, i.e. the cokernel of the linearization $D_f$ is trivial at all elements $f$ in the fiber over $\nu$.

Fix such a generic parameter $\nu$. We next describe the local model of a neighborhood in $\wt \fM(\F_{/ I})_\nu$ of a fixed point $f_0:C_0 \ra \Si_0$ in 
$\wt \fM(\Si_0)_\nu$. Recall that by definition $\wt \fM(\F_{/ I})= \cup_{s\in I}\wt \fM(\Si_s)$ is an {\em open} subset of the family $\ov\fM(\F_{/ I})=\cup_{s\in I} \ov\fM(\Si_s)$ of moduli spaces and consists of maps without any rubber components. Moreover, since $f_0$ has no rubber components, a sequence $f_n$ converges to $f_0$ in the topology of the compactified family $\ov \fM(\F_{/ I})= \cup_{s\in I}\ov \fM(\Si_s)$ if and only if $f_n:C_n\ra \Si_{s_n}\subset \F$ converge to $f_0$ in the Gromov topology on the space of real $(J, \nu)$-holomorphic maps $f:C\ra \F$. Thus a sufficiently small neighborhood $U_{f_0}$ of $f_0$ in the family $\ov\fM(\F_{/ I})_\nu=\cup_{s} \ov\fM(\Si_s)_\nu$ consists only of $\de$-flat maps for some  
$\de>0$. Finally, the cokernel of the linearization at $f_0$ is onto, and shrinking the neighborhood $U_{f_0}$ we may assume that it is a subset the open subset $\wt \fM(\F_{/ I})_\nu$ of $\ov\fM(\F_{/ I})_\nu$, thus disjoint from any of the codimension at least two strata.

Therefore, the gluing argument of \cite{ip-sum} applies to describe a sufficiently small neighborhood $U_{f_0}$ of $f_0$ in terms of the space of approximate maps, which in turn are constructed starting from the model space \eqref{local.mod.hat}, cf. \cite[Definitions~6.1-6.2]{ip-sum} (noting the change in notation).  

For generic $\nu$, the local model of the neighborhood $U_{f_0}$ of $f_0$ in $\wt \fM(\F_{/ I})_\nu$ is therefore described by tuples 
\bear\label{local.mod}
(f, \tau, s) \quad \mbox{ such that }  \quad a(f; y_{i1}^+)  a(f; y_{i2}^+) \tau_i^{\la_i} = s \quad \mbox{ for all }\quad i=1, \dots, \ell.  
\eear
Here $f:C\ra \Si_0$ belongs to a neighborhood $\O_{f_0}$ of $f_0:C_0\ra \Si_0$ in the moduli space $\ov\fM(\Si_0)_\nu$ and $a(f; y)$ is the leading coefficient of $f$ at $y$. As before, $\tau=(\tau_i)_i$ are the gluing parameters of the domain and $s\in \R_{\ge 0}$ is the gluing parameter of the target. Moreover $\wt \fM(\Si_0)_\nu$ is an open subset of  $\ov\fM(\Si_0)_\nu$, therefore we may assume that $\O_{f_0}$ is a neighborhood of 
$f_0$ in $\wt \fM(\Si_0)_\nu$ after possibly shrinking it. 

By the definition \eqref{def.q} of the space $\wh \fM(\F_{/ I})$ and of its topology, the local model of 
$\wh \fM(\F_{/ I})_\nu$ around the point $(f_0, \al_0) \in \wh \fM(\Si_0)_\nu$ is therefore described by tuples $(f, \al, \tau, s)$ as in \eqref{local.mod.hat} where 
$(f, \al)$ is in the neighborhood of $(f_0, \al_0)$ in $\wh \fM(\Si_0)_\nu$, while $\tau=(\tau_i)_i$ and $s$ are the gluing parameters as above. But $\wh \fM(\Si_0)_\nu$ is a topological manifold and we can uniquely solve the equations defining \eqref{local.mod.hat} for the variable $\tau$ in terms of $(f,\al)$ and $s\in \R_{\ge 0}$, thus the fiber of the local model is a topological manifold with boundary at $s=0$, and therefore so is the fiber of $\wh \fM(\F_{/ I})$ in a neighborhood of $s=0$.
\end{proof}

We next consider the orientation sheaves of these spaces. When $\wt \fM(\F_{/ I})= \cup_{s\in I} \wt \fM(\Si_s)$ is regarded as a family over the parameter space $\cP \ti I$, its fiber at $(\nu, s)$ is $\wt \fM(\Si_s)_\nu$, and the (fiber-wise) linearization $D_f$ is given by the formula \eqref{lin.rel}, including when $s=0$ cf. \eqref{lin.nodal.0}. The kernel of this linearization encodes the (formal) tangent space to the fiber and its cokernel is the obstruction. 
When  $\wt \fM(\F_{/ I})= \cup_{s\in I} \wt \fM(\Si_s)$ is regarded as a family over $\cP$, we can similarly consider the fiberwise linearization ${\mathbb D}_f$, as $s$ is now allowed to vary in the interval $I$; the operator ${\mathbb D}_f$ is an extension of $D_f$ and its domain has an extra term  keeping track of the variation $\de s\in T_s I\cong \R$ in $s$. As $f:C\ra \Si_s\subset \F$ varies in $\wt \fM(\F_{/ I})$, we get a family of real maps to a {\em smooth} target $X=\F$ and up to 0'th order terms, the family of operators ${\mathbb D}_f$ is the same as a  (pullback) family of the form $\del_{f^*(\T, c_\T)}\oplus 0_{W_f}\oplus 0_\R$; see also paragraphs containing \eqref{det.D.T} and \eqref{5.E.Real.bd}. Therefore as before, we get an identification
\best
\det {\mathbb D}  \cong \det (\del_{(\T, c_\T)}\oplus 0_{W} \oplus 0_\R) 
\eest
over $\wt \fM(\F_{/ I})$.

By definition, the (virtual) relative orientation sheaves of the moduli spaces $\wh \fM(\Si_0)$ and $\wh \fM(\F_{/ I})$ over the parameter space $\cP$ are the (pullback) of the determinant bundle of the family of operators $D_f$ and respectively ${\mathbb D}_f$ as $f$ varies in 
$\wt \fM(\F_{/ I})$. The proof above also implies that for generic $\nu$, the fiber $\wh \fM(\F_{/ I})_\nu$  is a topological cobordism whose orientation sheaf is canonically identified with $\det {\mathbb D}_f$ (up to homotopy), including along the codimension 1 strata, while the orientation sheaf of its boundary $\wh  \fM(\Si_0)_ \nu \sqcup \wt \fM(\Si_{s_0})_\nu$ is similarly given by $\det D_f$. Both statements follow from the proof of the gluing theorem, which shows that the corresponding  moduli space is a small deformation of the space of approximate maps, and that the tangent space to the space of approximate maps is a small deformation of the kernel of the linearization. 

This gives the following result. 
\begin{cor}\label{C.or.cob} The (relative) orientation sheaf of the family $\wh \fM(\F_{/ I})$ over $\cP$ can be canonically identified (up to homotopy) with the pullback under \eqref{def.q} of the bundle 
\bear\label{det.wh.M}
\det \del_{(\T, c_\T)} \otimes 
\mathfrak{f}^* \det T \oM^{\R} \otimes \pi^* \det T I.  
\eear
Here $\T\ra \F$ is the relative tangent bundle \eqref{rel.T.fam}.
\end{cor} 
The bundle \eqref{det.wh.M} is orientable over $\wt \fM(\F_{/ I})$ and canonically oriented after choosing a twisted orientation data on $\T\ra \F$, as described after \eqref{or.sheaf.cob.2}. In particular, any choice of orientation on  \eqref{det.wh.M} determines an orientation of the cobodism $\wh \fM(\F_{/ I})_\nu$ as well as on its two boundaries. With respect to this choice, Proposition~\ref{P.cob} gives the equality of the fundamental classes
\bear\label{wh.cob.eq} 
[\wt \fM (\Si_0)_\nu] = [ \wh \fM(\Si_{s_0})_\nu]\quad \mbox{ in } \quad \cHH_*(\wh \fM(\F_{/ I})_\nu; \Q)  
\eear
for generic $\nu$ in $\cP$. 
\begin{rem} Here we are using the convention that if $M_0$ and $M_1$ are oriented topological manifolds, and $W$ is an oriented cobordism from $M_0$ to $M_1$ then $W$ is a manifold whose boundary is $\partial W=-M_0 \sqcup M_1$ as oriented manifolds. 
\end{rem} 

\section{Orientations}  \label{S.orient}
\medskip

In this section we review the orientation procedure for all the moduli spaces considered in \S\ref{S.VFC.mod}-\S\ref{S.cobord} and show that  a choice of (twisted) orientation data on $\Si$ or on the family $\F$ canonically orients these moduli spaces (over the space of Ruan-Tian perturbations on $\Si$ or $\F$). 
\medskip

First, recall that if $X$ is a space with an involution $c$, then a Real\footnote{called a real bundle pair in \cite[\S1.1]{gz}} bundle over 
$(X, c)$ is a pair $(E, \phi)$ where $E\ra X$ is a complex vector bundle and $\phi$ is an anti-complex linear involution on $E$ covering $c$; then $\phi$ is called the real structure on $E$.  An isomorphism between two Real bundles $(E_i, \phi_i)$ is a complex linear isomorphism $\psi$ compatible with the real structures i.e. such that $\psi^*\phi_{2}=\phi_{1}$. We denote by $\ov E$ is the same real bundle as $E$, but with opposite complex structure $J\mapsto -J$. 

\begin{defn}[{\cite[Appendix]{GI}}]\label{TRO.gen}  Let $(X, c)$ be a Real symplectic manifold and $(W,\phi)$ a Real bundle over it. A \textsf{twisted orientation} $\mo=(L, \psi,\mathfrak{s})$ for the bundle $(W, \phi)$ consists of 
	\begin{enumerate}[(i)]
		\item a complex line bundle $L\ra X$ such that the Real bundle $(E, c_E)=(L\oplus c^* \ov L, c_{tw})$ satisfies: 
		\bear\label{cond.tw.or}
		w_2(W^\phi)= w_2(E^{c_E}) \quad \mbox{ and } \quad 
		\La^{\mathrm{top}} (W, \phi) \cong  \La^{\mathrm{top}} (E, c_E)
		\eear	
		\item  a homotopy class $[\psi]$ of isomorphisms satisfying \eqref{cond.tw.or}. 
		\item  a spin structure $\mathfrak{s}$ on the real vector bundle $W^{\phi}\oplus (E^\vee)^{c_{E}^\vee}$ over the real locus, compatible with the orientation induced by $\psi$.
	\end{enumerate}	
\end{defn}
Here $E^\vee$ denotes the dual of $E$ and the real structure $c_{tw}$ on $L\oplus c^* \ov L$ is defined by 
\bear\label{L.plus.twL}
c_{tw}(x; v, w)=(c(x); w, v)
\eear
for all $x\in X$, $v\in L_x$ and $w\in \ov L_{c(x)}$. Note that given any complex line bundle $L$ satisfying condition (i), one can then {\em chose} (not uniquely) a homotopy class $[\psi]$ as in (ii) and a spin structure $\mathfrak{s}$ as in (iii) such that together they give a 
choice of twisted orientation on $(W, \phi)$.
\smallskip

As in \cite[Proposition~5.2]{gz}, the conditions in Definition~\ref{TRO.gen} determine a canonical homotopy class of isomorphisms
\bear\label{identif.f}
f^*(W \oplus E^\vee, \phi \oplus c_{E}^\vee) \cong (C\times \cx^{n+2}, c_{std}) 
\eear
which vary continuously as $f$ varies in a space of real maps $f:C\ra X$. (Here $n$ is the complex rank of $W$). In particular, as in the proof of \cite[Theorem~1.3]{gz} they determine a canonical homotopy class of isomorphisms 
\bear\label{can.iso.V}
\det \del_{(W, \phi)}\otimes \det \del_{(E, c_{E})^\vee }\cong \det\del_{(\cx^{n+2}, c_{std})}. 
\eear

But since $(E, c_E)^\vee=(L\oplus c^*\ov L, c_{tw})^\vee$ there is also a canonical isomorphism 
\bear\label{det.E.L}
\det \del _{(E, c_E)^\vee}  \cong p_1^*\det \del _{L^\vee}
\eear
obtained as in \cite[(2.9)]{GI} by projecting onto the first factor. The 
RHS of \eqref{det.E.L} is the determinant bundle of a family of complex operators thus has a canonical (complex) orientation.   
Together with the canonical orientation on the square of a bundle,  \eqref{can.iso.V} induces a trivialization
\bear\label{V.o.iso} 
\begin{tikzcd} 
		\det \del_{(W, \phi)}\otimes \det\del_{(\cx^n, c_{std})} \arrow{r}{\mo}[swap]{ \cong}&\underline{\R} 
\end{tikzcd}
\eear
which depends on the choice of twisted orientation data $\mo$. Here $\underline{\R}$ denotes the trivial real line bundle (canonically trivialized).  Note that the existence of a complex line bundle $L$ satisfying condition (i) of Definition~\ref{TRO.gen} ensures that the real line bundle on the right-hand side of \eqref{V.o.iso} is {\em orientable}, while conditions (ii)-(iii) of Definition~\ref{TRO.gen} fix an orientation on it.

\subsection{Existence of twisted orientations}\label{S.tw.exist}  We next show that, for symmetric curves {\em without any real special points}, there exits a choice of twisted orientation data on their relative tangent bundle. We also show that the same is true for the family of targets \eqref{family.smoothings}. 
\smallskip

First, when $(\Si, c)$ is a (possibly nodal) symmetric curve, then up to isomorphism, Real line bundles $(W, \phi)$ over $X$ are classified by 
$c_1(W)$ and $w_1(W^\phi)$, cf. \cite[Theorem~1.1]{GZ3}. Thus such a bundle has a twisted orientation data if and only if $w_1(W^\phi)=0$ and there exits a complex line bundle $L\ra \Si$ such that 
\bear\label{c.1.W}
c_1(L\otimes c^*\ov L)= c_1(W)
\eear
 cf. \eqref{cond.tw.or}; the first condition in \eqref{cond.tw.or} is automatic in this case.  
\smallskip

Let $(\Si,c)$ be a (possibly nodal) symmetric marked curve with marked points $V$ and without any real special points. Let $W=\T_\Si$ be its relative tangent bundle, regarded here as a Real line bundle, i.e. a complex line bundle with a real structure $c_\T$. Recall that $\T_\Si$ is given by \eqref{T.punctured} when $\Si$ is smooth and by \eqref{T.SES} when $\Si$ is nodal. Since none of the special points are real, the real locus of $\T_\Si$ is orientable. Thus a twisted orientation data $\mo=(\Theta, \psi, \mathfrak s)$ on 
$\T_\Si$ exists provided we can find a complex line bundle $\Theta \ra \Si$ such that 
\bear\label{c.1.T}
c_1(\Theta\otimes c^*\ov \Theta)= c_1(\T_\Si).
\eear
We next show that such a bundle $\Theta$ exists for any such $(\Si, c)$. When $\Si$ is smooth and connected, then the bundle $\Theta$ must be a complex line bundle whose Chern number is equal to $\tfrac 12c_1(\T_\Si)[\Si]=\tfrac  12\chi (\Si\setminus V)\in \Z$; such a bundle exists and is unique up to complex linear isomorphism. 

When $\Si$ is a smooth doublet (i.e. has two connected components interchanged by $c$) then $\Theta$ can be any complex line bundle on $\Si$ whose Chern numbers on the two halves sum to $\tfrac{1}{2}\chi(\Si\setminus V)\in \Z$; see \cite[Example 2.2]{GI}. Therefore a bundle satisfying condition \eqref{c.1.T} exists on every {\em smooth} symmetric curve $\Si$ without real marked points. Finally, when $\Si_0$ is a nodal symmetric curve which has no real special points, the bundle $\Theta_0$ on 
$\Si_0$ satisfying condition \eqref{c.1.T} can be obtained from one on the normalization after choosing an identification at the marked points corresponding to the nodes, compatible with the real structure; such a choice is unique up to homotopy. 

\smallskip

Consider next the family of targets \eqref{family.smoothings}, and let $\T\ra \F$ be its relative tangent bundle \eqref{rel.T.fam}, regarded as a Real line bundle. Note that $(\F, c_\F)$ deformation retracts to $(\Si_0, c_0)$ through real maps, so $\iota_0 \circ r$ is real 
homotopic to the identity, where $r: \F\ra \Si_0$  denotes the retraction.  Moreover, isomorphism classes of Real bundles on a Real space $X$ are in 1-1 correspondence with homotopy classes of real maps from $X$ to the classifying space, cf. \cite[Proposition~2.1]{real.bd.classif}. In particular, $r^*\T_{\Si_0}\cong \T$ as Real bundles. Since $ \F$ deformation retracts to $\Si_0$, every bundle on $\Si_0$ can be extended over the family  
$ \F$. Thus the complex line bundle $\Theta_0 \ra \Si_0$ described above extends to a complex line bundle $\Theta_\F$ over $\F$ with the property that there exists an isomorphism of Real bundles
$$
\psi_\F: (\T, c_\T) \lra (\Theta_\F\otimes c^*_\F\ov{\Theta}_\F, c_{tw}).
$$
This, together with the fact that  both the real part of $\cal T$ and the restriction of $\Theta_\F$ to $\F^{c_\F}$ are spin (i.e. have vanishing first and second Stiefel-Whitney classes)  implies that a twisted orientation $\mo_\F=(\Theta_\F, \psi_\F, \mathfrak s_\F)$ exists on the Real bundle $\T\ra \F$ for any family $\cal F$ as in  \eqref{family.smoothings}.

\begin{rem} When $\Si_0$ has no real special points, one can also show that every choice of twisted orientation data on the relative tangent bundle of the normalization $\wt \Si$ descends to one on the nodal curve $\Si_0$ and extends to one on the relative tangent bundle $\cal T \ra \F$ of the family of deformations of $\Si_0$. 
\end{rem}
\subsection{Orienting the moduli spaces}\label{S.or.mod} The considerations above are used to orient the real relative moduli spaces as described after \cite[(A.13)]{GI}, and in fact to simultaneously orient all the moduli spaces considered in \S\ref{S.VFC.mod}-\S\ref{S.cobord}.  For that it suffices to construct a homotopy class of (global) trivializations of the determinant bundle of the linearization over the union of virtual codimension at most 1 strata for each one of these moduli spaces over their parameter space.

Specifically, assume $X=\Si$ is a (smooth) marked symmetric curve, and consider the union $\wt \fM(\Si)$ of the virtual codimension at most 1 strata of the real relative moduli space associated to $\Si$.  Over the space of Ruan-Tian perturbations, identify the (relative) orientation sheaf of 
$\wt \fM(\Si)$ with 
\bear\label{or.sheaf.2}
\mo_{\wt \fM(\Si)} \cong \det \del_{(\T_\Si, c_\T)} \otimes \mathfrak{f}^* \det T \oM^{\R}_{\chi, \ell(\vec \mu)}
\eear
cf. \eqref{or.sheaf}. It suffices to describe a homotopy class of trivializations of the real line bundle \eqref{or.sheaf.2}. A choice of twisted orientation data 
$\mo=(\Theta, \psi, \mathfrak{s})$ on the relative tangent bundle $\T_\Si$ of 
$\Si$ gives rise to a canonical homotopy class of trivializations 
\bear\label{T.o.iso} 
\begin{tikzcd} 
		\det \del_{(\T_\Si, c_\T)}\otimes \det\del_{(\cx, c_{std})} \arrow{r}{\mo}[swap]{ \cong}&\underline{\R} 
\end{tikzcd}
\eear
cf. \eqref{V.o.iso} with $W=\T_\Si$. Combining this with the canonical homotopy class of isomorphisms  \cite[(3.1)]{gz} 
\bear\label{det.DM} 
\mathfrak{f}^* \det T \oM^{\R}_{\chi, \ell}=  \det \del _{(\cx, c_{\mathrm{std}})} 
\eear
induces a trivialization of \eqref{or.sheaf.2}, canonically up to homotopy. 
\smallskip

When $\Si=\Si_0$ has a pair of conjugate nodes the procedure above also applies to the moduli space $\wt \fM(\Si_0)$ by trivializing  the line bundle 
\bear\label{or.sheaf.node.8}
\mo_{\wt \fM(\Si_0)}\cong \det \del_{(\T_{\Si_0}, c_\T)} \otimes 
	\mathfrak{f}^* \det T\oM^{\R}_{\chi, \ell(\vec \mu)} , 
\eear
cf. \eqref{or.sheaf.node}. In this case a choice of twisted orientation data $\mo_0=(\Theta_0, \psi_0, \mathfrak{s}_0)$ on $\cal{T}_{\Si_0}$ as in \S\ref{S.tw.exist} determines by \cite[Proposition 4.3]{GZ2} a canonical homotopy class of isomorphisms as in \eqref{identif.f} and consequently as in \eqref{V.o.iso} with $W=\cal{T}_{\Si_0}$ and $L=\Theta_0$. 

\smallskip
This procedure also extends to {\em simultaneously} induce trivializations of the determinant line bundles of the moduli spaces in \eqref{fam.moduli.consider.td}, \eqref{def.q.0} and \eqref{def.q} as follows.

Fix a choice  of twisted orientation data $\mo_\F=(\Theta_\F, \psi_\F, \mathfrak{s}_\F)$ for the bundle $(\T, c_\T)$  over the family $(\F, c_\F)$ as in \S\ref{S.tw.exist}.  The pullback under \eqref{si.to.F} of $\T$ to $\Si_s$ and $\wt \Si$ gives their corresponding relative tangent bundles. Furthermore $\mathfrak o_\F$ pulls back to a twisted orientation data
\bear\label{6.mo.pull}
\mo_s=\iota_s^*\mo_\F\quad \text{on} \quad \T_{\Si_s}\lra\Si_s\quad \text{and respectively} \quad \wt \mo=\phi^*\mo_\F
\quad \text{on} \quad  \wt\T \lra\wt \Si.
\eear
Applying the procedure in \eqref{or.sheaf.2}-\eqref{det.DM} for $\mo_s$ and $\wt \mo$ to induce a canonical (relative) orientation on each of the families of moduli spaces 
\bear\label{7.mod.to.orient}
\wt \fM(\Si_{s\ne 0}), \quad \wt \fM(\Si_0), \quad \mbox{ and } \quad \wt \fM(\wt \Si)
\eear
over $\cal P$ defined in \eqref{fam.moduli.consider.td};  note that for $s=0$ this procedure is applied to orient \eqref{or.sheaf.node.8},  and it is the same as the one described in the paragraph containing it.  In particular, a choice of $\mo_\F$ determines a canonical orientation on the generic fibers $\wt \fM(-)_\nu$ of the moduli spaces in \eqref{7.mod.to.orient}.
\medskip

Consider next the families
\bear\label{6.def.q}
\wh \fM(\F_{/ I}) \ra \wt \fM(\F_{/ I})
\eear
over $\cP$ defined in \eqref{def.q}, and recall that by construction they restrict to a covering 
\bear\label{6.def.q.0}
q_0: \wh \fM(\Si_0) \ra \wt \fM(\Si_0) 
\eear 
over $s=0$. By Corollary~\ref{C.or.cob} the relative orientation sheaf of $\wh \fM(\F_{/ I})$ is canonically isomorphic to the pullback by  \eqref{6.def.q} of 
\bear\label{or.sheaf.cob.2}
\det \del_{(\T, c_\T)} \otimes 
\mathfrak{f}^* \det T \oM^{\R}_{\chi, \ell(\vec \mu)} \otimes \pi^* \det T I.  
\eear
Moreover, after orienting the segment $I$ from 0 to $s_0$, the line bundle \eqref{or.sheaf.cob.2} is canonically oriented by a choice of twisted orientation data $\mo_\F$ using the same procedure of combining \eqref{det.DM}  with \eqref{V.o.iso} for $W=\T$ over the moduli space $\wt \fM(\F_{/ I})=\cup_{s\in I} \wt \fM(\Si_s)$. Thus such a choice canonically orients the cobordism $\wh \fM(\F_{/ I})_\nu$ for generic 
$\nu$, and therefore also its boundary, in such a way that 
\begin{enumerate}[(i)]
\item $\wh \fM(\F_{/ I})_\nu$ is an oriented cobordism from the cover 
$\wh \fM(\Si_0)_\nu$ of $\wt \fM(\Si_0)_\nu$ to 
$\wt \fM(\Si_{s_0})_\nu$, and
\item the orientation on $\wh \fM(\Si_0)_\nu$ is the pullback of one on $\wt \fM(\Si_0)_\nu$.
\end{enumerate}
Furthermore, these  orientations on $\wt \fM(\Si_0)_\nu$ and $\wt \fM(\Si_{s_0})_\nu$ agree with those described in the paragraph containing \eqref{7.mod.to.orient} for the same choice of $\mo_\F$. With respect to these orientations, the degree of the covering map $q_0$ is given by Lemma~\ref{C.q.deg}.

\smallskip
However, there is {\em another} natural procedure of fixing a relative orientation of the moduli space $\wt \M(\Si_0)$ over the parameter space $\cP$, which may {\em a priori} lead to a different orientation on its generic fibers than the one described in the paragraph above. As described above Lemma~\ref{L.phi.orient}, the relative  orientation of $\wt \fM(\wt \Si)$ induced by $\wt \mo$ descends to a relative orientation on $\wt \fM(\Si_0)$ under the attaching map
\bear\label{phi.7.att}
\Phi: \wt \fM(\wt \Si) \ra \wt \fM(\Si_0). 
\eear
The degree of this cover with respect to {\em these} orientations is given by Lemma~\ref{L.phi.orient}. 
\subsection{Comparing orientations}\label{S.comp.or} The purpose of this section is to show that the two natural orientations on $\wt \fM(\Si_0)$ described in \S\ref{S.or.mod} (both induced by $\mo_\F$) agree, cf. Proposition~\ref{P.phi.orient.2} below. This is the crucial sign computation in this paper, which fixes the signs that might otherwise enter in the VFC splitting formula \eqref{split.VFC}. Its proof is a rather tedious but straightforward combination of many diagram chases, while keeping careful track at each step of the identifications made in the process of trivializing the determinant bundles of various families of operators.

Recall that as described in  \S\ref{S.mod.nod.lin}, the relative orientation sheaf of the moduli space $\wt \fM(\Si_0)$ has several equivalent descriptions obtained by writing the linearization in different coordinates, cf. Remark~\ref{R.lin.many}. Each of these can be trivialized in more than one way, depending on the choice of the orientation {\em procedure}, leading to {\em possibly different} orientations on the moduli space 
$\wt \fM(\Si_0)$. However, the difference between any two such resulting orientations can be, at least in principle, specifically determined, and our goal here is to compare two such specific orientations. In practice, one way to do so is by describing a natural family of identifications between the two families of linearizations ({\em de facto} describing the linearization of the identity in two types of charts on the moduli space) and then calculating the sign of the resulting isomorphism at the level of determinant bundles with respect to an {\em a priori} specified procedure of orienting each one of them (independent of the identification between the two linearizations). 
\smallskip

We start with some preliminary considerations comparing the pullback under $\Phi$ of the determinant bundle of a family of operators to the determinant bundle of the pullback family over the normalization. 

Assume $(E, c_E)$ is Real vector bundle over $\Si_0$ and let $(\wt E, c_{\wt E})$ denote its pullback to the normalization $\wt \Si$ of $\Si_0$. Pulling back $\del_{(E, c_E)}$ and $\del_{(\wt E, c_{\wt E})}$ to $\wt \fM(\Si_0)$ and respectively $\wt \fM(\wt \Si)$ gives rise to two families of operators over these moduli spaces; let $\det \del_{(E, c_E)}$ and $\det \del_{(\wt E, c_{\wt E})}$ or more precisely 
\bear\label{det.bd.wt}
 \det \; [\del_{(E, c_E)}\ra \wt \fM(\Si_0)]  \quad \mbox{ and } \quad   \det\; [\del_{(\wt E, c_{\wt E})}\ra \wt \fM(\wt \Si)]
\eear
denote their corresponding determinant line bundles. Pulling back the normalization short exact sequence \eqref{SES.E} over the moduli space (and using \eqref{D.wt.les} for the pullback operators) gives an exact sequence 
\bear\label{or.E.pull.SES}
0\ra \det \del_{(\wt E, c_{\wt E})} \ra  \Phi^* \det \del_{(E, c_E)} \ra  \det (\ev^*_{\bf y} E)^\R \ra 0,
\eear
inducing a canonical isomorphism
\bear\label{or.E.pull}
 \Phi^* \det \; \del_{(E, c_E)}=\det\; \del_{(\wt E, c_{\wt E})} \otimes  \det  (\ev^*_{\bf y}E)^\R. 
\eear
Here $\bf y$ denotes the collection of contact nodes of the domains, i.e. those in the preimage of the 
nodes $\bf x$ of $\Si_0$. But recall that $\Si_0$ has a pair of conjugate nodes denoted $x^\pm$ (and no other nodes).  The restriction to $x^+$ induces a canonical identification 
\bear\label{E.pm.fix}
(E_{\bf x}, c_E)^\R \cong E_{x^+}.
\eear
Furthermore, over the moduli space $\wt \fM(\Si_0)$ this induces a canonical identification
\bear\label{E.y=+}
(\ev^*_{\bf y}E)^\R \cong \ev^*_{\bf y^+}(E)= \ev^*_{\bf y^+}(E_{x^+}), 
\eear
where $\bf y^+$ denotes the collection of nodes in the preimage of $x^+$. But $x^+$ is fixed, and $E_{x^+}$ is a complex vector space.  Thus the isomorphism \eqref{or.E.pull}, after using the complex orientation induced by \eqref{E.y=+} on its last term, induces the isomorphism 
\bear\label{or.E.pull.triv}
\det\; [\del_{(\wt E, c_{\wt E})}\ra \wt \fM(\wt \Si)] \cong \Phi^* \det \; [\del_{(E, c_E)}\ra \wt \fM(\Si_0)].
\eear

\medskip

Next recall the relation described in \S\ref{S.mod.nod.lin} between the family of linearizations $D_{f_0}$ given by \eqref{lin.nodal.0} as $f_0$ varies in $\wt \fM(\Si_0)$, and the family of linearizations $D_{\wt f}$ given by \eqref{D.wt.f}, associated to the lifts $\wt f$, as $\wt f$ varies in $\wt \fM(\wt \Si)$. The relation between these two families of linearizations,  described in  Remark~\ref{R.lin.many}, is obtained via intermediate identifications with the linearizations $\wh D_{f_0}$ given by \eqref{lin.nodal} and $\wt D_{f_0}$ given by \eqref{lin.nodal.push}. Note that up to 0'th order terms these are pullback families of operators of the form $\del_{(E, c_E)}\oplus 0_W$, as described in paragraph containing \eqref{det.W}. The natural identifications between their domains, described in Remark~\ref{R.lin.many}, induce the following isomorphisms at the level of the determinant bundles
\bear\label{isom.phi.all}
\Phi^* \det D_{f_0}\ma \cong \Phi^* \det \wh D_{f_0}\ma\cong \Phi^* \det \wt D_{f_0} \ma\cong  \det D_{\wt f}, 
\eear
where
 \bear\label{or.sheaf.node.7}
\det D_{f_0} &\cong&\det \del_{(\T_{\Si_0}, c_\T)} \otimes \det
	\mathfrak{f}^* T\oM^{\R}_{\chi, \ell(\vec \mu)}, 
\\
\label{or.sheaf.node.desc.7}
\det \wh D_{f_0} &\cong&\det \del_{(\T_{\Si_0}, c_\T)} \otimes \det (\T_{\bf y^+})\otimes 
	\det \mathfrak{f}^* T\ov\N_\ell^\R 
\\
\label{or.sheaf.nodal.push.7}
\det\wt  D_{f_0} &\cong&\det \del_{(\T_{\Si_0}, c_\T)}\otimes \det \ev_{\bf y^+}^*(\T_{|x^+}) \otimes \det  \mathfrak{f}^*T\ov\N_\ell^\R.
\\
\label{or.sheaf.wt.7}
\det D_{\wt f}&\cong &\det \del_{(\T_{\wt \Si}, c_\T)}\otimes  \mathfrak f^* \det  T \ov \M^\R_{\chi+4\ell, \ell(\vec\mu)+2 \ell}.
\eear

Specifically, as in Remark~\ref{R.lin.many}, the first identification in \eqref{isom.phi.all}, between the pullbacks of  \eqref{or.sheaf.node.7} and  \eqref{or.sheaf.node.desc.7}, is induced by the decomposition  \eqref{tan-norDM}  into tangent and normal directions to the nodal stratum $\N_\ell$ (with $\ell=\ell(\la)$ pairs of conjugate nodes), cf. \eqref{T.DM} and \eqref{SES.T.h1}, which induces the isomorphism
 \bear\label{7.SES.DM}
\mathfrak{f}^* \det T\oM^{\R}_{\chi, \ell(\vec \mu)}\cong  \det (\T_{\bf y^+})\otimes \mathfrak{f}^* \det T\ov\N_\ell^\R. 
\eear
Here $T \ov \N_{\ell}^\R$ parametrizes those variations in $C_0$ that do not smooth the nodes. Note that its pullback under the attaching map is $T \ov\M_{\chi+4\ell, \ell(\vec\mu)+2 \ell}$ which parametrizes the variations in the normalization $\wt C$ of $C_0$, so
\bear\label{phi.pull.T.nor}
\Phi^*   \mathfrak f^* \det T  \ov \N_{\ell}^\R \cong \mathfrak f^* \det  \ov \M^\R_{\chi+4\ell, \ell(\vec\mu)+2 \ell}. 
\eear
The pullback of \eqref{7.SES.DM} combines with \eqref{phi.pull.T.nor} to give
\bear\label{7.pull.N=M}
\Phi^*\mathfrak f^* \det T\ov \M^\R_{\chi, \ell(\vec\mu)} \cong  \det \T_{\bf y^+} \otimes  \mathfrak f^*\det T\ov \M^\R_{\chi+4\ell, \ell(\vec\mu)+2 \ell}. 
\eear
The second identification in \eqref{isom.phi.all}, between the pullbacks of \eqref{or.sheaf.node.desc.7} and \eqref{or.sheaf.nodal.push.7} is induced by a complex linear isomorphism -- the linearization of \eqref{lead.coef.iso} -- cf. Remark~\ref{R.lin.many}. Finally, the last identification in \eqref{isom.phi.all},  between \eqref{or.sheaf.wt.7} and the pullback of \eqref{or.sheaf.nodal.push.7}, comes from the isomorphism \eqref{phi.pull.T.nor} and the isomorphism 
\bear\label{7.ses.t}
 \det \del_{(\T_{\wt \Si}, c_\T)}\; \ma\cong_{\eqref{or.E.pull}} \;\Phi^* \det \del_{(\T_{\Si_0}, c_\T)}\otimes \det \ev_{\bf y^+}^*(\T_{|x_+})
\eear
induced by the normalization short exact sequence for $\T_{\Si_0}$, cf. Remark~\ref{R.lin.many}.

\medskip

With these preliminary considerations understood, the main result in this section is the following. 
\begin{prop}\label{P.phi.orient.2}  Fix a choice of orientation data $\mo_\F$ on the family $\F$, and orient $\wt \fM(\wt \Si)$ and $\wt \fM(\Si_0)$ by the procedure outlined in the paragraph containing \eqref{7.mod.to.orient}. Then the attaching map $\Phi$ in \eqref{phi.7.att} is orientation preserving with respect to these orientations. 
\end{prop} 
\begin{proof} Orient $\wt \fM(\Si_0)$ by identifying its orientation sheaf with \eqref{or.sheaf.node.7} and trivializing the latter by combining
\bear\label{7.isom.on.nodal}
\det \del_{(\T_{\Si_0}, c_\T)} \; \ma \cong _{\eqref{T.o.iso}}^{\mo_0}\; \det \del_{(\cx, c_{std})}  \quad\mbox{ and }  \quad 
\mathfrak{f}^* \det T\oM^{\R}_{\chi, \ell(\vec \mu)}  \; \ma \cong _{\eqref{det.DM}}\;  \det \del_{(\cx, c_{std})}.  
\eear
We must compare the pullback under $\Phi$ of this orientation with the orientation of $\wt \fM(\wt \Si)$ obtained by trivializing \eqref{or.sheaf.wt.7} using 
\bear\label{7.isom.on.res}
\det \del_{(\T_{\wt \Si}, c_\T)} \; \ma \cong _{\eqref{T.o.iso}}^{\wt \mo} \; \det \del_{(\cx, c_{std})} \quad\mbox{ and }  \quad 
\mathfrak{f}^* \det T\oM^{\R}_{\chi+4\ell, \ell(\vec \mu)+2\ell}  \; \ma \cong _{\eqref{det.DM}}\;  \det \del_{(\cx, c_{std})}.   
\eear
\smallskip

This means calculating the sign, with respect to these orientations, of the isomorphism 
\bear\label{7.pull.0=4}
\Phi^*( \det \del_{(\T_{\Si_0}, c_\T)}  \otimes \mathfrak{f}^* \det T\oM^{\R}_{\chi, \ell(\vec \mu)} ) \cong 
\det \del_{(\T_{\wt \Si}, c_\T)} \otimes \mathfrak{f}^* \det T\oM^{\R}_{\chi+4\ell, \ell(\vec \mu)+2\ell} 
\eear
obtained from the composition \eqref{isom.phi.all} described above. The discussion after \eqref{isom.phi.all} implies that \eqref{7.pull.0=4} is the composition 
\begin{align*}
 \Phi^* \det \del_{(\T_{\Si_0}, c_\T)}  \otimes \Phi^* \mathfrak{f}^* \det T\oM^{\R}_{\chi, \ell(\vec \mu)} 
\ma \cong_{\eqref{7.pull.N=M}}&
\Phi^* \det \del_{(\T_{\Si_0}, c_\T)} \otimes   \det (\T_{\bf y^+})\otimes \mathfrak{f}^*  \det T\oM^{\R}_{\chi+4\ell, \ell(\vec \mu)+2\ell} 
\\
\ma\cong_{\eqref{lead.coef.iso}} \; &
\Phi^* \det \del_{(\T_{\Si_0}, c_\T)} \otimes   \det \ev_{\bf y^+}^*(\T_{|x_+})  \otimes \mathfrak{f}^*  \det T\oM^{\R}_{\chi+4\ell, \ell(\vec \mu)+2\ell}
\\
\ma \cong_{\eqref{7.ses.t}} &\det \del_{(\T_{\wt \Si}, c_\T)} \otimes \mathfrak{f}^* \det T\oM^{\R}_{\chi+4\ell, \ell(\vec \mu)+2\ell}.  
\end{align*}
The second isomorphism comes from the complex identification induced by the linearization of \eqref{lead.coef.iso}.

We are now ready to calculate to sign of the isomorphism \eqref{7.pull.0=4}; we do this in several steps. In addition to \eqref{7.pull.N=M} and \eqref{7.ses.t} we will also use the identification
\bear\label{7.pull.ses.cx}
\Phi^* \det \del_{(\cx, c_{std})} \otimes  \det   \ev^*_{\bf y_+}(\cx)\;  \ma \cong_{  \eqref{or.E.pull}} \; \det \del_{(\cx, c_{std})} 
\eear
induced by the normalization short exact sequence for $E=\cx$. 

\medskip
\textsc{ Step 1.} Consider first the tensor product of the isomorphisms \eqref{7.pull.N=M} and \eqref{7.pull.ses.cx}; it corresponds to \cite[(4.41)]{GZ2} (taking into account the change of notation).  After using the complex orientation on $\T_{\bf y^+}$ and  $\ev^*_{\bf y_+}(\cx)$ it induces the isomorphism 
\bear\label{7.isom.pull.DM=C}
\Phi^* \big(\det \del_{(\cx, c_{std})}   \otimes  \mathfrak{f}^* \det T\oM^{\R}_{\chi, \ell(\vec \mu)} \big) \cong  \det \del_{(\cx, c_{std})}   \otimes  
\mathfrak f^* \det T\ov \M^\R_{\chi+4\ell, \ell(\vec\mu)+2 \ell}.
\eear
Both sides are oriented via the canonical orientation on the tensor product of the terms in \eqref{det.DM}, i.e. by the second isomorphism in \eqref{7.isom.on.nodal} and respectively \eqref{7.isom.on.res}. By \cite[Proposition 4.18]{GZ2}, the sign of the  isomorphism \eqref{7.isom.pull.DM=C}, with respect to these orientations, is given by the number mod 2 of pairs of conjugate nodes of the domain i.e. is equal to $(-1)^{\ell}$. 
\smallskip

\textsc{ Step 2.} Consider next the tensor product of the isomorphisms \eqref{7.ses.t} and  \eqref{7.pull.ses.cx} corresponding to the normalization short exact sequences for $E=\T_{\Si_0}$ and respectively $E=\cx$. After using the complex orientation on $\ev_{\bf y^+}^*(\T_{|x_+})$ and $\ev^*_{\bf y_+}(\cx)$ it induces the isomorphism 
\bear\label{7.isom.pull.C=T}
\Phi^*  \big(\det \del_{(\T_{\Si_0}, c_\T)} \otimes \det \del_{(\cx, c_{std})}  \big) \ma\cong_{\eqref{or.E.pull.triv}}  \det \del_{(\T_{\wt \Si}, c_\T)}   \otimes \det \del_{(\cx, c_{std})}. 
\eear
Both sides are oriented via the canonical trivialization \eqref{T.o.iso} for $\mo_0$ on $\T_{\Si_0}$ and 
$\wt \mo$ on $\T_{\wt \Si}$, i.e. by the first isomorphisms in \eqref{7.isom.on.nodal} and respectively \eqref{7.isom.on.res}. By Lemma~\ref{C.com.sign} below, the sign of the isomorphism \eqref{7.isom.pull.C=T}, with respect to these orientations is also $(-1)^{\ell}$.
\smallskip

\textsc{ Step 3.} Finally, consider the tensor product of two copies of the isomorphism \eqref{7.pull.ses.cx}; after using the complex orientation on 
$\ev^*_{\bf y_+}(\cx)$ twice, it induces the isomorphism 
\bear\label{7.isom.pull.C=C}
\Phi^*   \big(\det \del_{(\cx, c_{std})} \otimes \det \del_{(\cx, c_{std})}  \big)  \ma\cong_{\eqref{or.E.pull.triv}} \det \del_{(\cx, c_{std})} \otimes \det \del_{(\cx, c_{std})}. 
\eear
Both sides are oriented via the canonical trivialization of twice of a bundle; with respect to these orientations, \eqref{7.isom.pull.C=C} is orientation preserving since the normalization short exact sequence is natural and the complex orientation on $\ev^*_{\bf y_+}(\cx\oplus \cx)$ agrees with the one induced by regarding it as twice of a bundle. 
\smallskip

The proof is completed by observing that the tensor product of \eqref{7.isom.pull.DM=C} and \eqref{7.isom.pull.C=T} agrees (up to homotopy) with 
the tensor product of  \eqref{7.isom.pull.C=C} and \eqref{7.pull.0=4} and the orientations discussed above also agree. Thus the sign of \eqref{7.pull.0=4} with respect to the orientations induced 
by \eqref{7.isom.on.nodal} and respectively \eqref{7.isom.on.res} is $(-1)^{\ell} \cdot (-1)^{\ell}=+1$. 
 \end{proof}

We conclude this section with two results on the behavior of the trivialization \eqref{V.o.iso} under pullback to the normalization. A special case of the following result was used in Step 2 above. 
\begin{lemma}\label{C.com.sign} Assume $\mo_0$ is a twisted orientation on a Real line bundle $(W,\phi)$ over the nodal symmetric marked curve $\Si_0$, and let $\wt \mo$ denote its pullback, which is a twisted orientation on the pullback bundle $(\wt W, \wt \phi)$ over the normalization $\wt \Si$. Consider the isomorphism 
\best
\Phi^*(\det \del_{(W, \phi)}\otimes  \det \del_{(\cx, c_{std})} ) \; \ma\cong_{\eqref{or.E.pull.triv}} 
\det \del_{(\wt W, \wt \phi)} \otimes \det \del_{(\cx, c_{std})} 
\eest
where the determinant line bundles are as in \eqref{det.bd.wt}. 
Orient both sides by \eqref{V.o.iso} for $\mo_0$ and respectively $\wt \mo$. Then, with respect to these orientations,  
this isomorphism has sign $(-1)^\ell$, where $\ell$ is the number of conjugate pairs of nodes of the domain.
\end{lemma} 
\begin{proof} The isomorphism \eqref{V.o.iso} is obtained by combining the isomorphisms \eqref{identif.f}-\eqref{det.E.L} with the canonical orientation on the square of a bundle. Therefore it suffices to understand how each one of these isomorphisms  behaves under pullback by $\Phi$. 
\medskip

First of all, when $(E, c_E)=(2L, 2c_L)$ is twice a Real bundle, the isomorphism 
\bear\label{phi.7.E}
\Phi^*\det \del_{(E, c_E)}\; \ma\cong_{\eqref{or.E.pull.triv} } \;\det \del_{(\wt E, c_{\wt E})} 
\eear
is orientation preserving with respect to orienting both $\det \del_{(E, c_E)}$ and  $\det \del_{(\wt E, c_{\wt E})}$ as a square of a real bundle. This follows by the naturality of the normalization short exact sequence and the fact that both the pullback of $E$ and the restriction of $E$ to $\bf y^+$ is twice a bundle. Moreover, the complex orientation on twice of a complex bundle agrees with the canonical orientation on twice a bundle. 
\smallskip

We next focus on the isomorphism \eqref{can.iso.V}, which is induced by the canonical homotopy class of isomophisms \eqref{identif.f}. But  \eqref{identif.f} pulls back to the normalization to give (up to homotopy) the isomorphism \eqref{identif.f} for the pullback structure; it also restricts to the nodes $\bf y^+$ as a complex linear isomorphism (unique up to homotopy). Therefore the diagram 
\best
\begin{tikzcd}
\Phi^*\det \del_{(W, \phi)\oplus (E, c_{E})^\vee} \arrow{r}[equal]{\mo_0}[swap] {\eqref{can.iso.V}} \arrow{d}[swap]{\eqref{or.E.pull.triv}}
&\Phi^* \det \del_{(\cx^{n+2}, c_{std})}  \arrow{d} {\eqref{or.E.pull.triv}}
\\
\det \del_{(\wt W, \wt \phi)\oplus (\wt E, c_{\wt E})^\vee}\arrow{r}[equal]{\wt \mo}[swap] {\eqref{can.iso.V}} &\det \del_{(\cx^{n+2}, c_{std})} 
\end{tikzcd}
\eest
commutes up to homotopy. Equivalently, consider the isomorphism 
\bear\label{phi.7.triv}
\Phi^*(\det \del_{(W, \phi)}\otimes   \det \del_{(E, c_{E})^\vee } \otimes  \det \del_{(\cx^{n+2}, c_{std})} ) 
\ma=_{\eqref{or.E.pull.triv}}  \det \del_{(\wt W, \wt \phi) } \otimes \det \del_{(\wt E, c_{\wt E})^\vee} \otimes 
\det \del_{(\cx^{n+2}, c_{std})} 
\eear
induced by the normalization short exact sequence and the complex orientation on the restriction of the bundles $W$, $E$ and $\cx^{n+2}$ to the nodes $\bf y^+$.  Orient both sides of \eqref{phi.7.triv} by the canonical isomorphism induced by the tensor product of the two sides of \eqref{can.iso.V} for $\mo_0$ and respectively $\wt \mo$. Then with respect to these orientations, the isomorphism \eqref{phi.7.triv} is orientation preserving.
\smallskip

It remains to understand the behavior of the isomorphism \eqref{det.E.L} under pullback. For that, consider the isomorphism \eqref{or.E.pull.triv} for
$(E, c_E)=(L\oplus c^*\ov L, c_{tw})$. By Lemma~\ref{C.twist.dif} below the sign of this isomorphism, with respect to the complex orientations induced by \eqref{det.E.L}, is the number mod 2 of pairs of conjugate nodes of the domain i.e. is equal to $(-1)^{\ell}$. This completes the proof. 
\end{proof} 

\begin{lemma}\label{C.twist.dif} Assume $L$ is a complex line bundle over $\Si_0$. Let $(E, c_E)= (L\oplus c^*\ov L, c_{tw})$ and denote by $\wt E$ its pullback to the normalization $\wt \Si$. Consider the isomorphism
\best
\Phi^*\det \del_{(E, c_E)} \; \ma\cong_{\eqref{or.E.pull.triv}} \det \del_{(\wt E, c_{\wt E})}
\eest
where the determinant line bundles are as in \eqref{det.bd.wt}. Orient both sides by the complex orientation induced by the isomorphism \eqref{det.E.L} for $E$ and respectively $\wt E$. Then with respect to these orientations, this isomorphism has sign $(-1)^{\ell}$, where $\ell$ is the number of pairs of conjugate nodes of the domain. 
\end{lemma} 
\begin{proof} The isomorphism \eqref{or.E.pull.triv} is induced by the normalization short exact sequence \eqref{SES.E} together with the identification \eqref{E.y=+}. Consider also the normalization short exact sequence 
\best
0\lra L\lra \wt L   \lra L_{{\bf x}} \lra 0.
\eest
for the complex line bundle $L$.  Compare the pullback of the short exact sequences for $E$ and $L$ with the isomorphisms \eqref{det.E.L} for $E$ and $\wt E$ obtained by projecting onto the first factor. We would get a commutative diagram if we used the corresponding identification
\best
(\ev^*_{\bf y}E )^\R \ma\cong_{p_1^*} \ev^*_{\bf y}L = \ev^*_{\bf y^+}L\oplus \ev^*_{\bf y^-}L. 
\eest
However, we are using instead the identification \eqref{E.y=+} induced by the restriction to ${\bf y^+}$, i.e. 
\best
(\ev^*_{\bf y}E )^\R \cong  \ev^*_{\bf y^+}E =\ev^*_{\bf y^+}(L\oplus c^*\ov L) = \ev^*_{\bf y^+} L\oplus \ev^*_{\bf y^-} \ov L.
\eest
The difference between the complex orientations of the last two displayed equations is $(-1)^{\ell}$, coming from the rank $\ell$ of the complex bundle $\ev^*_{\bf y^-}L$. 
\end{proof} 
\smallskip

\section{Proof of the VFC Splitting Theorem \ref{T.splitVFC}}\label{S.pf.splitVFC}
\medskip

Let $\Si_0$ be a nodal symmetric curve with $r$ pairs of conjugate marked points $V_0$ as in \eqref{marked.points} and a pair of conjugate nodes $x^\pm$ (and no other special points). We first assume that $\Si_0$ has no rational connected components without marked points; this assumption is removed at the end of this section, cf. Remark~\ref{R.VFC.tori}.

\smallskip

Consider the family $\cal F=\cup_s\Si_s $ of targets from \S\ref{S.fam.curves}. Fix a parameter $s_0\ne 0$ and let $I$ denote the segment 
$[0, s_0]\subset \De$. Fix also the degree $d$, the Euler characteristic $\chi$, and the ramification profile $\vec \mu=(\mu^1, \dots, \mu^r)$ and consider the families
\best
\ov \fM(\F_{/ I})=\ma \cup_{s\in I} \ov \fM (\Si_s) \quad \mbox{ and } \quad \ov \fM(\wt \Si)
\eest
of real relative moduli spaces defined in \eqref{mo.over.fam}. Recall that these are families over the parameter space $\cP=\JV(\F_{/\De})$ of RT-perturbations $\nu$, whose fibers $\ov \fM(-)_\nu$ at $\nu$ are respectively 
\bear\label{M.frak.7}
\ov \fM(\F_{/ I})_{\nu}= \ma \cup_{s\in I} \ov \fM^\R_{d, \chi, \vec \mu}(\Si_s, V_s)_\nu  \quad \mbox{ and } \quad 
\ov \fM(\wt \Si)_{\nu} =  \ma\sqcup_{\la \vdash d} \ov \fM^\R_{d, \chi+4\ell(\la), \vec \mu,\la, \la}(\wt \Si, \wt V)_\nu. 
\eear
By construction we have proper maps
\best
\begin{tikzcd}
\ov \fM(\wt \Si)  \ar[r, "\Phi"]&\ov \fM(\Si_0) & \mbox{ and } &  \ov \fM(\Si_s) \ar[r, hook]& 
\ov \fM(\F_{/ I})=\ma \cup_{s\in I} \ov \fM(\Si_s)
\end{tikzcd}
\eest
for all $s\in I$, as in \eqref{map.attach.0} and \eqref{map.include} (which correspond to $\nu=0$). 
Moreover, restricting to the open subsets $\wt \fM(-)\subseteq \ov\fM(-)$, cf. \eqref{fam.moduli.consider.td}, gives rise to proper maps 
\best
\begin{tikzcd}
\wt \fM(\wt \Si) \ar[r, "\Phi|"]&\wt \fM (\Si_0) & \mbox{ and } &  \wt \fM(\Si_s) \ar[r, hook]& \wt \fM(\F_{/ I}) \ma=^{\rm def}\ma\cup_{s\in I} \wt \fM(\Si_s). 
\end{tikzcd}
\eest
Notice that throughout this argument we must work with proper maps in order to push forward classes in rational Cech homology (e.g. maps that extend continuously to some compactification). 
\smallskip

Fix an orientation data $\mo_\F$ on the family $\F$ and let $\mo_s$ and respectively $\wt \mo$ be its pullback to $\Si_s$ and $\wt \Si$. 

\begin{proof}[\bf \em  Proof of the VFC Splitting Theorem \ref{T.splitVFC}]
 Assume that $\Si_0$ has no rational connected components without any marked points. Then by  Theorem~\ref{T.VFC.rel.2}, $\ov \fM(\Si_{s_0})$, $\ov \fM(\Si_0)$ and $\ov \fM(\wt \Si) $ are thinly compactified families over the parameter space $\cP =\JV(\F_{/\De})$. In particular, for generic $\nu$, the fibers of these families are thin compactifications of topological manifolds, cf. Proposition~\ref{P.or.mfl}. The latter are canonically oriented by the choice of $\mo_\F$ as described in the paragraph containing \eqref{7.mod.to.orient}.  
Our goal is to compare the images of the fundamental classes 
\best
[\wt \fM(\Si_{s_0})_\nu], \quad \Phi_*[\wt \fM(\wt \Si)_\nu], \quad \mbox{ and } \quad 
[\wt \fM(\Si_{0})_\nu]\qquad \mbox{ in } \quad \cHH_b(\wt \fM(\F_{/ I})_\nu; \Q) 
\eest
of these manifolds (for generic $\nu$). Here $b$ is the dimension of these moduli spaces, cf.  \eqref{dim.M=b}.
\smallskip

For generic $\nu$, by Lemma~\ref{L.phi.orient} and Proposition~\ref{P.phi.orient.2}, the attaching map $\Phi$ restricts to a finite degree orientation preserving proper map between the oriented topological manifolds $\wt \fM(\wt \Si)_\nu$ and $\wt\fM(\Si_0)_\nu$; each of these spaces decomposes as a disjoint union of open and closed subsets indexed by $\la$. Moreover by \eqref{Phi.0.push}
\bear\label{7.deg.phi}
[\wt \fM_\la (\Si_{0})_\nu]  = \frac {1}{|\Aut \; \la|} \cdot \Phi_*[ \wt \fM_\la (\wt \Si)_\nu]
\eear
in $\cHH_b(\wt \fM_\la(\Si_0)_\nu; \Q)$ and therefore in  $\cHH_b(\wt \fM(\Si_0)_\nu; \Q)$. 
\medskip

Consider next the auxiliary spaces $\wh \fM(\Si_0)$ and  $\wh \fM(\F_{/ I})$  constructed in \S\ref{S.cobord}. They are also families over the parameter space $\cP$ and come with proper 
maps 
\best
\begin{tikzcd}
\wh \fM(\Si_0) \ar[r, hook]  \ar[d, swap, "q_0"]&  \wh \fM(\F_{/ I}) \ar[d, "q"]
\\
\wt \fM(\Si_0) \ar[r, hook]  & \wt \fM(\F_{/ I})
\end{tikzcd}
\eest
cf. \eqref{def.q.0} and \eqref{def.q}. 
But for generic $\nu$, by Proposition~\ref{P.cob} the fiber $\wh \fM(\F_{/ I})_\nu$ of $\wh \fM(\F_{/ I})$ is a cobordism between the manifold $\wt  \fM(\Si_{s_0})_\nu$ and the cover 
$\wh \fM(\Si_0)_\nu$ of the manifold $\wt \fM(\Si_0)_\nu$. This cobordism is also canonically oriented by the choice $\mo_\F$ and the restriction of this orientation to the boundary agrees with the one above, cf. the paragraph containing \eqref{or.sheaf.cob.2}. Thus
\bear\label{7.cob.bd}
 [ \wt \fM(\Si_{s_0})_\nu]  = [\wh \fM(\Si_0)_\nu] =\sum_\la  [ \wh \fM_\la (\Si_{0})_\nu]
\eear
in $\cHH_b(\wh \fM(\F_{/ I})_\nu; \Q)$, where the first equality is  \eqref{wh.cob.eq}, and the second equality holds because, as in \eqref{wh.M}, the manifold $\wh \fM(\Si_{0})_\nu$ is a disjoint union of open and closed subsets indexed by $\la$. 

Furthermore, by Lemma~\ref{C.q.deg} 
\bear\label{7.deg.q}
(q_0)_*[\wh \fM_\la (\Si_0)]_\nu= \frac{\zeta(\la)}{|\Aut \;\la|}\cdot [\wt \fM_\la (\Si_0)_\nu] 
\eear
in $\cHH_b(\wt \fM(\Si_0)_\nu; \Q)$.
\smallskip

Pushing forward \eqref{7.cob.bd} by the proper map $q$ and combining it with the pushforward of \eqref{7.deg.q} and \eqref{7.deg.phi} under the proper inclusion $\wt \fM(\Si_0)_\nu \hookrightarrow \wt \fM(\F_{/ I})_\nu$ gives the equality 
\bear\label{7.vfc.tilde}
[ \wt \fM(\Si_{s_0})_\nu] = \sum_\la \frac{\zeta(\la)}{|\Aut \;\la|}\cdot [\wt \fM_\la (\Si_0)_\nu] = 
\sum_\la \frac{\zeta(\la)}{|\Aut \;\la|^2}\cdot \Phi_*[\wt \fM_\la (\wt\Si)_\nu] 
\eear
in $\cHH_b(\wt \fM(\F_{/ I})_\nu; \Q)$; this holds for generic $\nu$ in $\cP$. 
\medskip

It remains to show that the equality \eqref{7.vfc.tilde} lifts to an equality between the fundamental classes of the thin compactifications 
(for generic $\nu$) and thus uniquely extends to an equality between their VFCs for all $\nu$. 

For that, as in \eqref{wt.S}, consider the difference
\best
{\cal S}_\nu=\ov \fM(\F_{/ I})_\nu\setminus \wt \fM(\F_{/ I})_\nu 
\eest
which is a closed subset of $\ov \fM(\F_{/ I})_\nu$. For generic $\nu$, by Proposition~\ref{P.or.mfl}, it has homological dimension at most $(b+1)-2=b-1$,  i.e. $\cHH_{*}({\cal S}_\nu; \Q)$=0 for all $*>b-1$. As in the proof of \cite[Lemma~2.10]{ip-thin}, the long exact sequence \eqref{LES.Cech} in rational Cech homology
\best
\dots \ra \cHH_{*}({\cal S}_\nu;\Q) \ra \cHH_*(\ov\fM(\F_{/ I})_\nu; \Q) \ma\ra^\rho \cHH_*(\wt \fM(\F_{/ I})_\nu; \Q)\ra \dots 
\eest
associated to the closed pair $(\ov \fM(\F_{/ I})_\nu, {\cal S}_\nu)$ then implies that $\rho$ is {\em injective} in dimension  $*=b$. Therefore \eqref{7.vfc.tilde} implies that for generic $\nu$ 
\bear\label{7.vfc.bar}
[ \ov \fM(\Si_{s_0})_\nu] = \sum_\la \frac{\zeta(\la)}{|\Aut \;\la|^2}\cdot \Phi_*[\ov \fM_\la (\wt\Si)_\nu] 
\eear
in $\cHH_b(\ov\fM(\F_{/ I})_\nu; \Q)$ as follows. It suffices to check that $\rho$ maps \eqref{7.vfc.bar} to \eqref{7.vfc.tilde}. This follows from the naturality of the long exact sequence \eqref{LES.Cech}, combined with the facts that by construction (i) the intersection of $\wt \fM(\F_{/ I})_\nu$ with $\ov\fM(\Si_s)_\nu$ is equal to $\wt \fM(\Si_s)_\nu$, (ii) the inverse image of $\wt \fM(\Si_0)_\nu$ under $\Phi$ is equal to $\wt \fM(\wt \Si)_\nu$, and (iii) 
$\ov\fM(\Si_{s_0})_\nu$ and $\ov\fM(\wt \Si)_\nu$ are thin compactifications of the manifolds $\wt \fM(\Si_{s_0})_\nu$ and respectively $\wt \fM(\wt \Si)_\nu$  (for generic $\nu$ in $\cP$) by Proposition~\ref{P.or.mfl}. 
\medskip

Finally, recall that by Theorem~\ref{T.VFC.rel.2}, $\ov \fM(\Si_{s_0})$  and $\ov \fM(\wt \Si) $ are thinly compactified families over $\cP$, thus each fiber carries a VFC as in \eqref{VFC.gen}. 
By \cite[Lemma 3.4]{ip-thin} the relation \eqref{7.vfc.bar} for generic $\nu$ extends uniquely to the corresponding relation between the VFCs for all parameters $\nu\in \cP$, including for $\nu=0$. 
Applying Lemma~\ref{L.VFCs.match} and switching back to the original notation \eqref{M.frak.7} for the fibers at $\nu=0$  gives \eqref{split.VFC}. This completes the proof of Theorem~\ref{T.splitVFC} under the assumption that all rational connected components of $\Si_0$ have at least one marked point. The remaining case is treated in Remark~\ref{R.VFC.tori}  below. 
\end{proof}

\subsection{A relation between the absolute and relative VFCs}\label{S.abs=rel} We end this section with a consequence of the VFC Splitting Theorem, describing how the VFC behaves under adding a pair of conjugate marked points of the target. To state it, we need to first introduce some notation. 
\medskip

Recall that if $\Si$ is a smooth complex curve, there is a flat family $\cal C\ra \cx$ of deformations of $\Si$ whose 
central fiber is the nodal curve 
$\Si\cup \P^1$ with one node. It can be obtained by fixing a point $x\in \Si$ and blowing up the constant family $\Si \ti \cx\ra \cx$ at the point $x \ti 0$. The exceptional divisor is $\P(T_x \Si \oplus \cx)=\P^1$. So the 
restriction of $\cal C$ over $\cx\setminus 0$ is the product family, but $\cal C$ has a nodal fiber $\Si\cup \P^1$ over $s=0$.
\medskip

Assume next that $(\Si, V)$ is a smooth symmetric curve with $r$ pairs of conjugate marked points as in \eqref{marked.points}, and choose ${\bf x}= \{ x^\pm\}$ an additional pair of conjugate points on $\Si$, disjoint from $V$.  Blowing up the product family $\Si \ti \cx^2$ along the submanifold $(x^+ \ti 0 \ti \cx) \sqcup (x^-\ti \cx \ti 0)$ similarly yields a flat family $\cal F$ of deformations as in \eqref{family.smoothings}-\eqref{fam.smooth.real}. Each smooth fiber of $\F_{/\De}=\ma\cup_{s\in \De} \Si_s$ is $\Si_s=\Si$ for all $s\ne 0$, while the central fiber is a nodal curve 
\bear\label{8.nodal.Si}
\Si_0=\Si\cup \P_{\bf x}
\eear 
with a pair of conjugate nodes and an additional pair 
$\P_{\bf x}$ of conjugate spherical components; see also \eqref{def.PV}. The marked points $V_0$ of $\Si_0$ correspond to the marked points $V$ of $\Si$ under the inclusion $\Si \subset \Si_0$. 
\smallskip

Let $p: \Si_0\ra \Si$ denote the map which collapses $\P_{\bf x}$ to $\bf x$ and restricts to the identity on $\Si$.  This induces a collapsing map from any building associated to $\Si_0=\Si \cup \P_{\bf x}$ to one associated to $\Si$ (collapsing the entire chain 
$\P_{\bf x} \ma\cup_{\bf x_0=x_\infty}\P_{\bf x}\ma\cup_{\bf x_0=x_\infty} \dots\ma\cup_{\bf x_0=x_\infty} \P_{\bf x}$ attached at ${\bf x}$, but keeping the chains attached at $V$). 

Therefore $p$ induces a proper continuous map 
 \best
{ \mathfrak p}: \ov\M^\R_{d, \chi, \vec\mu}(\Si_0, V_0) \ra  \ov\M^\R_{d, \chi, \vec\mu}(\Si, V) 
 \eest
 at the level of (unperturbed) moduli spaces.
 \medskip
  
Consider the family of targets described above \eqref{8.nodal.Si}, and let $\mo_\F$ be a choice of twisted orientation data on it. Assume that $(\Si, V)$ has no rational components without any special points; then the same is true for $(\wt \Si, \wt V)$ and $(\Si_0, V_0)$. Therefore the  moduli spaces $\ov \M^{\R}(\Si, V)$,  $\ov \M^{\R}(\Si_0, V_0)$ and $\ov \M^{\R}(\wt \Si, \wt V)$ have a well defined VFC, as defined in \S\ref{S.VFC.mod}, and the VFC Splitting Theorem holds, as proved at the beginning of \S \ref{S.pf.splitVFC}. This has the following consequence. 
\begin{cor}\label{C.vfc.stab}  
For every $d$, $\chi$, and $\vec \mu$, the equality
\bear\label{9.VFC.stab}
	[\ov \M^{\R}_{d, \chi,\vec \mu }(\Si, V)]^{\vir, \mo}
	=\sum_{\la\vdash d}  \frac{\zeta(\la)}  {|\Aut (\la)|^2} \; {\mathfrak p}_*
	\Phi_*[\ov \M^{\R}_{d, \chi+4\ell(\la),\vec \mu, \la,\la }(\wt\Si, \wt V)]^{\vir, \wt\mo}
\eear
holds in the rational Cech homology of $\ov \M^{\R}_{d, \chi,\vec \mu}(\Si, V)$. 
\end{cor}
\begin{proof} By the VFC Splitting Theorem \ref{T.splitVFC}, for every $s\in \De \setminus 0$, we get the equality 
\bear\label{8.vfc.s}
[\ov \M^{\R}_{d, \chi,\vec \mu }(\Si_s, V_s)]^{\vir, \mo_s}
	=\sum_{\la\vdash d}  \frac{\zeta(\la)}  {|\Aut (\la)|^2} \; 
	\Phi_*[\ov \M^{\R}_{d, \chi+4\ell(\la),\vec \mu, \la,\la }(\wt\Si, \wt V)]^{\vir, \wt\mo}
\eear
in the rational Cech homology of the family $ \ma\cup_{s\in \De} \ov\M_{d,\chi,\vec\mu}^\R (\Si_s, V_s)$. It suffices to show that \eqref{8.vfc.s} pushes forward to \eqref{9.VFC.stab} under a suitable proper continuous map.  

The collapsing map $p: \Si_0 \ra \Si$ is the restriction of the composition 
\bear\label{9.proj.F}
\F \ra \Si \ti \De^2\xrightarrow{\mathrm{pr}_\Si} \Si
\eear
where the first arrow collapses the exceptional divisor and the second arrow is the projection onto $\Si$; note that both of these are holomorphic maps, compatible with the real structures. Then the  composition \eqref{9.proj.F} similarly induces a proper continuous map 
\best
\mathfrak{p}: \ma\cup_{s\in \De} \ov\M_{d,\chi,\vec\mu}^\R (\Si_s, V_s)\ra \ov\M_{d,\chi,\vec\mu}^\R (\Si, V)
\eest 
at the level of (unperturbed) moduli spaces, and therefore a map between their rational Cech homology groups. Consider next the pushforward of \eqref{8.vfc.s} under this map. Note that for  $s\ne 0$, the map \eqref{9.proj.F} restricts to the identity $\Si_s=\Si$, therefore 
\best
{\mathfrak p}_*[\ov \M^{\R}_{d, \chi,\vec \mu }(\Si_s, V_s)]^{\vir, \mo_s}= [\ov \M^{\R}_{d, \chi,\vec \mu }(\Si, V)]^{\vir, \mo} 
\eest
for all $s\ne 0$. Thus the pushforward of \eqref{8.vfc.s}  is \eqref{9.VFC.stab}. 
\end{proof}

\begin{rem}\label{R.VFC.tori} So far both the construction of the VFC and the proof the VFC Splitting Theorem were done under the technical assumption that the target curves have no spherical components without any special points. In general, we can always first deform $\Si$ to a nodal curve $\Si_0$ whose normalization $\wt \Si= \Si \sqcup \P_{\bf x}$ has at least one marked point on each spherical component cf. \eqref{8.nodal.Si}, and then {\em define} the VFC of $\ov\M^\R(\Si, V)$ by the formula \eqref{9.VFC.stab}. This is well defined, and the VFC Splitting Theorem \ref{T.splitVFC} {\em automatically} holds in this case. In particular, when $V=\emptyset$,  \eqref{9.VFC.stab} becomes
\best
[\ov \M^{\R}_{d, \chi}(\Si)]^{\vir, \mo}=\sum_{\la\vdash d}  \frac{\zeta(\la)}  {|\Aut (\la)|^2} \; {\mathfrak p}_*
	\Phi_*[\ov \M^{\R}_{d, \chi+4\ell(\la),\la,\la }( \Si \sqcup \P_{\bf x},  {\bf x} \sqcup {\bf x}_\infty)]^{\vir, \wt\mo}.
\eest 
\end{rem}


\section{Proof of the RGW Splitting Theorem~\ref{T.gluing}}\label{S.pf.RGW}
\medskip

We begin with a brief review of the definition of the integrand in \eqref{RGW.chern} and its properties. 
\subsection{Index bundles} Let $E$ be a holomorphic vector bundle over a complex curve $\Si$. The complex operator $\del_E$ determines by pullback a family of complex operators over the moduli spaces of holomorphic maps to $\Si$; the fiber at $f:C\ra \Si$ is the pullback operator $\del_{f^* E}$. Denote by 
\bear\label{def.ind}
\mathrm{Ind}\;  \del_{E}= R^\bullet \pi_*\ev^*(E) 
\eear
the index bundle associated to this family of operators, regarded as an element in K-theory. Here 
$\pi:\cal C\ra \ov \M$ is the universal curve over the moduli space and $\ev: \cal C \ra \Si$ is the evaluation map. When $\Si$ is a marked curve, Bryan and Pandharipande considered the index bundle $\mathrm{Ind}\;  \del_ E$ over the relative moduli space associated to $\Si$, cf. \cite[\S2.2]{bp1} (they denote the evaluation map $\ev$ by $f$).  This index bundle is defined by the same formula \eqref{def.ind}, but now the domain of $\pi$ is the universal curve over the relative moduli space. Let 
\bear\label{c.k}
c_k(- \mathrm{Ind}\;  \del_{E}) 
\eear
denote the corresponding $k$'th Chern class, regarded as an element in the Cech cohomology of the relative moduli space. 
\subsection{The integrand} Let $(\Si, c)$ be a symmetric marked curve with marked points $V$ as in \eqref{marked.points}, and let $L\ra \Si$ be a holomorphic line bundle over the underlying complex curve $\Si$. Fix the topological data $d, \chi$ and $\vec \mu$ and let $b$ be as in \eqref{dim.M=b}. 

Consider the real relative moduli 
space $\ov\M^\R_{d, \chi, \vec \mu}(\Si, V)$. It comes with a forgetful map to the (usual) relative moduli space associated to the complex marked curve $\Si$. Denote by 
\bear\label{integrand}
I(\Si; L)=c_{b/2}(-\mathrm{Ind} \; \del_{L}) \in \cHH^b(\ov \M_{d, \chi, \vec \mu}^{\R}(\Si, V); \Q)
\eear
the pullback of the corresponding Chern class \eqref{c.k} on the usual (complex) relative moduli space. This is the integrand that appears in \eqref{RGW.chern},  regarded as a cohomology class on the real relative moduli space. The RGW invariant  \eqref{RGW.chern} is then the pairing 
\bear\label{rgw.pair}
RGW^{c, \mo}_{d, \chi}(\Si, L)_{\vec \mu}=
	&
	 \langle I(\Si; L),\;  \frac1{|\Aut(\vec \mu)|} [\ov \M_{d, \chi, \vec \mu}^\R(\Si, V)]^{\vir, \mo}\rangle 
\eear
with the corresponding VFC of the moduli space.

\subsection{Splitting the integrand} Assume next that $L\ra \F$ is a holomorphic line bundle over the holomorphic family 
$\F=\cup_s \Si_s$ where $\F$ is as in \eqref{family.smoothings}. Denote by $L_s$ and respectively $\wt L$ the pullback of $L$ to 
$\Si_s$ and respectively the normalization $\wt \Si$ of $\Si_0$. As $s$ varies in $\De$, consider the class 
\best
I(\Si_s, L_s)=  c_{b/2}( -\ind \del_{L_s}) \in \cHH^b(\ov\M^\R_{d, \chi, \vec \mu} (\Si_s, V_s); \Q)
\eest
associated to $\Si_s$ as in \eqref{integrand}. Denote by 
\bear\label{I.on.fam.abs}
I(\F, L)= c_{b/2}( -\ind \del_L)  \in \cHH^b(\ov\M_{d, \chi, \vec \mu} (\F_{/\De}); \Q) 
\eear
the corresponding class on the family \eqref{mod.fam.la} of real relative moduli spaces; as in \eqref{integrand}, this is the pullback of the class 
$c_{b/2}(-R^\bullet \pi_*\ev^* L)$ on the family of (complex) relative moduli spaces. We also consider the class 
$I(\wt \Si, \wt L)$ associated by \eqref{integrand} to $\wt \Si$. These classes are related as follows. 
\begin{lemma}\label{L.int.pull.back} The pullback of $I(\F, L)$ under the inclusion \eqref{map.include} is equal to the class $I(\Si_s, L_s)$. 
Moreover, the pullback of 
$I(\F, L)$ under the attaching map \eqref{map.attach} is the corresponding class $I(\wt \Si, \wt L)$ associated to $\wt \Si$, i.e.
	\bear\label{8.pull.int}
	\Phi^* I(\F, L)= I(\wt \Si, \wt L).
	\eear
\end{lemma} 
\begin{proof} Forgetting the real structure commutes with the maps at level of moduli spaces induced by 
$\Si_s\hookrightarrow \F$ and $\wt \Si \ra \Si_0 \hookrightarrow \F$. The lemma then follows from the corresponding results for the usual (complex) family of relative moduli spaces proved by Bryan-Pandharipande, as given in the proofs of \cite[Theorem~3.2]{bp1} and \cite[Proposition~A.1]{bp-TQFT}. 

Specifically, the first statement is an immediate consequence of the fact that the universal curve of the relative moduli space associated to $\Si_s$ is the pullback of the universal curve on the family of moduli spaces. 

The second statement follows from the normalization exact sequence   
$$
0\lra L_{|\Si_0}\lra \wt L\lra L_{|x^+}\oplus L_{|x^-}\lra 0,
$$ 
for the holomorphic line bundle $L$, where $x^\pm$ is the pair of (conjugate) nodes of $\Si_0$. Pulling this back over the (complex) moduli space gives 
$$
\Phi^*c(-\Ind\del_{L_{|\Si_0}})=c(-\Ind\del_{ \wt L})
$$
for the total Chern classes as in \cite{bp1}, since the pullback of the last term is a trivial rank $2\ell(\la)$ bundle.
\end{proof}

\begin{proof}[\bf \em Proof of the  RGW Splitting Theorem~\ref{T.gluing}]
Consider the invariant \eqref{RGW.chern} associated to $\Si_s$, written in the form  \eqref{rgw.pair}. {\em A priori}, it is a pairing between a homology and a cohomology class defined on the moduli space associated to $\Si_s$. However, by Lemma~\ref{L.int.pull.back} the integrand is pulled back from the class \eqref{I.on.fam.abs} on the family.  Combining this with the splitting of the VFC formula \eqref{split.VFC} gives 
	\begin{align*}
	RGW^{c_s, \mo_s}_{d, \chi}(\Si_s, L_s)_{\vec \mu}=
	&\;
	\frac1{|\Aut(\vec \mu)|}  \langle I(\F, L),\; [\ov \M_{d, \chi, \vec \mu}^{\R}(\Si_s, V_s)]^{\vir, \mo_s}\rangle
	\\
	=&\; \frac1{|\Aut(\vec \mu)|} \sum_{\la\vdash d}  \frac{\zeta(\la)}  {|\Aut (\la)|^2} 
	\langle I(\F, L),\; \Phi_*[\ov \M_{d, \chi+4\ell(\la), \vec \mu, \la, \la}^{\R}(\wt \Si, \wt V)]^{\vir, \wt \mo}\rangle
	\\
	=&\; \frac1{|\Aut(\vec \mu)|} \sum_{\la\vdash d}  \frac{\zeta(\la)}  {|\Aut (\la)|^2} 
	\langle\Phi^* I(\F, L),\; [\ov \M_{d, \chi+4\ell(\la), \vec \mu, \la, \la}^{\R}(\wt \Si, \wt V)]^{\vir, \wt\mo}\rangle
	\\
	=&	\sum_{\la\vdash d} \zeta(\la) RGW^{\wt c, \wt \mo}_{d, \chi+4\ell(\la)}(\wt \Si, \wt L)_{\vec \mu, \la, \la},
	\end{align*}
where the last equality uses \eqref{8.pull.int}. 
\end{proof}
  
Using the fact that the local RGW invariants are constant under smooth deformations of the target and passing to the (shifted) generating functions \eqref{RGW.real.rel} we also obtain the following consequence. 
\begin{cor}\label{C.rgw}  Let $(\Si_0, c_0)$ be a symmetric marked curve with a pair of conjugate nodes (and no real marked points),  $L_0$ a complex line bundle on $\Si_0$, and $\mo_0$ a choice of twisted orientation data on $\Si_0$. Then the local RGW invariants of any smooth deformation $(\Si, c, L, \mo)$ of $(\Si_0, c_0, L_0, \mo_0)$ are related to those of the 
normalization $\wt \Si$ of $\Si_0$ by 
\begin{align*}
	RGW^{c, \mo}_{d}(\Si, L)_{\vec \mu}= \sum_{\la\vdash d} \zeta(\la)t^{2\ell(\la)}  RGW^{\wt c, \wt \mo}_{d}(\wt \Si, \wt L)_{\vec \mu, \la, \la}. 
\end{align*}
\end{cor}
As discussed in \cite[\S4]{GI}, every smooth symmetric curve $\Si$ (without any real marked points) can be deformed into such a nodal symmetric curve $\Si_0$ by pinching a pair of {\em conjugate splitting circles}. Furthermore, every complex line bundle $L$ over $\Si$ and choice $\mo$ of twisted orientation data for $\Si$ can also be deformed to the nodal curve $\Si_0$ and then lifted to the normalization $\wt \Si$ of $\Si_0$. Conversely, every complex line bundle $\wt L$ and twisted orientation data $\wt \mo$ on 
$\wt \Si$ descend to $\Si_0$ and can be deformed to a complex line bundle $L$ and orientation data $\mo$ on $\Si$. Therefore Corollary~\ref{C.rgw} implies \cite[Theorem~4.1]{GI}. 
\smallskip
\setcounter{equation}{0}
\setcounter{section}{0}
\setcounter{theorem}{0}

\appendix

\section{Linearizations} 
\medskip

In this appendix we include more details about  the various linearizations we considered and the precise relation between them. 
For transversality and gluing arguments, we use Fredholm completions in the type of (weighted) Sobolev norms defined in \cite[p. 80]{ip-rel},  \cite[p. 971]{ip-sum} and  \cite[(1.1) of Corrigendum]{ip-sum}, reviewed in \S\ref{S.A.weighted}  and used in \cite{ip-rel} and \cite{ip-sum} to analyze the relative moduli spaces (relative a smooth divisor) appearing in the symplectic sum construction. For the absolute moduli spaces, similar norms are defined in \cite[\S B.4]{pardon}, and used by Pardon for a proof of the usual gluing theorem. For orientability and the construction of determinant line bundles of real CR-operators we use \cite[Appendix~A.2]{gz}, where the authors take Fredholm completions in a different variant of these norms, those originally used by Li-Tian in the proof of the usual gluing theorem, cf. \cite[\S3]{li-tian}. The monograph \cite{ms} uses both usual Sobolev norms reviewed in \cite[\S B.1]{ms} for transversality \cite[\S 3.2]{ms} and another version of weighted Sobolev norms \cite[\S 10.3]{ms} for the proof of the usual gluing theorem \cite[\S 10.1]{ms}.

\subsection{General considerations}\label{S.gen.consid} 
Let $(C, c)$ be a (smooth) symmetric marked curve, and $\bf x$ denote a symmetric subset of  the marked points of $C$.  Let $(E,c_E)\ra (C, c)$ be a Real bundle, i.e. a complex vector bundle with a real structure $c_E$ covering $c$ (called a real bundle pair in \cite{gz}). Denote by $\Gamma(E)^\R$ the space of smooth sections of $E$ which are real, i.e. invariant under the involution $\xi\mapsto c_E\circ \xi \circ c$.  Let
\bear\label{Real.CR}
D:\Gamma(E)^\R \ra \La^{01}(E)^\R=\Gamma(\Omega^{01}_C\otimes_\cx E)^\R, \quad D\xi=\del\xi+ A\xi
\eear
be a real CR-operator on the Real bundle $(E, c_E)$. Here $\del$ is a holomorphic $\del$-operator on $E$ compatible with the real structure (i.e. such that $c_E$ is anti-holomorphic), and 
\best
A\in \Gamma(C; \Hom_\R (E, \Omega^{01}_C \otimes_\cx E))^\R
\eest
is a 0'th order term; see \cite[(4.4)]{gz}. After completing in (weighted) Sobolev norms, $D$ becomes a Fredholm operator. Examples include
\begin{enumerate}[(i)]
\item Sobolev completions $\D:W^{k, p}\Gamma(E)^\R \ra W^{k-1,p}\La^{01}(E)^\R$ for $p>1$ and $kp>2$ as in \cite[\S B.1]{ms}.  
\item completions as in \cite[\S3]{li-tian}. 
\item completions $\D:W^{k,p, \de}\Gamma(E)^\R \ra W^{k-1,p,\de}\La^{01}(E)^\R$ as in \cite[\S B.4]{pardon}, weighted at the points in 
$\bf x$, for $\de\in (0,1)$ and  $k, p$ so that $W^{k, p, \de} \hookrightarrow C^0$ eg. $k\ge 1$, $p\ge 4$ (Pardon uses $p=2$ and $k\ge 6$).
\end{enumerate}
By elliptic regularity, the kernels and cokernels of all these Fredholm completions are the same, and consist of smooth elements, see eg. \cite[\S C.1]{ms}  and  \cite[Lemma~B.5.2]{pardon}. Therefore we let  
\bear\label{ker.D0}
\ker D:= \ker \fD,   \quad  \cok D:=\cok \fD
\eear
where $\D: \E\ra \fF$ is {\em any} of the completions (i)-(iii) of the real CR operator $D: \Gamma(E)^\R  \ra \La^{01}(E)^\R$. 
Note that subspace $\ker D\subset \Gamma(E)^\R$ is {\em independent} of these choices. Moreover, 
$\cok D=\ker D^*$, where $D^*$ is the formal adjoint of \eqref{Real.CR}, regarded as a real CR operator on the bundle $\ov F$, 
where $F=\Omega^{01}_C\otimes E$. When $(E, c_E)\ra (C,c)$ is a Real line bundle, the index of \eqref{Real.CR} is equal to
\bear\label{A.ind.D}
\ind D= \ind\del_{(E, c_E)}=c_1(E)+\tfrac 12\chi(C).
\eear
  
Consider the short exact sequence of vector spaces 
\bear\label{split.gamma.E}
0\ra \Gamma_{\bf x}(E)^\R \ra  \Gamma(E)^\R \xrightarrow{\ev_{\bf x}}  
(E_{\bf x}, c_E)^\R \ra 0.
\eear
Here $\Gamma_{\bf x}(E)^\R $ denotes the subspace of sections which vanish at all the points in $\bf x$, and $(E_{\bf x}, c_E)$ denotes the restriction of $(E,c_E)$ to $\bf x$.  A choice of splitting of  \eqref{split.gamma.E} induces an isomorphism 
\bear\label{gamma.split.val}
\Gamma(E)^\R\cong \Gamma_{\bf x}(E)^\R  \oplus (E_{\bf x})^\R, 
\eear
where $ (E_{\bf x})^\R:=(E_{\bf x}, c_E)^\R$ denotes the real locus of $(E_{\bf x}, c_E)$. 
\begin{rem}\label{R.split} A splitting of \eqref{split.gamma.E} can be obtained as follows. Fix a trivialization of $E$ in a neighborhood of ${\bf x}$ and bump functions $\beta_x$ at each point $x\in \bf x$. Starting with  $\al =(\al_x)_{x\in \bf x}\in (E_{\bf x})^\R$, extend the values $\al_{x}\in E_{x}$ to a constant section of $E$ in a neighborhood of 
$x\in {\bf x}$ and multiply the result by the bump function to get a section $\sum_{x\in \bf x} \beta_x \al_x$ of $E$ with value $\al_{x}$ at $x\in \bf x$. Finally symmetrize it to obtain a real section. 
\end{rem}

Consider next the restriction 
\bear\label{D.prime}
D': \Gamma_{\bf x}(E)^\R  \ra \La^{01}(E)^\R 
\eear
of the real CR operator \eqref{Real.CR} to the subspace in \eqref{split.gamma.E}. Regard $D$ as defined on the right hand side of \eqref{gamma.split.val}, where it becomes the operator
\bear\label{split.off.Ex}
 \Gamma_{\bf x}(E)^\R  \oplus (E_{\bf x})^\R \ra \La^{01}(E)^\R, \quad  (\zeta, \al) \mapsto D'\zeta +\gamma( \al). 
\eear
Here $\gamma: (E_{\bf x})^\R \ra \La^{01}(E)^\R$ depends on the choice of splitting in \eqref{gamma.split.val}. Note that when   the 0'th order term $A$ is supported away from $\bf x$, then the splitting $\gamma$ can be chosen so that its image consists of forms supported away from $\bf x$. After completing in any of the (weighted) Sobolev norms (i)-(iii),  \eqref{split.gamma.E} becomes an exact sequence of Banach spaces and $D$ and therefore $D'$ become Fredholm operators. This induces a long exact sequence
\bear\label{SES.D'}
0\ra \ker D'\ra \ker D \ra (E_{\bf x})^\R \ra \cok D' \ra \cok D \ra 0. 
\eear
The map $(E_{\bf x})^\R\ra \cok D' $ is the projection of $\gamma$ to the cokernel. 
\smallskip

\begin{rem} In fact,  any short exact sequence 
\bear\label{D.fred.comp.D}
\begin{tikzcd} 
0\ar[r]&\E' \ar[d,  "\fD'"] \ar[r]& \E \ar[d,  "\fD"]  \ar[r]& \E''\ar[r]\ar[d, "\fD'' "]&0
\\
0\ar[r]&\F'\ar[r] &\F\ar[r]&\F''\ar[r]& 0
\end{tikzcd} 
\eear
of Fredholm operators induces (by the Snake Lemma) a long exact sequence
\bear\label{LES.gen}
0\ra \ker \fD'\ra \ker \fD \ra \ker \fD'' \ra \cok \fD' \ra \cok \fD \ra \cok \fD'' \ra 0. 
\eear
Examples include the restriction $\fD'$ of a Fredholm operator $\fD$ to a closed, finite codimension subspace, e.g. as in \eqref{SES.D'}.
\end{rem}

\subsection{Conjugating by a section}\label{S.A.conj.s} Consider  next the holomorphic line bundle $\O({\bf x})$ over $(C, c)$ with its natural real structure induced by the action of $c$ on $\bf x$.  This is defined as follows. Fix local holomorphic coordinates on $C$ at the points $x$ in $\bf x$, compatible with the real structure as in \eqref{loc.coord.C.marked}. Let 
\bear\label{sect.s.0}
 s(z)=z \quad\mbox{ for all }z\in U_x \mbox{ and }x\in \bf x, 
 \eear
 and $s(z)= 1$ outside a smaller neighborhood $U'$ of $\bf x$. Then $s$ is a holomorphic section of  $\O({\bf x})$ whose divisor is ${\bf x}$.  The real structure on $C$ induces a real structure 
 on $(\ma\sqcup_{x\in \bf x} U_x)\ti \cx$ and $U'\ti \cx$, and thus on the bundle $\O({\bf x})$, and the section $s$ is equivariant with respect to the real structures.
\smallskip

When $E\ra C$ is a holomorphic bundle (with a compatible real structure $c_E$), multiplication by $s$ gives rise to a short exact sequence of holomorphic sheaves 
\bear\label{o.x.seq} 
0\ra E\otimes\O(-{\bf x})\xrightarrow{\cdot s} E \ra  E_{\bf x} \ra 0
\eear
relating the two $\del$-operators on $E\otimes\O(-{\bf x})$ and $E$; this exact sequence is compatible with the real structures. Therefore it induces a long exact sequence in cohomology 
\bear\label{D.holo.les}
0\ra H^0(E\otimes\O(-{\bf x}))^\R\ra H^0(E)^\R\ \ra (E_{\bf x}, c_E)^\R \ra 
H^1(E\otimes\O(-{\bf x}))^\R \ra H^1(E)^\R \ra 0. 
\eear

Consider next a real CR operator \eqref{Real.CR}, and assume the 0'th order term $A$ vanishes in a neighborhood of the special points of $C$.  Conjugate $D=\del+A$ by $s$ to obtain the operator 
\bear\label{def.Ds}
D^s=s^{-1} D s= \del+ s^{-1} A s.
\eear 
This can be regarded as a real CR-operator on the Real bundle $E\otimes \O(-{\bf x})$ (well defined at least when the 0'th order term $A$ vanishes on a neighborhood of ${\bf x}$). The following standard result for $\del$-operators extends to real CR operators.
\begin{lemma}\label{L.ker.D1=Ds} Multiplication by $s$ induces canonical identifications 
\bear\label{kerD1=Ds}
\ker D^s \cong \ker D' \quad \mbox{ and } \quad \cok D^s\cong \cok D', 
\eear
where $D^s$ and $D'$ are given by \eqref{def.Ds} and \eqref{D.prime}. 
\end{lemma}
\begin{proof} Multiplication by $s$ maps sections of $E\otimes\O(-{\bf x})$ to sections of $E$ which vanish at $\bf x$ so we get a diagram
\bear\label{D.fred.Ds}
\begin{tikzcd} 
0\ar[r]&\Gamma(E\otimes \O(-{\bf x}))^\R\ar[d,  "D^s"] \ar[r, hook, "\cdot s"]&\Gamma_{\bf x}(E)^\R \ar[d,  "D'"] 
\\
0\ar[r]&\La^{01}(E\otimes \O(-{\bf x}))\ar[r, hook, "\cdot s"]& \La^{01}(E)^\R.
\end{tikzcd} 
\eear
After completing in the usual Sobolev norms as above \eqref{ker.D0} and passing to the quotient, we get a short exact sequence $0\ra D^s \ra D'\ra D''\ra 0$ of Fredholm operators, cf. \eqref{D.fred.comp.D}-\eqref{LES.gen}. Note that $\ind D^s= \ind D'$. When if $E$ is a Real line bundle, then $c_1(E\otimes\O(-{\bf x}))= c_1(E)- \sum_{x\in \bf x}1$ thus  $\ind D^s=\ind D - \sum_{x\in \bf x} 1= \ind D'$, cf. \eqref{A.ind.D} and  \eqref{SES.D'}. 

As in the proof of \cite[Lemma~B.5.2]{pardon}, this 
implies \eqref{kerD1=Ds}  as follows. It suffices to check that the induced map (i)  $\ker D^s\ra \ker D'$ is surjective (it is injective by \eqref{LES.gen}), and (ii) $\cok D^s\ra \cok D'$ is injective (surjectivity then follows from (i) and the fact that $\ind D^s=\ind D'$).  
 
This can be verified directly using \eqref{SES.D'} and elliptic regularity. Assume $\xi \in  \ker D'$.  Thus $\xi \in \ker D$ and $\xi$ vanishes at all the points in $\bf x$. Then by elliptic regularity it has an expansion $\xi(z)= \al z(1+O(|z|))$ around a point in $\bf x$, thus $s^{-1} \xi \in \ker D^s$. %

It remains to show that $\cok D^s\ra \cok D'$ is injective. Let $\eta \in \cok D^s$. Since $D^s$ is a real CR operator, then  $\eta\in \ker (D^s)^*$ is  smooth by elliptic regularity. It suffices to show that if $s\eta=D\xi$ where $\xi\in W^{1, p}$ and $\xi(x)=0$ for all $x$ in $\bf x$ then $s^{-1} \xi\in W^{1, p}$. By elliptic regularity $\xi$ is smooth and has an expansion $\xi(z)=\al_1 \ov z + a z(1+O(|z|))$ around each point in $\bf x$. But then $D\xi= \al_1+ O(|z|)$ while $s\eta =bz(1+O(|z|))$ thus $\al_1 =0$. This implies $s^{-1}\xi \in W^{1, p}$ completing the proof.
\end{proof}
One can similarly obtain the exact sequence 
\bear\label{D.s.les}
0\ra \ker D^s\ra \ker D \ra (E_{\bf x}, c_E)^\R  \ra \cok D^s\ra \cok D \ra 0
\eear
generalizing  \eqref{D.holo.les}, which corresponds to $D= \del_{(E, c_E)}$ (i.e. $A=0$).

More generally, fix $\la = (\la_x)_{x\in \bf x}$ a symmetric sequence of positive integer multiplicities. Let $s=s_\la$ be a symmetric section of the Real line bundle $\O(\la \cdot {\bf x})=\O({\textstyle\sum}_{x\in \bf x}\la_x x)$ such that 
\bear\label{conj.sect}
s(z)= z^{\la_x} \quad \mbox{ for all } \quad z\in U_x \mbox{ and $x$ in $\bf x$}, 
\eear 
and $s$ is smooth and nonzero outside $U=\sqcup_{x\in \bf x} U_x$. Then the conjugate operator \eqref{def.Ds} is a real CR operator $D^s$ on 
$E\otimes_\cx \O(-\la \cdot {\bf x})$, and multiplication by $s$ induces a canonical isomorphism 
\bear\label{fD.la}
\ker D^s \cong \ker D^\la,  \quad \cok D^s \cong   \cok D^\la 
\eear
as in Lemma~\ref{L.ker.D1=Ds}. . 
Here 
\bear\label{D.prime.la}
D^\la: \Gamma_{ \la; \bf x}(E)^\R  \ra \La^{01}_{\la-\bf 1; \bf x}(E)^\R 
\eear 
is the restriction of $D=\del+A$ to the subspace $\Gamma_{\la; \bf x}(E)^\R$ of sections which vanish to order $\la$ at $\bf x$ and whose target is the subspace $\La^{01}_{\la-{\bf 1}; \bf x}(E)^\R$ of (0,1)-forms which vanish to order $\la-\bf 1$ at $\bf x$; it becomes a Fredholm operator $\cal D^\la$  after completing in the usual Sobolev norms for $k$ sufficiently large (so that $W^{k, p}\hookrightarrow C^{\la_x}$ for all $x\in \bf x$).  As in \eqref{ker.D0}, these completions $\D^\la$ have the same kernel and cokernel. The $\la=\bf 1$ case  of \eqref{fD.la} corresponds to \eqref{kerD1=Ds}; for $\la>\bf 1$,  \eqref{fD.la} follows by induction from \eqref{D.fred.Ds} with $D$ replaced by $D^{s_{\la -\bf 1}}$.
\begin{rem}\label{R.other.ker} The kernel and cokernel of the restriction $D^\la$ of $D$ 
consists of those elements  of $\ker D$ and $\cok D=\ker D^*$ which vanish to order $\la$ and respectively $\la-\bf 1$ at $\bf x$. There are other Fredholm completions of \eqref{D.prime.la} that have the same kernel and cokernel, cf. \S\ref{S.A.weighted}. 
\end{rem}

\subsection{The normalization short exact sequence} The situation is similar in the case $C_0$ is a nodal symmetric curve. Let ${\bf y}=\{y_1, \dots, y_\ell\}$ denote the collection of nodes of $C_0$.  Let  $\wt C$ denote the normalization of $C_0$, and let $y_{i1}$, $y_{i2}$ be the two points of $\wt C$ which are the preimage of the node $y_i$. Then $\wt C$ has a symmetric subset $\wt {\bf y}$  of marked points, where  
\best
 {\wt {\bf y}= {\bf y}_1\sqcup {\bf y}_2,\quad \text{and} \quad {\bf y}_k=\{ y_{1k}, \dots, y_{\ell k}\}, \quad \text{ for } k=1, 2}.
\eest

If $(E, c_E)\ra (C_0, c)$ is a Real bundle, let $(\wt E, c_{\wt E})$ denote the pullback bundle. Then we have a normalization short exact sequence
\bear\label{SES.E}
0\lra (E, c_E) \lra (\wt E, c_{\wt E})  \lra (E_{{\bf y}}, c_E)\lra 0.
\eear
It induces a short exact sequence
\bear\label{split.gamma}
0\ra \Gamma(C_0; E)^\R \ra  \Gamma(\wt C, \wt E)^\R \ra  (E_{\bf y}, c_E)^\R \ra 0, 
\eear
recording the fact that a continuous section of $E$ is a section of $\wt E$ that has the same value on the pair of points of $\wt C$ that correspond to a node of $C_0$. Choosing a splitting of \eqref{split.gamma} provides an isomorphism
\bear\label{split.gamma.chose}
 \Gamma(\wt C, \wt E)^\R  \cong \Gamma(C_0; E)^\R \oplus  (E_{{\bf y}})^\R
\eear
as before. Then similarly one can relate a real CR operator  
\bear\label{A.D.0.gen}
D_0: \Gamma(C_0; E)^\R \ra \La^{01}(C_0; E)^\R, \quad D_0\xi = \del \xi + A_0(\xi) 
\eear
on $E$ to the pullback operator $\wt D$ on $\wt E$. In particular, we get a long exact sequence 
\bear\label{D.wt.les}
0\ra \ker D_0\ra \ker \wt D \ra (E_{\bf y})^\R \ra \cok D_0\ra \cok \wt D \ra 0. 
\eear
\begin{rem}\label{R.split.norm} A splitting of \eqref{split.gamma.chose} can be obtained by first lifting $v\in E_{\bf y}$ to the normalization as 
$(v, -v)\in \wt E_{{\bf y}_1}\oplus  \wt E_{{\bf y}_2}$ and then proceeding as in Remark~\ref{R.split} to obtain the section  of $\Gamma(E)^\R$ supported in a neighborhood of the marked points ${\bf y}_1\sqcup {\bf y}_2$ of $\wt C$ and taking opposite values on any pair of marked points that correspond to a node of $C_0$. 
\end{rem}
\subsection{Variations in the domain}\label{S.var.domain} 
Recall from  \cite[p. 721]{gz} that the Kodaira-Spencer deformation theory canonically identifies the tangent space of the real Deligne-Mumford moduli space at a stable symmetric curve $C$ with
\bear\label{T.DM}
T_C \ov \M^\R\cong H^1(\T_{C})^\R 
\eear
(after passing to covers in the presence of automorphisms); see  also \cite[\S3]{melissa}. Here $\T_C$ is the relative tangent bundle to $C$, cf. \S \ref{S.rel.tan.bd}. In fact, as in \cite[\S4.1]{ip-gv}, the variation in $C$  as a marked (symmetric) curve can be regarded as a (symmetric) $(0, 1)$-form $h$ with values in $\T_C$ and supported away from the special points of $C$ (here $h$ is the variation $\de j$ in the complex structure $j$). This gives a diagram
\bear\label{T.slice.app}
\begin{tikzcd}
T_C \ov \M^\R\ar[r, hook] \ar[rd, swap, "\cong"]& \La^{01}(\T_C)^\R \ar[d, "\pi_{\cok}"]
\\
&  H^1(\T_{C})^\R
\end{tikzcd}
\eear
where the vertical arrow is the projection onto the cokernel $H^1(\T_{C})^\R$ of the $\del$ operator on $(\T_C, c_\T)$. 

\begin{rem}\label{R.t=var.coord} Assume $C$ is a stable symmetric curve with marked points ${\bf x}=\{ x_1, \dots, x_m\}$ and let $(\T_C)|_{x_i}$ denote the restriction of $\T_C$ to $x_i$.  As $C$ varies in the real Deligne-Mumford moduli space,  $(\T_C)|_{x_i}$ defines a complex orbibundle denoted $\T_{|x_i}$ over the moduli space, equal to the pullback of the relative tangent bundle $\T$ on the universal curve via the section 
$C \mapsto (C, x_i)$; see \S\ref{S.rel.tan.bd}. An element $(C, v_i)\in \T_{|x_i}\setminus 0$ can be regarded as a curve $C$ together with a germ of a local holomorphic coordinate $z_i$ at $x_i$. Moreover, the exact sequence \eqref{o.x.seq} for $E= \T_C$ induces the exact sequence 
\bear\label{var.in.holo.coord.smooth}
\begin{tikzcd}
0 \ar[r]& (\T_{|\bf x})^\R \ar[r]& H^1(\T_{C}\otimes\O(-{\bf x}))^\R \ar[r]& H^1(\T_{C})^\R  \ar[r]& 0, 
\end{tikzcd} 
\eear
cf. \eqref{D.holo.les}.  The last term parametrizes the variation in the marked symmetric curve $C$, cf. \eqref{T.DM}. The middle term in \eqref{var.in.holo.coord.smooth} encodes the variation in $(C, v)$, where $v= (v_x)_{x\in \bf x}$ consists of germs of holomorphic coordinates at the marked points compatible with the real structure as in \eqref{loc.coord.C.marked}. The first term encodes just the variation in $v$ (with $C$ fixed).  
\end{rem} 

Assume next $C_0$ is a nodal curve, let $\bf y$ denote its collection of nodes and $\wt C$ its normalization. Then the normalization short exact sequence for the holomorphic bundle $E=\T_{C_0}$ induces the long exact sequence 
\bear
\label{tan-norDM.ses}
0\ra H^0(\T_{C_0})^\R \ra H^0(\T_{\wt C})^\R \ra  (\T_{|\bf y})^\R \ra H^1(\T_{C_0})^\R \ra H^1(\T_{\wt C})^\R\ra 0.
\eear

When $C_0$ is stable so is $\wt C$. Therefore both $H^0$ terms vanish, and the $H^1$ terms model the tangent spaces to the real Deligne-Mumford spaces containing $C_0$ and respectively
$\wt C$, so \eqref{tan-norDM.ses} reduces to
\bear\label{SES.T.h1}
\begin{tikzcd}
0 \ar[r]&  (\T_{|\bf y})^\R \ar[r]& H^1(\T_{C_0})^\R \ar[r]& H^1(\T_{\wt C})^\R\ar[r]& 0.
\end{tikzcd} 
\eear
The last term parametrizes variations in $C_0$ that do not smooth its nodes, while the middle one parametrizes all variations in $C_0$, including those that smooth the nodes. A choice of splitting of the short exact sequence \eqref{SES.T.h1} and the isomorphism \eqref{T.DM} for $C=C_0$ and $\wt C$  determines a decomposition
\bear
\label{tan-norDM}
T_{C_0} \ov \M^\R=  (\T_{|\bf y})^\R \oplus  T_{C_0} \ov \N^\R. 
\eear
Here $\ov \N^\R$ denotes the stratum containing $C_0$, and $(\T_{|\bf y})^\R$ can be regarded as its normal bundle. Moreover, those variations in $C_0$ that map to 0 in $H^1(\T_{\wt C})^\R$ via \eqref{SES.T.h1} can be regarded as variations in the gluing parameters for a family \eqref{tau.tauC.A} of versal deformations of $C_0$ smoothing the nodes. In local coordinates, an element of 
$(\T_{|\bf y})^\R$ can therefore be regarded as a variation $\de \tau$ in the gluing parameter $\tau$ (or in cylindrical coordinates as the variation $\frac{\de \tau}\tau=\de (\log \tau)$), for example via similar considerations as those in Remark~\ref{R.t=var.coord}. 

\begin{rem} \label{R.split.T} In \eqref{SES.T.h1}, $H^1(\T_{C_0})^\R$ and $H^1(\T_{\wt C})^\R$ are the cokernels of the $\del$-operator on the relative tangent bundle to $C_0$ and respectively $\wt C$; the former operator is the restriction of the latter to a subspace. As in \eqref{D.wt.les}, the 
second arrow in \eqref{SES.T.h1} is the projection onto the cokernel of $\del_{(\T_{C_0}, c)}$ of 
\bear\label{h.map}
h: (\T_{|\bf y})^\R \ra \La^{01}(\T_{\wt C})^\R=\La^{01}(\T_{C_0})^\R,
\eear
where $h(v)=\del_{(\T_{\wt C}, c)} r(v)$ for some choice of splitting 
$r:(\T_{|\bf y})^\R  \ra  \Gamma(\T_{\wt C})^\R$ of \eqref{split.gamma.chose} for $E=\T_{C_0}$ as in   Remark \ref{R.split.norm}. 
\end{rem}
\subsection{Weighted Sobolev norms and real CR operators}\label{S.A.weighted} 
Let $C$ be a smooth curve. Fix a subset  $\bf y$ of marked points of $C$, and $\la=(\la_y)_{y\in \bf y}$ a collection of \textsf{multiplicities} associated to $\bf y$, where $\la_y\in \Z\setminus 0$ for all $y$ in $\bf y$. Fix $z=e^{-(t+i\theta)}$ a local holomorphic coordinate at each point $y$ in $\bf y$. Let 
\best
C'= C\setminus {\bf y}
\eest
regarded as a manifold with infinite cylindrical ends parametrized by $(t, \theta)\in[0,\infty)\ti S^1$, one for each point $y\in \bf y$. Fix a weight function $\rho$ and a cylindrical metric $g'$ on $C'$ such that 
\bear\label{rho.cyl}
\rho=|z|= e^{-t},    \quad g' = dt^2+d\theta^2 
\eear
in the neighborhood $U_y$ of each point $y\in \bf y$ and $\rho>0$ elsewhere; denote by $U'_y\subset U_y$ a slightly smaller neighborhood of $y$. When $E$ is a bundle over $C$, let 
\bear\label{wSob.maps.la.B}
\| \zeta\|_{k, p, \de; \la}^p= \int_{C\setminus U_y'}\l(\sum_{|m|\le k} |\nabla^m \zeta |^p\r) d vol_{g'} + \sum_{y\in \bf y} \int_{C' \cap U_y}\l(\sum_{|m|\le k} |\rho^{1-\la_y-\de} \nabla^m \zeta |^p\r) d vol_{g'}
\eear
denote the $\la$-\textsf{weighted Sobolev norm} on $\Gamma(E)$, associated to $k\ge 1$, $p\ge 2$, weight $\de\in\R\setminus \Z$ and the fixed multiplicities $\la=(\la_y)_{y\in \bf y}$.  The $\la$-norms on $\La^{01}(E)= \Gamma(\Omega^{01}_{C} \otimes_\cx E)$  are defined by the same formula. These norms are defined with respect to a fixed metric and compatible connection on $E$, and the cylindrical metric $g'$ and its Levi-Civita connection on $TC'= TC|_{C'}$. 

When $(C, c)$ is a symmetric curve and $(E, c_E)$ is a Real bundle over it, the norms on $\Gamma(E)^\R$ and  
$ \La^{01}(E)^\R$ are defined by the same formulas by using throughout choices compatible with the real structure, i.e. (i) symmetric weight functions $\rho$ and multiplicities $\la=(\la_y)_{y\in \bf y}$ associated to a symmetric collection of points $\bf y$ and 
(ii) symmetric hermitian metrics and connections on $E$ and $TC'$, with symmetric cylindrical ends. 
\medskip

Let $D: \Gamma(E)^\R \ra \La^{01}(E)^\R$, $D= \del+ A$ be a real CR operator on $(E, c_E)$. Completing in the 
$\la$-norms \eqref{wSob.maps.la.B} gives rise to an operator
\bear\label{D1.weighted.B}
\fD^\la:W^{k, p, \de}_{\la; \bf y}\Gamma(C'; E)^\R \ra W^{k-1, p, \de}_{\la; \bf y}\La^{01}(C';E)^\R,  \qquad \fD = \del+ A
\eear
on the punctured curve $C'=C\setminus \bf y$, which is Fredholm as in Lockhart and McOwen \cite[Theorem 6.2]{LM} since $\de\not\in \Z$ is not an eigenvalue of the asymptotic linearized operator on any end (the eigenvalues are $\Z$, corresponding to the powers $z^n$ on the punctured disk). When $(E, c_E)$ is a line bundle and $\de\in (0,1)$, the index of \eqref{D1.weighted.B} is equal to 
\bear\label{ind.D1}
\ind \fD^\la= c_1(E)[C]+ \tfrac 12 \chi(C)- \sum_{y\in \bf y} \la_y. 
\eear

\smallskip

Assume next that the multiplicities $\la$ are positive, and fix $\de\in (0,1)$.  Then \eqref{D1.weighted.B} can be regarded as an operator on the closed curve $C$ as follows. Restrict to $k\ge 1, \;p\ge 4$ (or else $k\ge 2$ and $p\ge 2$) so that $W^{k,p}_{loc}\hookrightarrow C^0_{loc}$ by the Sobolev embedding theorem.  Then sections of $E|_{C'}$ with finite $\la$-norm are continuous on $C'$ and exponentially decay on the ends, thus extend continuously to sections of $E$ on $C$ which vanish to order $\la$ at $\bf y$. Let 
\bear\label{D.la.B}
D^\la: \Gamma_{\la; \bf y}(E)^\R \ra  \La^{01}_{\la-\bf 1; \bf y}(E)^\R
\eear
denote the restriction of $D=\del+A$ defined as in \eqref{D.prime.la}. 
Note that on $C$, smooth sections of $E$ have finite $\la$-norm if and only if they vanish to order $\la$ at $\bf y$; however, since $|dz|=|z|$ in cylindrical coordinates, (0,1)-forms with finite 
$\la$-norm only vanish to order $\la-\bf 1$ at $\bf y$. Fixing $\de\in (0,1)$ and completing \eqref{D.la.B} in the same $\la$-norms as \eqref{D1.weighted.B} gives rise to a Fredholm operator 
\bear\label{D.la2.B}
D^\la:\E^{k,p} \ra \fF^{k-1,p}
\eear
on the closed curve $C$, canonically identified with \eqref{D1.weighted.B}.  By elliptic regularity its kernel and cokernel are independent of $\de\in (0,1)$ and $k, p$ in the range above, and consist of those smooth elements in the kernel and respectively cokernel of $D$ with finite $\la$-norm. In particular, the kernel and cokernel of the Fredholm completion of \eqref{D.la.B} in these $\la$-norms is the same as that of the Fredholm completion in the usual Sobolev norms (for $k$ sufficiently large) considered in \S\ref{S.A.conj.s}. 

\subsection{Linearizations for smooth targets}\label{S.lin.smooth} The considerations above allow us to regard the linearization to the relative moduli space $ \ov\M^\R(\Si, V)$ 
in several equivalent ways. Assume $f:C \ra \Si$ is an element of the relative moduli space $\ov\M^\R(\Si, V)$, and that $f$ has no rubber components. Then by definition it is an element of the absolute moduli space $\ov\M^\R(\Si, \emptyset)$ which has a certain branching over $V$, where $f^{-1}(V)$ consists only of marked points of the domain. This means that in local coordinates on the target around $x\in V$ and on the domain in the neighborhood of each marked point $y\in f^{-1}(x)$, the map $f$ has an expansion  
\bear\label{lead.coef.f}
f(z)= a_y z^{\la(y)}\wh f(z)\,,  \quad \mbox{ where } \quad \wh f(z)= 1+ O(|z|), 
\eear
$a_y\ne 0$ is the leading coefficient of $f$ at $y$, and $\la(y)>0$ is the contact multiplicity. Intrinsically, the leading 
coefficient $a_y(f)$ of $f$ at $y$ is an element 
\bear\label{lead.coef.f.1}
a_y(f)\in (T_yC)^{-\la(y)}\otimes_\cx T_{f(y)} \Si. 
\eear
As $f:C\ra \Si$ varies in $\ov\M(\Si, V)^\R$, this determines a section $a_y$ of the complex line bundle 
\bear\label{lead.coef.f.2}
(\T_y) ^{-\la(y)}\otimes_\cx \ev_y^*T\Si 
\eear
cf.  \cite[Lemma~7.1]{ip-sum}; note that $f(y)=x\in V$, thus the second factor in \eqref{lead.coef.f.2} is trivial.
\smallskip

Conditions (i)-(iv) of Definition~\ref{D.moduli.rel.or} imply that a local variation in $f:C \ra \Si$ with fixed domain and target (including fixed marked points) also fixes the collection of contact points $\bf y$ and their contact multiplicity $\la$. Variations $\de f$ in $f$ along $\ov \M^\R(\Si,V)$ will vanish to order $\la(y)$ at $y$, and in fact will continue to have an expansion  of the form \eqref{lead.coef.f}. The leading coefficient of the expansion of $\de f$ at $y$ is the variation $\de a_y$ in the leading coefficient of $f$ at $y$ (assuming the coordinate systems are fixed). In particular, variations $\de f$ which vanish to next order at all the contact points fix the leading coefficients of $f$. Of course, if a smooth function $f$ vanishes of order $\la(y)$ at the point $y$, then $\del f$ vanishes to order $\la(y)-1$.
\medskip

Fix a map $f$ with contact multiplicity $\la$ and contact points $\bf y$. Denote by $\Gamma_{\la; \bf y} (f^*T\Si)^\R$ the subspace of sections that vanish to order $\la(y)$ at $y$ for $y\in \bf y$, as in \eqref{D.prime.la}. Then the linearization $L_f$ to $\ov\M^\R(\Si, V)$ is the restriction 
\bear\label{L.f.A.non}
L_f:\Gamma_{\la; \bf y} (f^*T\Si)^\R\oplus T_C \ov \M^\R \ra \La^{01}_{\la-\bf 1; \bf y}(f^*T\Si)^\R
\eear
of the "usual" linearization
\begin{equation}\label{A.lin.abs}
\begin{split}
&{\mathbb L}_f: \Gamma(f^*T\Si)^\R\oplus T_C \ov \M^\R \ra \La^{01}(f^*T\Si)^\R
\\
&{\mathbb L}_f(\xi, h)= \del \xi+ [\nabla_\xi \nu + \tfrac 12 J df h]^{01}. 
\end{split}
\end{equation}
This follows as in the proof of \cite[Lemma~4.2]{ip-rel}, which also applies in our case by working with symmetric choices throughout. Specifically, 
 fix $\de\in (0,1)$ and restrict to $k\ge 1, \;p\ge 4$ (or else $k\ge 2$ and $p\ge 2$). After passing to completions as in \eqref{D.la2.B}, the linearization \eqref{L.f.A.non} to $\ov\M^\R(\Si, V)$ at a map $f$ with contact multiplicity $\la$ becomes the Fredholm operator
\begin{equation}\label{lin.rel.norms.weighted}
\begin{split}
&\L_f: \E_f^{k,p}\oplus T_C \ov \M^\R \ra \F_f^{k-1,p} 
\\
&\L_f (\xi, h)=  \del \xi+ [\nabla_\xi \nu + \tfrac 12 J df h]^{01} \; = \;  \del_{f^*(T\Si, dc)} \xi+ \hat A_f(\xi)+ \hat b_f(h )
\end{split}
\end{equation}
obtained by completing \eqref{L.f.A.non} using the $\la$-weighted Sobolev norms of \S\ref{S.A.weighted}, cf. \cite[p.~971]{ip-sum}, as corrected in \cite[eq. (1.1) of Corrigendum]{ip-sum}. Here 
\best
\hat A_f \in \Hom_\R( f^*T\Si, \La^{01}_C \otimes_\cx f^*T\Si)^\R 
\eest
is a (symmetric) 0'th order term given by $\hat A_f(\xi)=[\nabla_\xi\nu]^{01}$, while $\hat b_f: T_C \ov \M^\R\ra  \La^{01}(f^*T\Si)^\R$ is induced by the map $h\mapsto \tfrac 12 J df h$. Here $h$ is a $(0,1)$-form with values in the relative tangent bundle to $C$ which vanishes in a sufficiently small neighborhood of the special points of $C$ 
(a variation in $C$ can be regarded as such a form $h$ via \eqref{T.slice.app}, as reviewed in \S\ref{S.var.domain}). 

For a map $f$ with contact multiplicity $\la$, consider (a) the completion \eqref{lin.rel.norms.weighted} of $L_f$ for $k$, $p$ in the range above and fixed $\de\in (0,1)$ or  (b) the completion $\cal L_f$ of $L_f$ in the usual Sobolev norms as above \eqref{ker.D0}, but with $k\gg 0$. Then as in \eqref{D.la.B}-\eqref{D.la2.B}, $\ker \L_f$ and $\cok \L_f$ are independent of these choices, and are denoted
\bear\label{A.L.complete}
\ker L_f =\ker \L_f \quad \text{and} \quad \cok L_f=\cok \L_f.
\eear
\begin{rem}\label{R.dim.counts} The linearization \eqref{L.f.A.non} to the real relative moduli space $\ov\M^\R(\Si, V)$ is precisely the restriction to the invariant part of the domain and target of the linearization to the (usual) relative moduli space $\ov\M(\Si, V)$ defined in \cite{ip-rel}. In particular, the index of the former is half the index of the latter. Thus the virtual dimension counts of \cite[Lemmas~7.5-7.6]{ip-rel} directly translate into counts for the moduli spaces $\ov\M^\R(\Si, V)$ considered in this paper. In particular,  the virtual dimension 
of the moduli space $\ov\M^\R_{d, \chi, \vec \mu} (\Si, V)$ is  
\begin{align}\label{A.v.dim=b} 
b&=c_1(T\Si)A +\tfrac 12 \chi-V \cdot A +\dim \ov\M^\R_{\chi, \ell(\vec \mu)}=d \chi(\Si\setminus V)- \chi + 2\ell (\vec \mu)
\\ \nonumber
&= d \chi(\Si) - \chi +2\de(\vec \mu).
\end{align}
Here $A=f_*[C]=d[\Si]\in H_2(\Si)$ and $V$ has $2r$ points, so $V \cdot A= 2dr$, while $\dim \ov\M^\R_{\chi, \ell(\vec \mu)}= -\tfrac 32 \chi+2\ell(\vec \mu)$.
\end{rem} 

An equivalent description of the linearization at $f$ is obtained as follows. Consider an invariant section $\si$ of $\O({\bf x})$ on the target $\Si$ whose 
divisor is  $V=\bf x$ and such that $\si$ is holomorphic in a neighborhood of $\bf x$, and smooth and nonzero elsewhere. Pull it back to get a section $s=f^*\si$ of $f^*\O({\bf x})$ which vanishes to order $\la$ at the points 
${\bf y}$ and nowhere else. Conjugating \eqref{A.lin.abs} by $s$ as in \S\ref{S.A.conj.s} and \S\ref{S.A.weighted} allows us to regard the linearization to $\ov\M^\R(\Si, V)$ as 
\bear\label{lin.f.rel.A}
D_f: \Gamma(f^*\T_\Si)^\R \oplus T_C \ov \M^\R \ra \La^{01}(f^*\T_\Si)^\R
\eear
given by the formula
\bear\label{lin.f.rel.form.A}
D_f(\xi,h)\;=\; \del_{f^*(\T, c_\T)} \xi +  s^{-1}[\nabla_{s\xi} \nu + \tfrac 12 J df h]^{01} \; = \;  \del_{f^*(\T, c_\T)} \xi + A_f(\xi)+ b_f(h). 
\eear
By construction both $\nabla \nu$ and $h$ vanish in a sufficiently small neighborhood $U$ of the special points of the domain, so we can chose the section $\si$ (depending on $f$) so that $s=f^*\si$ is identically 1 outside this neighborhood. In that case \eqref{lin.f.rel.A} is given by exactly the same formula as $L_f$, except that $T\Si$ is replaced by $\T_\Si$, cf. \eqref{T.punctured}. 

Note that the linearization $D_f$ is a 0'th order perturbation of $\del_{(E, c_E)}\oplus 0$, where 
$\del_{(E, c_E)}$ is a real CR operator \eqref{Real.CR} on the Real bundle $(E, c_E)=f^*(\T_\Si, c_\T)$.  Let $\D_f$ denote the Fredholm operator obtained by completing $D_f$ in the usual Sobolev norms or any other choice from those listed above \eqref{ker.D0}. Then as in  \eqref{ker.D0}, $\ker \D_f$ and $\cok \D_f$ are independent of these choices, and are denoted $\ker D_f$ and $\cok D_f$.

 \begin{lemma} \label{L.all.lin.sameA}  Dividing by $s$ induces an identification 
\bear\label{ker.L=D.A}
\ker L_f \cong \ker D_f \quad \mbox{ and }\quad \cok L_f \cong  \cok D_f,
\eear 
canonical up to homotopy in $s$. Thus either one of these linearizations locally describes the moduli space $\ov\M^\R(X,V)$ around $f$. This induces canonical 
 (up to homotopy) identifications between the determinant lines 
\bear\label{det.L=D}
\det L_f \ma \cong_{ \eqref{ker.L=D.A}} \det D_f \ma \cong_{ \eqref{lin.f.rel.form.A}} \det (\del_{f^*(\T_\Si, c_\T)}\oplus 0)=  \det \del_{f^*(\T_\Si, c_\T)}\otimes \mathfrak f^* \det  T \ov \M^\R.
\eear
\end{lemma}
\begin{proof} Consider the operator $\mathbb L_f$ given by \eqref{A.lin.abs}. Up to 0'th order terms,  $\mathbb L_f$  is the same as $\del_{(E, c_E)}\oplus 0$, where $\del_{(E, c_E)}$ is a real CR operator \eqref{Real.CR} on the bundle $(E, c_E)=f^*(T\Si, dc)$. The linearization $L_f$ is the restriction of $\mathbb L_f$ to subspaces of sections which vanish to order $\la$ cf. \eqref{L.f.A.non} while $D_f$ is the conjugate operator 
$s^{-1}{\mathbb L}_f s$ cf. \eqref{lin.f.rel.form.A}. Therefore \eqref{ker.L=D.A} follows the same way as  \eqref{fD.la}. 

Since up to 0'th order terms $D_f$ is the same as $\del_{f^*(\T_\Si, c_\T)}\oplus 0$ cf. \eqref{lin.f.rel.form.A},  the second identification in \eqref{det.L=D} follows as in \cite[(3.3) and \S4.3]{gz} with $E=TX$ replaced by $E=\T_\Si$.
\end{proof}

\begin{rem}\label{R.A.lead.term.f} Under the transformation $\xi \mapsto  \frac1 {f^*\si} \xi $, a variation $\de f$ in $f$ (regarded as a section of $f^*T\Si$) is mapped to a section $\xi= \frac{\de f } {f^*\si}$ of $f^*\T_\Si$; the leading coefficient of $\de f$ at $y$ is mapped to the value
\bear\label{val.y}
\xi(y)= \frac{\de a_y } {a_y}=\de (\log a_y)
\eear
where $\de a_y$ denotes the variation in the leading coefficient $a_y$ of $f$ at $y$, cf. \eqref{lead.coef.f}. 
\end{rem}

We can also separate the leading coefficient of the variation in $f$ as in  
\cite{ip-sum}. For the linearization \eqref{lin.f.rel.A}, this corresponds to separating the value 
$\xi(y)\in (f^*\cal T)_y$ at each contact point $y$,  which is given by \eqref{val.y}. Specifically, using a decomposition such as in \eqref{split.off.Ex} for $E= f^*\T_\Si$ over $C$ and ${\bf x^+}$ replaced by the contact points ${\bf y^+}$,  the linearization \eqref{lin.f.rel.A} rewrites as the operator 
\begin{equation}\label{D.lin.push.a}
\begin{split}
&{\bf D}_f: \Gamma_{\bf y}(f^*\T_\Si)^\R \oplus (f^*\T_{\Si})_{\bf y^+} \oplus T_C \ov \M^\R \ra \La^{0,1}(f^*\T_\Si)^\R
\\ 
&{\bf D}_f(\zeta, \al, h)= D_f(\zeta, h)+ \gamma(\al). 
\end{split}
\end{equation}
Here $\zeta \in \Gamma(f^* \T_{\Si})^\R$ gives rise to a variation in $f$ with fixed domain and target and also fixed leading coefficient \eqref{lead.coef.f} at each one of the marked points, while the  middle term records the variation in the leading coefficients, cf. \eqref{val.y}. The term  $\gamma: (f^*\T_{\Si})_{\bf y^+} \ra \La^{01}(f^*\T_\Si)^\R$ depends on the choice of splitting, see  Remark~\ref{R.split}. 
\subsection{Linearizations for nodal targets}\label{S.lin.nodal} The discussion above extends to the case where the target is a nodal curve $\Si_0$ with complex conjugate nodes; here we are discussing the case when the target has a single pair of complex conjugate nodes $x^\pm$. Consider a continuous real map $f_0:C_0 \ra \Si_0$ such that the preimage of $x^+$ consists only of nodes of the domain. Denote by 
${\bf y^+}=\{y_1^+, \dots y_\ell^+\}$ the preimage of the node $x^+$ of the target, and let $\bf y={\bf y^+} \sqcup {\bf y^-}$ where $y_i^-=c(y^+_i)$. Then $\bf y$ is a subset of nodes of $C$, and let $\wt C$ be the (partial) normalization of $C$ resolving the nodes in $\bf y$; let  $\wt \Si$ be the normalization of $\Si_0$.  Then $f_0:C_0\ra \Si_0$ lifts to a continuous real map 
\bear\label{f.lift.A}
\wt f:\wt C \ra \wt \Si
\eear 
with smooth target (but possibly nodal domain). When $f_0:C_0\ra \Si_0$ is an element of the real relative moduli space 
$\ov \fM(\Si_0)$ of \eqref{mo.over}, then (local) variations 
in $f_0$ can be described in terms of those of its lift \eqref{f.lift.A}. 

We can thus consider the linearization at $f_0:C_0 \ra \Si_0$ to $\ov \fM(\Si_0)$ (with fixed target and fixed $\nu$) induced by the linearization 
 \bear\label{D.wt.f.A}
 D_{\wt f}: \Gamma(\wt f^* \T_{\wt \Si})^\R  \oplus T_{\wt C} \M^\R \ra 
\La^{0,1}(\wt f^* \T_{\wt \Si})^\R 
\eear
at a lift $\wt f: \wt C \ra \wt \Si$ of $f_0$ to the normalizations. This has the form 
\bear\label{D.wt.form.A}
 (\xi, h)\mapsto  \del \xi+ A_{\wt f}(\xi)+ b_{\wt f}(h)
 \eear 
 cf. \eqref{lin.f.rel.form.A} with $f$ replaced by $\wt f$.  
 
 Using \eqref{T.SES} and the normalization short exact sequence  \eqref{SES.E} for 
$E=f_0^* \T_{\Si_0}$ we can rewrite  \eqref{D.wt.f.A} as a map
\bear\label{lin.f.nodal.A}
\wt D_{f_0}: \Gamma(f_0^* \T_{\Si_0})^\R \oplus ( f_0^*\T_{\Si_0})_{|\bf y^+}  \oplus T_{\wt C} \M^\R \ra 
\La^{01}(\wt f^* \T_{\wt \Si})^\R
\eear
An element $\zeta \in \Gamma(f_0^* \T_{\Si_0})^\R$ corresponds to a variation in $f_0$ with fixed domain, target, and also fixed product of the leading order terms at each one of the nodes in the inverse image of $x^+$, while the middle term records the variation in the product of leading coefficients (or in cylindrical coordinates the variation $\frac {\de (a_1a_2)}{a_1 a_2}=  \frac {\de a_1}{a_1}+ \frac {\de a_2}{a_2}$). In particular, the linearization \eqref{lin.f.nodal.A} has the form 
\bear\label{D.nod.up.A}
\wt D_{f_0}(\zeta, \al, h)=  \del \zeta + A_{f_0}(\zeta)+\gamma(\al)+ b_{f_0}(h) 
\eear
where $A_{f_0}$ is the restriction of $A_{\wt f}$  to $\Gamma(f_0^* \T_{\Si_0})^\R$, and $b_{f_0}$ is equal to $b_{\wt f}$ cf. \eqref{D.wt.form.A}. Here $\gamma$ depends on the choice of splitting of  \eqref{SES.E} for $E=f_0^* \T_{\Si_0}$; such a splitting can be obtained for example by pulling back a splitting for $E=\T_{\Si_0}$. The last term $b_{f_0}=b_{\wt f}$ is induced by the map
\bear\label{free.b}
\La^{01}(\wt C; \T_{\wt C})^\R \ra \La^{01}(\wt C; \wt f^* \T_{\wt \Si})^\R
\eear
obtained from $h\mapsto \tfrac 12 J d f h$ after dividing by the section $\wt s=\wt f^*\si$ and projecting to the (0,1)-part.
\medskip

There is another perspective for the linearization at $f_0$ that is better suited when deforming the nodal target by smoothing its nodes. For that consider instead the operator 
\begin{equation}
\label{lin.f.nod1.A}
\begin{split}
&D_{f_0}: \Gamma(f_0^* \T_{\Si_0})^\R \oplus T_{C_0} \ov \M^\R \ra \La^{01}(\wt f^* \T_{\wt \Si})^\R
\\
& D_{f_0}(\zeta, h)= \del \zeta + A_{f_0}(\zeta)+b_{f_0}(h)
\end{split}
\end{equation}
where $A_{f_0}$ is as in \eqref{D.nod.up.A} and $b_{f_0}$ is induced by the same formula \eqref{free.b}.  However here, unlike in \eqref{D.nod.up.A}, the variation $h\in \La^{01}(\T_{C_0})^\R$ is tangent to the entire Real Deligne-Mumford moduli space, not just to the nodal stratum containing $C_0$. Using the exact sequence \eqref{SES.T.h1} and a decomposition of  \eqref{tan-norDM}, this operator can therefore be identified with
\begin{equation}
\label{lin.f.nod2.A}
\begin{split}
&\wh D_{f_0}: \Gamma(f_0^* \T_{\Si_0})^\R \oplus \T_{|\bf y^+}  \oplus T_{\wt C} \M^\R 
 \ra \La^{01}(\wt f^* \T_{\wt \Si})^\R
\\
&\wh D_{f_0}(\zeta,v, h)= \del \zeta + A_{f_0}(\zeta)+\wh \gamma(v)+ b_{f_0}(h). 
\end{split}
\end{equation}
Here $A_{f_0}$ and $b_{f_0}$ are exactly as in \eqref{D.nod.up.A}, but $\wh \gamma: \T_{|\bf y^+}  \ra  \La^{01}(\wt f^* \T_{\wt \Si})^\R$ depends on the choice of splitting 
of \eqref{split.gamma} for $E=\T_{C_0}$, while \eqref{D.nod.up.A} depends on one for $E=\T_{\Si_0}$. Specifically, $\wh \gamma$ is the composition of \eqref{free.b} and \eqref{h.map}. 

The two forms  \eqref{D.nod.up.A} and \eqref{lin.f.nod2.A} of the linearization at $f_0$ are in fact equivalent: using the formulas for $\gamma$ and $\wh \gamma$ in terms of the two splittings one sees that the linearization \eqref{lin.f.nod2.A} is the same as the linearization \eqref{lin.f.nodal.A}. Their middle terms are identified via the complex linear isomorphism induced by the linearization of \eqref{lead.coef.iso}, given by 
\bear\label{a,to.tau}
 \frac {\de a_y}{a_y}\mapsto -\la(y) \frac {\de \tau_y}{\tau_y}.
\eear
Here $a_y$ denotes the product of the leading coefficients of $f_0$ at $y$, while $\tau_y$ can be regarded as the gluing parameter of the domain at the node $y$. 
\begin{rem}\label{R.lead.coord} The leading coefficients depend on the germs of holomorphic coordinates used, as do the gluing parameters. The product $a_y$ induces the isomorphism \eqref{lead.coef.iso}, relating the gluing parameter in the target with a power of that of the domain; its linearization gives \eqref{a,to.tau}.  
\end{rem} 

\subsection{The normal operator}\label{S.normal.compl} The discussion above extends to the strata of the real relative moduli space $\ov \fM^\R(\Si, V)$ containing maps $f: C\ra \Si[m]$ with rubber components (i.e. $m\ne 0$); the target $\Si[m]$ is nodal but comes with an extra rescaling action, cf. \S\ref{S.rel.moduli}. To describe such strata, we adapt \cite[Definition~7.2 and Remark~7.7]{ip-rel} to our setting. It is based on the following simple observation: away from the infinity divisor, $\P_V=\P(N_V\oplus \cx)$ is equal to $N_V$, via the canonical embedding $w \mapsto [w,1]$ in homogenous coordinates. Thus for a map $\rho:C\ra V$, sections 
$\xi$ of $\rho^*N_V$ (with finitely many zeros and poles) can be equivalently regarded as lifts 
\bear\label{f=rho}
f=(\rho, \xi)
\eear
of $\rho$ to $\P_V$ (with finitely many points in the preimage of the zero and infinity divisor). 

Let $f:C \ra \Si[m]$ represent an equivalence class in $\ov\fM^\R(\Si, V)$, cf. \eqref{rel.moduli.jv.all}.  Let ${\bf y}\subset C$ denote the special points of the domain which are in the inverse image of the special points $\bf x$ of the target $\Si[m]$. Let 
\bear\label{lift.f.A}
\wt f: \wt C \ra \wt{\Si[m]}
\eear
denote the lift of $f$ defined as in  \eqref{f.lift.A}, where $\wt{\Si[m]}$ is the normalization of the building $\Si[m]$, thus a disjoint union of $\wt\Si$ and $m_i$ copies of $\P_{x^\pm_i}\subset \P_V$, cf. \eqref{def.PV}. We can then decompose the lift \eqref{lift.f.A} as a disjoint union of pieces 
\bear\label{f.decomp}
f_0:C_0 \ra \Si \subset \Si[m] \quad \mbox{ and } \quad f_r:C_r \ra \P_V, \; \mbox{ for } r\ge 1, 
\eear
where $C_0$ denotes the union of all level 0 (i.e. non-rubber) components, while the other pieces $C_r$, for $r\ge 1$, are symmetric curves consisting of (rubber) components mapped to one of the $m_i$ copies of 
$\P_{x_i^\pm}\subset \P_V$ of the building $\Si[m]$. 

\smallskip

 In particular, for $r\ge 1$, each rubber piece $f_r:C_r \ra \P_V$ of the decomposition \eqref{f.decomp} is a real map with prescribed multiplicity 
 $|\la_r|>0$ of contact to the $0$ and $\infty$-divisor of $\P^1_V$ at the special 
points ${\bf y}_r$ of $C_r$; equivalently $f_r$ is a nontrivial section $\xi_r$ of the pullback normal bundle $N_V=T\Si_{|V}$ to $V$ with zeros and poles of prescribed {\em order} $\la_r:{\bf y}_r \ra \Z$ at the special points and no other zeros/poles (recall that the order of a meromorphic section is {\em negative} at a pole and positive at a zero). Therefore 
\bear\label{sum.la.0}
\sum_{y\in {\bf y_r}} \la_r(y) =0 \quad \mbox{for all }r\ge 1, 
\eear
since the intersection numbers of $f_r$ with the zero and the infinity divisor of $\P_V$ are equal. 
Note that the $\cx^*$ action on $\P_V$ induces a natural $\cx^*$ action both on maps $f_r:C_r \ra \P_V$ and on sections $\xi_r$ of the pullback normal bundle. 
\smallskip

Moreover, along a stratum of $\ov\fM^\R(\Si, V)$ the topological type of the maps $f:C\ra \Si[m]$ and of their lifts  \eqref{lift.f.A}-\eqref{f.decomp} is fixed,  including the topological type of the domain, target and the contact multiplicities. Therefore the conditions describing a stratum of $\ov\fM^\R(\Si, V)$ represented by maps 
\best
f:C\ra \Si[m]\quad \text{ with $m\ne 0$}
\eest
can be expressed in terms of conditions on their lifts \eqref{lift.f.A}-\eqref{f.decomp}, or equivalently on their projections
\bear\label{f.decomp.2}
f_0:C_0 \ra \Si \subset \Si[m] \quad \mbox{ and } \quad \rho_r:C_r \ra V, \; \mbox{ for } r\ge 1
\eear
under the collapsing map to $\Si$, as in \cite[\S7]{ip-rel}.  
 If we denote by $\rho=p(f)$ the projection of $f$ under the collapsing map $\Si[m] \ra \Si$, then these conditions become (a) $\rho$ is a real $(J, \nu)$-holomorphic map and (b) its restriction $\rho_r:C_r \ra V$ to each rubber piece $C_r$ is a real map which satisfies the additional condition that there exits an element 
\bear\label{xi.not.0.A}
\xi_r \in \ker L^N_{\rho_r}
\eear
such that $\xi_r$ is non-trivial on each irreducible component of $C_r$. Here 
\begin{equation}\label{A.normal.op}
\begin{split}
&L^N_{\rho_r}:  W^{k, p, \de}_{\la_r; \bf y_r} \Gamma(C_r'; \rho_r^*N_V)^\R  \ra W^{k-1, p, \de}_{\la_r; \bf y_r} \La^{01}(C_r'; \rho_r^* N_V)^\R
\\
&L^N_{\rho_r}\xi= \del  \xi- \nabla_\xi \nu
\end{split}
\end{equation}
denotes the normal operator associated to any real map $\rho_r:C_r\ra V$, completed to a Fredholm operator on the punctured curve $C'_r=C_r\setminus \bf y_r$ in the $\la_r$-norms as in 
\eqref{wSob.maps.la.B}-\eqref{D1.weighted.B}, with $\de\in (0,1)$. By \eqref{ind.D1} and \eqref{sum.la.0} the index of this operator is 
\bear\label{A.ind.LN}
\ind L^N_{\rho_r}=c_1(\rho_r^*N_V)+\tfrac 12 \chi(C_r)-\sum_{y\in {\bf y_r}} \la_r(y)= \tfrac 12\chi(C_r).
\eear
Note that $V$ is 0-dimensional, so $E=\rho_r^*N_V$ is topologically trivial, thus $c_1(E)=0$.

Since $V=V^+\sqcup V^-$ consists only of pairs of complex conjugate points, then after further decomposing $\wt C$ into its connected components and reindexing the rubber ones, we may assume without loss of generality that we get a decomposition \eqref{f.decomp} where (a) $C_0$ is the union of all level 0 (i.e. non-rubber) components and 
(b) the other pieces $C_r$, for $r\ge 1$, are symmetric curves 
$$C_r= C_r^+\sqcup C_r^-$$ 
consisting of rubber components, where $C_r^\pm$ are {\em connected} and mapped to a point $p_r^\pm$ in $V^\pm$. Then the  
restriction of \eqref{A.normal.op} to $C^+_r$ is a {\em complex} linear Fredholm operator by the second condition in \eqref{nu.rel}, and therefore is a holomorphic $\del$-operator on $C_r^+$ with values in $(\rho_r^+)^*N_V$. Every nontrivial element $\xi$ in the kernel of $L^N_{\rho_r^+}$ is meromorphic and has finite $\la_r$-norm. Therefore  $\xi$ must have zeros/poles of {\em order} at least $\la_r(y)$ at the points $y\in {\bf y}_r^+:={\bf y}_r\cap C_r^+$ and no other poles (otherwise its $\la_r$-norm would not be finite). But the first Chern class of the bundle $(\rho_r^+)^*N_V$ is 0, and it is also equal to the sum of the orders of the zeros and poles of any nontrivial meromorphic section of it. Therefore $\xi$ must have order exactly $\la_r(y)$ at each $y\in {\bf y}_r^+$ and no other zeros nor poles. Note that the ratio of any two such sections is constant (on connected components).
\begin{rem}\label{R.xi.nor.A} An element $\xi_r$ satisfying \eqref{xi.not.0.A} can be equivalently regarded as a real map 
\bear\label{xi.in.ker.A}
\xi_r: C_r \ra \P_V \quad \mbox{  which is a solution of} \quad  \del  \xi = \nabla_\xi \nu
\eear 
and which has prescribed multiplicity of contact to 0 and $\infty$ at the special points ${\bf y}_r$ of $C_r$. The $\cx^*$ action rescaling $\P_V$ corresponds to the $\cx^*$-action on $\ker L^N$. Moreover as explained above, whenever $C_r^+$ is {\em connected} then either $\ker L^N=0$ or else $\ker L^N\cong \cx$. 
\end{rem}
\section{Families}
\subsection{Smoothings of symmetric curves}\label{S.A.smoothings} Let $(C_0, c_0)$ be a (possibly nodal) symmetric marked curve, and denote by $\bf x$ and $\bf y$ the collection of marked points and respectively nodes of $C_0$. In this section $C_0$ is allowed to have real nodes or real marked points. Flat families 
\bear\label{A.flat}
(\cal C, c)\ra (B, c_B)
\eear 
of deformations of $C_0$ are defined at the beginning of \cite[\S4.2]{gz};  their restriction to $b\in B^\R$ is a symmetric marked curve $C_b$ with the induced real structure $c_b$. 

Such families \eqref{A.flat} can be obtained by smoothing the nodes of $C_0$ with gluing parameters $\tau\in B$, using standard local models in a neighborhood $U$ of the nodes. This construction is described in \cite[\S6.2]{gz} in the case $C_0$ has a single node (which must be real), and in \cite[\S4.2]{GZ2} in the case $C_0$ has a pair of conjugate nodes. The general case when $C_0$ has more nodes is defined similarly as we briefly outline below.\smallskip

Choose holomorphic coordinates $z=z_x$ on $C_0$ around each marked point $x\in \bf x$ identifying a neighborhood of $x$ in $C_0$ with the unit disk $|z|\le 1$. We can furthermore arrange that in this neighborhood of $\bf x$ the real structure $c_0$ has a standard form in these coordinates, i.e. is given by 
\bear\label{loc.coord.C.marked}
z_x\mapsto \ov z_{c_0(x)}
\eear
as a map from the neighborhood of $x$ to the neighborhood of $c_0(x)$. Similarly, we can fix holomorphic coordinates $z_{1}, z_{2}$ on $C_0$ at each node $y\in {\bf y}$ identifying a neighborhood of $y$ in $C_0$ with the locus 
\bear\label{loc.coord.C.node}
z_1 z_2= 0, \quad |z_1|, |z_2| \le 1
\eear 
and in which the real structure on $C_0$ has a standard form (depending on the type of node). Note that the coordinates $z_1, z_2$ on $C_0$ at the node $y$ correspond to local holomorphic coordinates on the normalization $\wt C$ at the two points $y_1$, $y_2$ in the pre-image of $y\in \bf y$. 

Then in a neighborhood $U_y \subset \cal C$ of the node $y$, the fiber $C_\tau$ of $\cal C$ over $\tau=(\tau_y)_{y\in \bf y}\in B$ is given by 
\bear\label{loc.coord.C.sm.node}
z_1 z_2=\tau_y,  \quad \quad |z_1|, |z_2| \le 1, 
\eear
where $\tau_y$ is called the \textsf{gluing parameter} at the node $y$,  for  $y\in \bf y$. Here $B \subset \cx^{\bf y}$ is a sufficiently small neighborhood of 0 and its real structure $c_B$ is defined by $\tau_y\mapsto \ov \tau_y$ if $y$ is a real node and by $\tau_{y^\pm} \mapsto \ov \tau_{y^\mp}$ for a pair $y^\pm$ of conjugate nodes. The real structure on $\cal C$ also has a standard form in these coordinates, given by \cite[(6.11)]{gz} in a neighborhood of a real node, and by \cite[(4.9)]{GZ2} in the neighborhood of a pair of conjugate nodes. Finally, outside a smaller neighborhood $U'$ of the nodes the family \eqref{A.flat} is identified with the product family $(C_0 \setminus U', c_0)\ti (B, c_B)$.
\medskip

Intrinsically, the gluing parameters $\tau$ encode variations normal to the nodal stratum, cf. 
\eqref{SES.T.h1}-\eqref{tan-norDM}; for every node $y$ of $C_0$, the gluing parameter can be regarded as an element $\tau_y\in {\cal T}_{|y}$.  
Here $\T_{|y}=\T_{|y_1}\otimes \T_{|y_2}$ denotes the relative tangent bundle at the node $y$, and 
$\tau=(\tau_y)_{y\in \bf y}\in \T_{|\bf y}= \oplus_{y\in \bf y} \T_{|y}$.
\medskip

One can also include local deformations $C_{0u}$ of $C_0$ along the nodal stratum $\N$ containing $C_0$  e.g. obtained by varying $j$ away from the nodes and marked points. In that case, one gets flat families \eqref{A.flat} parametrized by pairs $(\tau, u)\in B=B_N \ti B_T$, where $u\in B_T$ parametrizes local deformations $C_{0u}$ of $C_0$ preserving its nodes, and $C_{\tau u}$ is obtained from $C_{0u}$ by smoothing its nodes with gluing parameter $\tau$. In particular, when $C_0$ is stable and 
$(\tau,u)$ vary in a sufficiently small neighborhood $\cal S$ of $0$ in ${\cal T}_{|\bf y} \oplus T_{C_0} \N=T_{C_0}\ov \M$, we get versal families
\bear\label{tau.tauC.A}
\begin{tikzcd}
(\cal C, c)\ar[r] &(\cal S, c_{\cal S})
\end{tikzcd}
\eear
called "local slices" at $C_0$. Their restriction to ${\cal S}^\R\subset T_{C_0}\ov \M^\R$ is a family $C_{\tau u}$ of symmetric marked curves, where $u\in T_{C_0} \N^\R$ is a variation along the nodal stratum  $\N^\R$ containing $C_0$ and 
$\tau\in ({\cal T}_{|\bf y})^\R$ is a normal variation to it,  cf. \S\ref{S.var.domain}. 

\subsection{The relative tangent bundle}\label{S.rel.tan.bd} Let $(\cal C, c)\ra (B, c_B)$ be a flat family of deformations of a fixed marked symmetric curve as in \S\ref{S.A.smoothings}.  Denote by
\bear\label{T.family.C}
(\T, c_\T) \ra (\C, c)
\eear
the relative tangent bundle to the fibers of this family. It is a holomorphic line bundle (locally free sheaf) over $\cal C$. Its dual $\T^\vee$ is the relative dualizing sheaf of the family: 
\begin{enumerate}[(i)]
\item in local coordinates \eqref{loc.coord.C.sm.node} near a node $y$, it is generated by the sections 
$\frac {d z_1} {z_1}$ and $\frac {dz_2} {z_2}$ modulo the relation $ 0=\frac {d (z_1z_2)}{z_1z_2}= \frac {d z_1} {z_1}+ \frac {dz_2}{z_2}$.
\item in local coordinates \eqref{loc.coord.C.marked} near a marked point $x$, it is generated by the section $\frac {d z} z$, where the marked point corresponds to $z=0$.
\end{enumerate} 
Both $\T$ and its dual $\T^\vee$ come with a natural real structure induced by $c$; see  \cite[Lemma 6.8]{gz} for the case $C_0$ has one real node and \cite[Lemma 4.8]{GZ2} for the case $C_0$ has a pair conjugate nodes (the general case follows similarly). The restriction of \eqref{T.family.C} to a fiber $C_b$ of $\cal C$ with $b\in B^\R$ is the relative tangent bundle \eqref{T.punctured} of $(C_b, c_b)$. Moreover, if $(\wt C, \wt c)$ denotes the 
normalization of $(C_0, c_0)$ then the pullback of $\T$ to $\wt C$ is $\T_{\wt C}$. 

By the Kodaira-Spencer deformation theory, infinitesimal automorphisms of the marked symmetric curve $C_0$ correspond to $H^0(\T_{C_0})^\R$ while infinitesimal variations correspond to 
$H^1(\T_{C_0})^\R$, cf. \cite[(5.22) and p. 736]{gz}. 
\subsection{Families of real CR operators}\label{A.S.families}  Let $(E, c_E)$ be a Real bundle over a flat family $(\cal C, c)$ of deformations of a fixed symmetric curve as in \S\ref{S.A.smoothings}. When $\nabla$ is a $c_E$-compatible (complex linear) connection on $E$, and $A$ is a smooth 0'th order term, they determine a (pull-back) family of CR operators 
\bear\label{fam.CR}
D_b: \Gamma(C_b; \iota_b^* E)^\R \ra \La^{01}(C_b; \iota_b^* E)^\R
\eear
as in \cite[(4.7)-(4.8)]{gz}, for $b \in B^\R$; here $\iota_b:C_b \ra \cal C$ is the inclusion. Note that (0,1)-forms on $C_b$ with values in $\iota_b^*E$ can be regarded as sections of $\iota_b^*F=\Omega^{0,1}_{C_b} \otimes_\cx \iota_b^* E$ where $F=\Omega^{0,1}_{\cal C/B} \otimes_\cx E$. Let  
\best
\det D =\cup_b \det D_b 
\eest 
denote the determinant line bundle of the family \eqref{fam.CR} of CR-operators, defined as in \cite[(4.9) and Appendix A]{gz}. As explained there, the fact that $\det D$ is a line bundle (i.e. is locally trivial over $B_\R$) follows from the proof of the gluing theorem (for linear CR operators). 
 \smallskip
 
Families of real CR-operators often arise by pulling back data associated to Real bundles over smooth targets as in \cite[\S4.3, p 711]{gz}. Specifically, assume $(X, J, c_X)$ is a 
a smooth almost complex manifold with an anti-complex involution, and $(E,c_E)\ra  (X, c_X)$ is Real bundle over it. Fix a $c_E$-compatible (complex linear) connection $\nabla$ on $E$. Given a family of real maps $f:C \ra X$, we can 
consider a family of CR operators 
\begin{equation}\label{lin.rel.A.E}
\begin{split}
&D_f: \Gamma (C;  f^*E)^\R \ra\La^{01}(C, f^*E)^\R, \quad 
\\
&D_f (\xi)=\del^\nabla_{f^*(E, c_E)}\xi +A_f \xi
\end{split}
\end{equation}
defined as in \cite[\S4.3, p 711]{gz}, where $A$ is a 0'th order term.  

 When $\Si$ is smooth, $\T_\Si\ra \Si$ is a Real bundle over a smooth target $X=\Si$, and we can similarly consider the family of linearizations 
\bear\label{lin.norms.22}
D_f: \Gamma (C; f^*\T_\Si)^\R \oplus T_C \ov \M^\R \ra \La^{01}(C; f^*\T_\Si)^\R
\eear
cf. \eqref{lin.f.rel.A} associated to a family of real maps $f:C\ra \Si$. Note that up to 0'th order terms, $D_f$ is equal to 
$\del_{f^*(\T_\Si,c_\T)}\oplus 0_{W_f}$, where $W_f= T_C \ov \M^\R$ and the considerations in \cite[Appendix A]{gz} continue to apply in this case.

\end{document}